\documentclass[11pt]{amsart}

\usepackage{amssymb,amsmath,mathrsfs}
\usepackage{graphicx}
\usepackage{tikz}
\tikzset{singZ/.style={draw=black,fill=red!75!black,rectangle,inner sep=0pt,minimum size=5pt},
  singH/.style={draw=black,fill=red!75!black,circle,inner sep=0pt,minimum size=5.5pt},
  hole/.style={draw=blue!70!black,fill=white,circle,thick,inner sep=0pt,minimum size=5.5pt}}

\usepackage[margin=2.6cm]{geometry}

\theoremstyle{thmstyleone}
\newtheorem{theorem}{Theorem}[section]
\newtheorem{proposition}[theorem]{Proposition}
\newtheorem{lemma}[theorem]{Lemma}
\newtheorem{corollary}[theorem]{Corollary}
\newtheorem{maintheorem}{Theorem}

\theoremstyle{definition}

\newtheorem{question}[theorem]{Question}

\theoremstyle{remark}
\newtheorem{remark}[theorem]{Remark}

\numberwithin{equation}{section}
\allowdisplaybreaks

\newcommand{\R}[1]{\mathbb{R}^{#1}}
\newcommand{\Z}{\mathbb{Z}}
\newcommand{\Ham}{\mathbb{H}}
\newcommand{\C}{\mathbb{C}}
\newcommand{\M}{\mathscr{M}}
\newcommand{\Mreg}{\mathscr{M}_{\mathrm{reg}}}
\newcommand{\Mo}{\mathscr{M}_{0}}
\newcommand{\Sing}{\operatorname{Sing}\mathscr{M}}
\newcommand{\Iso}{\operatorname{Iso}(\mathscr{M})}
\newcommand{\Oct}{W_{\mathcal{F}}(B_4)}
\newcommand{\F}{\mathcal{F}}

\newcommand{\scal}[2]{\langle #1, #2\rangle}
\newcommand{\Id}{\mathrm{Id}}
\DeclareMathOperator{\sinl}{sl}
\DeclareMathOperator{\cosl}{cl}
\DeclareMathOperator{\sn}{sn}
\DeclareMathOperator{\dn}{dn}
\DeclareMathOperator{\Aut}{Aut}

\DeclareMathOperator{\Cone}{Cone}

\title[A periodic minimal hypersurface with Clifford cone singularities]{A periodic minimal hypersurface in Euclidean 4-space\\ with Clifford cone singularities}

\author{Vladimir G. Tkachev}
\address{Department of Mathematics, Link\"oping University, Sweden}
\email{vladimir.tkatjev@liu.se}

\begin{document}


\begin{abstract}
Let $\sigma(t)=\sinl(\varpi t)$ be the lemniscatic sine, normalized to have period $2$. We study the zero set $\M$ of
$S(x)=\sigma(x_1)\sigma(x_3)-\sigma(x_2)\sigma(x_4)$
in $\R{4}$. We show that $\M$ is a complete embedded minimal hypersurface with infinitely many isolated singularities, which are exactly the vertices of the regular honeycomb of $\R{4}$ by $24$-cells. Near every singular point $\M$ is an analytic image of the Clifford cone, and $\M$ is a stationary integral varifold with density $\pi/2$ at each singularity. The isometry group of $\M$ is generated by the translations of the root lattice $D_4$ and a hyperoctahedral group of order $384$ fixing the centre of a $24$-cell; it has index three in the symmetry group of the honeycomb, acts transitively on the singular points, and contains a symmetry which is not visible from the defining equation. Seen from the centre of a $24$-cell, $\M$ is the zero set of a theta series of a translate of the Hurwitz quaternions, which makes the whole symmetry group manifest; the sections of $\M$ by a family of parallel hyperplanes are triply periodic surfaces with cubic symmetry which change topology through conical points. Finally, $\M$ is obtained from a union of $2$-planes by attaching cones over $2$-tori, and its quotient by the lattice of periods is a compact minimal hypersurface in a flat $4$-torus with exactly three singular points and infinite Morse index; we show that its regular part is one-sided and compute its Euler characteristic and integral homology.
\end{abstract}

\maketitle

\setcounter{tocdepth}{1}
\tableofcontents

\section{Introduction}\label{sec:intro}

Explicit complete minimal hypersurfaces in a Euclidean space $\R{n}$, $n\ge4$, are rare: apart from hyperplanes, catenoids and minimal cones \cite{Hsiang67,Simons,BGG,NTVbook,NTV12}, only a few families are known in closed form \cite{ChoeHoppe,LeeLee,HoppeTkachev}. Many other natural families, for instance higher-dimensional analogues of the Riemann, Scherk and Schwarz surfaces, are known to exist but are not given by explicit formulas \cite{KaabachiP,Pacard02,ChoeHoppe18}.

The situation is similar for minimal hypersurfaces with isolated singularities. The simplest examples are minimal cones, for instance the \emph{Clifford cone}
\begin{equation}\label{Clifford}
C=\{x\in\R{4}:\ x_1^2+x_2^2=x_3^2+x_4^2\},
\end{equation}
which is the cone over the Clifford torus $S^1(\frac{1}{\sqrt2})\times S^1(\frac{1}{\sqrt2})\subset S^3$ \cite{Lawson}. Caffarelli, Hardt and Simon \cite{CafHarSi} showed that every regular minimal cone is the tangent cone of many minimal hypersurfaces, defined in a ball, which are not cones and have an isolated singularity at the origin. Hardt and Simon \cite{HardtSimon} showed that an area minimizing cone with an isolated singularity is approximated on each side by smooth area minimizing hypersurfaces. Smale \cite{SmaleN} constructed stable embedded minimal hypersurfaces with boundary in $\R{n+1}$, $n\ge7$, with an arbitrary finite number of isolated singularities, by gluing stable cones along thin bridges; this was recently extended to higher codimension by Dimler \cite{Dimler}. Mazet \cite{Mazet} studied minimal hypersurfaces asymptotic to the Simons cones, and Simon \cite{SimonSing} constructed stable minimal hypersurfaces with prescribed (and quite arbitrary) singular sets. All these results are existence theorems, obtained by variational or perturbative methods.

In the introduction to \cite{SmaleN}, Smale wrote: ``An important question is what kinds of singular sets actually occur. In particular, in the case of isolated singularities, not much is known.'' His question concerns stable and area minimizing hypersurfaces, for which isolated singularities can only occur in dimensions $n\ge7$; the examples mentioned above have finitely many singular points and are modelled on stable cones. The present paper addresses the analogous question in the non-stable setting.

The purpose of this paper is to describe in detail a completely explicit example of a different nature: a periodic minimal hypersurface in $\R{4}$ with infinitely many isolated singularities, all modelled on the (unstable) Clifford cone, whose quotient is a compact minimal hypersurface without boundary in a flat $4$-torus with exactly three singular points. Let $\sinl t$ be the lemniscatic sine, that is, the solution of
$$
\sinl'(t)^2=1-\sinl^4(t),\qquad \sinl(0)=0,\ \sinl'(0)=1,
$$
and let $\varpi=\Gamma(\frac14)^2/(2\sqrt{2\pi})=2.62205\ldots$ be the lemniscate constant, so that $\sinl$ has period $2\varpi$ (see Section~\ref{sec:lemn}). We put
\begin{equation}\label{sigma}
\sigma(t):=\sinl(\varpi t),
\end{equation}
so that $\sigma$ has period $2$, vanishes exactly on $\Z$, and $\sigma(t+1)=-\sigma(t)$. Our main object is
\begin{equation}\label{Mdef}
\M=\{x\in\R{4}:\ S(x)=0\},\qquad S(x)=\sigma(x_1)\sigma(x_3)-\sigma(x_2)\sigma(x_4).
\end{equation}
Since $\sigma(t)\approx \varpi t$ for small $t$, near the origin $\M$ looks like the cone $x_1x_3=x_2x_4$, which is congruent to \eqref{Clifford}.

What makes this example interesting is not only that it is explicit, but that a single elementary equation carries several rich structures at once. Geometrically, $\M$ is a stationary integral varifold whose singular set is an infinite lattice-like set of Clifford-cone singularities (Theorem~\ref{thA}). Its symmetry group is governed by the root lattice $D_4$ and the regular honeycomb of $\R{4}$ by $24$-cells, and contains a symmetry which is not visible from the defining equation (Theorem~\ref{thB}). Seen from the centre of a $24$-cell, $\M$ becomes the zero set of a theta series of a translate of the Hurwitz order of integral quaternions, which makes the whole symmetry group manifest (Theorem~\ref{thD}). Finally, the topology of $\M$ and of its compact quotient in a flat torus can be described completely (Theorem~\ref{thC}). In this paper we develop these four aspects in turn.

Our first result says that $\M$ is a minimal hypersurface, and that the local picture near the origin repeats itself at every point of an infinite discrete set.

Throughout the paper, $e_1,\dots,e_4$ is the standard basis of $\R{4}$,
$$
D_4=\{m\in\Z^4:\ m_1+m_2+m_3+m_4\ \text{is even}\}
$$
is the checkerboard lattice, and
$$
h=\tfrac12(1,1,1,1),\qquad h'=\tfrac12(1,1,1,-1).
$$

\begin{maintheorem}\label{thA}
\begin{enumerate}
\item[(a)] $\M$ is a closed real analytic subset of $\R{4}$, invariant under the translations by vectors of $D_4$. The set of critical points of $S$ on $\M$ is
$$
\Sing=\Z^4\sqcup (h+D_4).
$$
\item[(b)] $\Mreg=\M\setminus \Sing$ is a properly embedded real analytic minimal hypersurface in the open set $\R{4}\setminus\Sing$.
\item[(c)] For every $y\in \Sing$ there exist a neighbourhood $U$ of $0$ and a real analytic diffeomorphism $\Phi$ of $U$ onto a neighbourhood of $y$ with $\Phi(0)=y$ and $d\Phi(0)=\Id$, such that
$$
\M\cap\Phi(U)=\Phi(C_y\cap U),
$$
where $C_y$ is the zero set of the quadratic part of $S$ at $y$; explicitly,
\begin{align*}
&&C_y&=\{x_1x_3=\epsilon x_2x_4\}, &&\text{with $\epsilon=(-1)^{y_1+y_2+y_3+y_4}$ if $y\in\Z^4$,} \\
&&C_y&=\{x_1^2+x_3^2=x_2^2+x_4^2\}, &&\text{if $y\in h+D_4.$}
\end{align*}
Each $C_y$ is congruent to the Clifford cone \eqref{Clifford}; in particular, $\M$ is not a topological manifold near $y$.
\item[(d)] $\M$, with multiplicity one, is a stationary integral varifold in $\R{4}$. At every $y\in\Sing$ its density equals $\pi/2$ and its unique tangent cone is $C_y$.
\end{enumerate}
\end{maintheorem}

Thus $\M$ is a complete embedded minimal hypersurface in $\R{4}$ with an infinite discrete singular set, each singularity being modelled on the Clifford cone. We are not aware of another explicit complete embedded minimal hypersurface in a Euclidean space with infinitely many isolated singularities. We emphasize that the Clifford cone in $\R{4}$ is not area minimizing, and is not even stable; accordingly, $\M$ is not stable either. In fact, every neighbourhood of a singular point carries infinitely many independent destabilizing variations, so that the compact quotient $\M/D_4$ described below has infinite Morse index (Theorem~\ref{thm:index}). The quotient $\M/D_4$ is a compact minimal hypersurface in the flat torus $\R{4}/D_4$ with exactly three isolated singular points (Theorem~\ref{thC}); this may be compared with the finitely many singularities in the examples of Smale \cite{SmaleN}.

Our second result describes the symmetries of $\M$. The singular set $\Sing$ consists of points of two different kinds: the points of $\Z^4$, where the tangent cone is $x_1x_3=\pm x_2x_4$, and the points of $h+D_4$, where it is $x_1^2+x_3^2=x_2^2+x_4^2$. It is natural to ask whether there is a symmetry of $\M$ that interchanges them. The answer is yes, but the corresponding isometry is not visible from the definition of $\M$; it comes from the addition theorem for the lemniscatic sine. To state the result, consider the orthogonal frame
\begin{equation}\label{frame}
\F=\{\pm f_1,\pm f_2,\pm f_3,\pm f_4\},
\end{equation}
where
\begin{equation}\label{frameff}
\qquad f_1=e_1+e_3,\ \ f_2=e_1-e_3,\ \ f_3=e_2+e_4,\ \ f_4=e_2-e_4,
\end{equation}
and let $\Oct$ denote its hyperoctahedral group, i.e. the group of orthogonal transformations of $\R{4}$ which permute the eight vectors of $\F$. In Coxeter--Weyl notation, $\Oct$ is the Weyl group of type $B_4$ (equivalently, $C_4$) for which $\pm f_1,\dots,\pm f_4$ are the short roots; it has order $2^4\cdot 4!=384$. Let $\Iso$ denote the group of all isometries $g$ of $\R{4}$ such that $g(\M)=\M$.

\begin{maintheorem}\label{thB}
\begin{enumerate}
\item[(a)] A translation $x\mapsto x+v$ preserves $\M$ if and only if $v\in D_4$.
\item[(b)] The isometry group of $\M$ is
$$
\Iso=\{x\mapsto h'+L(x-h')+n:\ \ L\in\Oct,\ n\in D_4\}\cong D_4\rtimes \Oct .
$$
In particular, $\Iso/D_4$ is isomorphic to $W(B_4)$ and has order $384$.
\item[(c)] $\Iso$ acts transitively on $\Sing$. The stabilizer of the origin has order $128$; it consists of the $L\in\Oct$ which preserve the partition $\F=\{\pm f_1,\pm f_4\}\sqcup\{\pm f_2,\pm f_3\}$, and it contains the orthogonal involution
\begin{equation}\label{Amatrix}
A=\frac12\begin{pmatrix}
1&1&1&1\\
1&1&-1&-1\\
1&-1&1&-1\\
1&-1&-1&1
\end{pmatrix},
\end{equation}
which maps $h$ to $e_1$. More precisely, $S(x)=Q(x)\,S(Ax)$ for all $x\in\R{4}$, where $Q$ is a real analytic function with $\frac14\le Q\le 4$.
\item[(d)] $\Sing$ is the vertex set of the regular honeycomb of $\R{4}$ by $24$-cells whose cells are the Voronoi cells of the points of $h'+D_4$, and $\Iso$ is a subgroup of index three in the symmetry group of this honeycomb.
\end{enumerate}
\end{maintheorem}

The centres $h'+D_4$ of the $24$-cells will be called the \emph{holes} of $\M$; at these points $|S|$ attains its maximum value $2$. The singular points of the two kinds and the holes are the four cosets of $D_4$ in the dual lattice $D_4^*=\Z^4\cup(h+\Z^4)$ (the Hurwitz integral quaternions). Part (b) shows that the symmetry group of $\M$ is generated by the translations by $D_4$ and by those symmetries of the honeycomb which fix the hole $h'$ and preserve the frame $\F$; by (d), it is a subgroup of index three in the symmetry group of the honeycomb. Note that $\Iso$ does \emph{not} preserve the two kinds of singular points: the involution $A$ in (c) interchanges them. This is the hidden symmetry mentioned above.

Our third result describes the topology of $\M$. The local structure near each singular point is described by Theorem A, but the global topology is not apparent from it. Computing the homology of the compact quotient $\M/D_4$ provides this global information: it shows, for instance, that every closed curve in the ambient $4$-torus is homologous to a curve on $\M/D_4$, and the one-sidedness of $\M/D_4$ is reflected in the vanishing of $H_3$ and in the $2$-torsion of $H_2$, as for non-orientable closed $3$-manifolds.

To state the result, let $\Mo$ be the union of the $2$-planes
$$
\{x\in\R{4}:\ x_i=a,\ x_j=b\},\qquad i\in\{1,3\},\ j\in\{2,4\},\ a,b\in\Z .
$$
Each of these planes lies in $\M$ (because $\sigma$ vanishes on $\Z$), and we call $\Mo$ the \emph{skeleton} of $\M$. For $n\in\Z^4$ we denote by $P_n=n+[0,1]^4$ the closed unit cube with the vertex $n$, and by $T_n$ the $2$-torus
$$
T_n=\Bigl\{x\in P_n:\ \begin{array}{l}(x_1,x_3)\in\partial\bigl([n_1,n_1+1]\times[n_3,n_3+1]\bigr),\\[2pt] (x_2,x_4)\in\partial\bigl([n_2,n_2+1]\times[n_4,n_4+1]\bigr)\end{array}\Bigr\};
$$
clearly $T_n\subset\Mo\cap\partial P_n$.

\begin{maintheorem}\label{thC}
\begin{enumerate}
\item[(a)] $\M$ does not meet the interior of $P_n$ if $n\notin D_4$. If $n\in D_4$, then $K_n:=\M\cap P_n$ is homeomorphic to the cone over the torus $T_n$, with vertex at the singular point $n+h$ (the centre of $P_n$). Thus
$$
\M=\Mo\cup\bigcup_{n\in D_4}K_n,
$$
and two different sets $K_n$ meet only along $\Mo$.
\item[(b)] $\M$ and $\Mreg$ are connected.
\item[(c)] The quotient $\M/D_4$ is a compact connected minimal hypersurface in the flat torus $\R{4}/D_4$ with exactly three singular points, the classes of $0$, $e_1$ and $h$. Its regular part is non-orientable, its Euler characteristic equals $3$, and its integral homology groups are
$$
H_0=\Z,\qquad H_1\cong\Z^4,\qquad H_2\cong\Z^6\oplus\Z_2,\qquad H_3=0;
$$
moreover, the inclusion $\M/D_4\subset\R{4}/D_4$ induces an isomorphism of the first homology groups.
\item[(d)] Let $D_4^+=\{n\in D_4: n_1+n_3\ \text{even}\}$. Then $\M/D_4^+$ is a double cover of $\M/D_4$ with orientable regular part; its Euler characteristic equals $6$ and $H_3(\M/D_4^+)\cong\Z$.
\end{enumerate}
\end{maintheorem}

Figure~\ref{fig:strata} illustrates part (a). Equivalently, the function $\lambda(x)=\sigma(x_1)\sigma(x_3)=\sigma(x_2)\sigma(x_4)$ on $\M$ vanishes exactly on the skeleton, its level sets $\{\lambda=r\}$, $0<|r|<1$, are disjoint unions of $2$-tori (one in each cube $P_n$, $n\in D_4$, with the right sign), and the tori shrink to the singular points $n+h$ as $|r|\to1$.

Our last result makes Theorem~\ref{thB}(b) transparent by looking at $\M$ from a hole. We put the origin at the hole $h'$, write $x=h'+y$, and use the \emph{frame coordinates}
$$
\xi_k=\xi_k(y)=\tfrac12\scal{y}{f_k},\qquad k=1,\dots,4,
$$
which are Cartesian coordinates adapted to the frame $\F$, up to the common factor $\frac1{\sqrt2}$. In these coordinates $\M$ is given by an equation which is symmetric in the four frame directions.
The second description of $\M$ uses theta series. For a translate $\Lambda=w+L$ of a lattice $L\subset\R{4}$ we consider the theta series
$$
\vartheta_\Lambda(y;\tau)=\sum_{v\in\Lambda}e^{\pi\mathrm{i}\tau|v|^2+2\pi\mathrm{i}\scal{v}{y}},\qquad y\in\R{4},\quad \operatorname{Im}\tau>0,
$$
which converges absolutely and is real analytic in $y$; for $y=0$ and $w=0$ it is the classical theta series of the lattice $L$ \cite[Chapters~2 and~4]{Conway}. We shall use it for the translate
$$
\Lambda=\tfrac12(e_1+e_3)+D_4^*=\tfrac12f_1+D_4^*
$$
of the dual lattice $D_4^*$, i.e. of the Hurwitz order of integral quaternions (in quaternionic notation, $\Lambda=\frac12(1+\mathrm{j})+\mathcal{O}$; see Remark~\ref{rem:Hurwitz}), and at the point $\tau=1+\mathrm{i}$ of the upper half-plane. Note that $\Lambda$ is not a lattice, since $0\notin\Lambda$; by its \emph{orthogonal group} we mean the group of orthogonal transformations $L$ of $\R{4}$ with $L(\Lambda)=\Lambda$.

\begin{maintheorem}\label{thD}
Let $\gamma(t)=\sigma(t+\frac12)$ be the normalized lemniscatic cosine, put $C_k=\gamma^2(\xi_k(y))$, $k=1,\dots,4$, and let $\mathcal{E}_j(C)$ denote the $j$-th elementary symmetric function of $C_1,\dots,C_4$. Then
$$
\M=h'+\Bigl\{y\in\R{4}:\ \sum_{k=1}^4\arctan C_k=\frac\pi2\Bigr\}
=h'+\bigl\{y\in\R{4}:\ \mathcal{E}_2(C)=1+\mathcal{E}_4(C)\bigr\}.
$$
Moreover, the function $-\mathrm{i}\,\vartheta_\Lambda(y;1+\mathrm{i})$ is real, and it is the product of $S(h'+y)$ and a positive real analytic function; in particular,
$$
\M=h'+\bigl\{y\in\R{4}:\ \vartheta_\Lambda(y;1+\mathrm{i})=0\bigr\}.
$$
Finally, the orthogonal group of $\Lambda$ is exactly $\Oct$.
\end{maintheorem}

In particular, the hidden symmetry $A$ in \eqref{Amatrix} becomes an actual invariance of the theta series, and the first description gives a second, one-line proof of minimality (Remark~\ref{rem:secondproof}): seen from a hole, $\M$ is a separable hypersurface $\sum\phi(\xi_k)=0$ with $\phi'^2=\varpi^2\cos4\phi$. In quaternionic terms (\S\ref{ssec:quaternionic}), the singular points split into three equivalent types labelled by the imaginary units $\mathrm{i},\mathrm{j},\mathrm{k}$, the tangent cone at a point of type $q$ is the Clifford cone of the complex structure $q$, and $A$ becomes the transposition of two imaginary units.

The symmetric equation also gives a simple way to see $\M$. Its sections by the hyperplanes $\xi_4=c$ orthogonal to the frame vector $f_4$ are triply periodic surfaces in $\R{3}$ with full cubic symmetry, shown in Figure~\ref{fig:sections}. For $c=\frac12$ the section is a packing of closed ``pillows'' which touch each other exactly at singular points of $\M$; for other values of $c$ the contact points open into necks. Although these pictures resemble well-known triply periodic minimal surfaces in $\R{3}$, such as Schwarz's P-surface, they are of a completely different nature: they are three-dimensional sections of a single minimal hypersurface of $\R{4}$, and they are not minimal surfaces themselves (\S\ref{ssec:sections}).

\medskip
\noindent\textit{Background.}
The hypersurface $\M$ was found by the author in 2008. It grew out of the earlier joint work of the author and V.~Sergienko on periodic zero mean curvature surfaces in the Lorentz--Minkowski space \cite{SergTk1999}, where elliptic functions and isolated conical singularities appear in a similar way. The example was presented in talks (for example, \cite{TkGranada}), and a large part of the present text was written in 2016.

On each open unit cube $n+(0,1)^4$ with $n\in D_4$, the hypersurface $\M$ is given by a separable equation $F_1(x_1)+F_2(x_2)+F_3(x_3)+F_4(x_4)=0$ (see \S\ref{ssec:separable}). In this sense it is a four-dimensional relative of Scherk's doubly periodic surface $e^{x_3}\cos x_1=\cos x_2$, the best known solution of Weingarten's problem on separable minimal surfaces in $\R{3}$ \cite{Weing}. In our approach, however, this separable form is a convenient accident of dimension four rather than a method, and it serves only as motivation. The relation to the literature on separable hypersurfaces is discussed in \S\ref{ssec:separable}. The present paper is not concerned with the separable ansatz; its subject is the geometry of $\M$: its singularities, its symmetries and lattice structure, its arithmetic (theta-function) description, and its topology.

\medskip
\noindent\textit{Organization of the paper.}
In Section~\ref{sec:lemn} we collect the facts on the lemniscatic functions that we need. Theorem~\ref{thA} is proved in Section~\ref{sec:thA}. In Section~\ref{sec:lattices} we recall some facts about the lattice $D_4$, its dual and the $24$-cell honeycomb. Theorem~\ref{thD} is proved in Section~\ref{sec:hole}, and it is used in Section~\ref{sec:sym} to prove Theorem~\ref{thB}. Theorem~\ref{thC} is proved in Section~\ref{sec:top}. In Section~\ref{sec:index} we show that the Morse index of $\M/D_4$ is infinite and explain the minimality of $\M$ by a factorization of the mean curvature operator. Section~\ref{sec:remarks} contains some remarks (on the lemniscatic catenoid and on the relation to separable hypersurfaces) and open questions.

\medskip
\noindent\textit{Notation.} For the reader's convenience we collect the main notation used throughout the paper.

\smallskip
\begin{tabular}{@{}p{0.28\textwidth}p{0.72\textwidth}@{}}
$\sinl t$, $\cosl t$, $\varpi$ & lemniscatic sine and cosine, lemniscate constant (Section~\ref{sec:lemn})\\[2pt]
$\sigma(t)=\sinl(\varpi t)$ & normalized lemniscatic sine, of period $2$\\[2pt]
$\gamma(t)=\cosl(\varpi t)$ & normalized lemniscatic cosine, $\gamma(t)=\sigma(t+\frac12)$\\[2pt]
$S(x)$, $\M=S^{-1}(0)$ & the defining function \eqref{Mdef} and the hypersurface\\[2pt]
$\Sing$, $\Mreg$ & the singular set $\Z^4\cup(h+D_4)$ and the regular part of $\M$\\[2pt]
$\Mo$ & the skeleton of $\M$, a union of coordinate $2$-planes\\[2pt]
$e_1,\dots,e_4$ & the standard basis of $\R{4}$\\[2pt]
$D_4$, $D_4^*$ & the checkerboard lattice and its dual $\Z^4\cup(h+\Z^4)$\\[2pt]
$\F=\{\pm f_1,\dots,\pm f_4\}$ & the frame \eqref{frame}: $f_{1,2}=e_1\pm e_3$, $f_{3,4}=e_2\pm e_4$\\[2pt]
$\xi_k=\frac12\scal{\cdot}{f_k}$ & frame coordinates \eqref{xi}\\[2pt]
$\mathcal{E}_j(C)$ & elementary symmetric functions of $C_k=\gamma^2(\xi_k)$ (Section~\ref{sec:hole})\\[2pt]
$\Oct$ & the hyperoctahedral group of $\F$, of order $384$\\[2pt]
$A$ & the hidden symmetry \eqref{Amatrix}\\[2pt]
$P_n$, $T_n$, $K_n$ & the unit cube $n+[0,1]^4$, the torus in its boundary, and $\M\cap P_n$\\[2pt]
$\Lambda$, $\vartheta_\Lambda$ & the set $\frac12(e_1+e_3)+D_4^*$ and its theta series (Section~\ref{sec:hole})\\
\end{tabular}

\section{Lemniscatic functions}\label{sec:lemn}

In this section we collect some classical facts on the lemniscatic sine; see, for example, \cite{Markush,LangerSinger,Schappacher,AkhiezerEll,Lawden}. For the reader's convenience we indicate short proofs.

Let $s(t)=\sinl t$ be the solution of the initial value problem
\begin{equation}\label{ODE2}
s''=-2s^3,\qquad s(0)=0,\quad s'(0)=1.
\end{equation}
Multiplying by $s'$ and integrating, we obtain the first integral
\begin{equation}\label{ODE1}
s'(t)^2=1-s(t)^4 .
\end{equation}
In particular $|s|\le1$, hence $s$ is defined, and real analytic, on the whole real line. Let
$$
\varpi=2\int_0^1\frac{du}{\sqrt{1-u^4}}=\frac{\Gamma(\frac14)^2}{2\sqrt{2\pi}}=2.6220575542\ldots
$$
be the lemniscate constant; the second equality follows from the substitution $u^4=v$ and the Euler integral for the beta function.

\begin{lemma}\label{lem:sl}
The function $s=\sinl$ has the following properties.
\begin{enumerate}
\item[(i)] $s$ is odd, $s(\varpi-t)=s(t)$ and $s(t+\varpi)=-s(t)$; in particular $s$ has period $2\varpi$.
\item[(ii)] $s(t)=0$ if and only if $t\in\varpi\Z$; $s>0$ on $(0,\varpi)$; $|s(t)|=1$ if and only if $t\in\frac{\varpi}2+\varpi\Z$, and $s(\frac{\varpi}{2}+k\varpi)=(-1)^k$.
\item[(iii)] The lemniscatic cosine $\cosl t:=s(\frac{\varpi}{2}-t)=s(\frac{\varpi}{2}+t)$ is even, satisfies $\cosl''=-2\cosl^3$, $\cosl(0)=1$, $\cosl'(0)=0$, and
\begin{equation}\label{Taylor}
s(t)=t-\frac{t^5}{10}+O(t^9),\qquad \cosl t=1-t^2+O(t^4)\qquad (t\to0).
\end{equation}
\item[(iv)] \textup{(Addition theorem)} For all real $u,v$,
\begin{equation}\label{addition}
s(u+v)\,s(u-v)=\frac{s^2(u)-s^2(v)}{1+s^2(u)s^2(v)} .
\end{equation}
\item[(v)] \textup{(Log-concavity)} On $(0,\varpi)$ one has $(\log s)''=-(1+s^4)/s^2<0$.
\item[(vi)] $\cosl t=s'(t)/(1+s^2(t))$ and $(1+s^2(t))(1+\cosl^2t)=2$ for all real $t$.
\end{enumerate}
\end{lemma}

\begin{proof}
(i)--(ii). Since \eqref{ODE2} is invariant under $s\mapsto -s$ and $t\mapsto -t$, uniqueness of solutions shows that $s$ is odd. By \eqref{ODE1}, $s$ increases on the maximal interval $[0,t_0]$ where $s'>0$, and its inverse function there is $t=\int_0^s du/\sqrt{1-u^4}$; hence $s(t_0)=1$, $s'(t_0)=0$ and $t_0=\varpi/2$. The functions $s(t)$ and $s(\varpi-t)$ solve \eqref{ODE2} with the same data at $t=\varpi/2$, hence they coincide. Combining this with the oddness, $s(t+\varpi)=s(\varpi-(-t))=s(-t)=-s(t)$. Now (ii) follows from the monotonicity of $s$ on $[0,\varpi/2]$ and the symmetries in (i).

(iii) The function $\cosl$ is even by (i), and the stated equation and initial values follow from \eqref{ODE2} and (ii). Expanding the solutions of $y''=-2y^3$ in power series, we obtain \eqref{Taylor}: in the first case $y=t+at^5+\dots$ gives $20a=-2$, and in the second case $y''(0)=-2$.

(iv) We use Euler's addition theorem for the lemniscatic sine \cite{Markush,LangerSinger},
$$
s(u+v)=\frac{s(u)s'(v)+s(v)s'(u)}{1+s^2(u)s^2(v)}.
$$
Replacing $v$ by $-v$ and multiplying, we obtain, by \eqref{ODE1},
$$
s(u+v)s(u-v)=\frac{s^2(u)(1-s^4(v))-s^2(v)(1-s^4(u))}{(1+s^2(u)s^2(v))^2}
=\frac{(s^2(u)-s^2(v))(1+s^2(u)s^2(v))}{(1+s^2(u)s^2(v))^2},
$$
which is \eqref{addition}. An equivalent form of \eqref{addition} can be found in \cite[(2.4.16)]{Lawden}, see also Remark~\ref{rem:wp}.

(v) By \eqref{ODE2} and \eqref{ODE1}, $ss''-s'^2=-2s^4-(1-s^4)=-(1+s^4)$.

(vi) Euler's addition theorem with $v=\varpi/2$ gives $\cosl(-u)=s(u+\frac{\varpi}2)=s'(u)/(1+s^2(u))$, because $s(\frac{\varpi}2)=1$ and $s'(\frac{\varpi}2)=0$. Hence, by \eqref{ODE1}, $\cosl^2=(1-s^4)/(1+s^2)^2=(1-s^2)/(1+s^2)$, which is equivalent to the second identity \cite[p.~53]{Markush}.
\end{proof}

\begin{remark}\label{rem:wp}
The function $s$ is the restriction to the real line of a doubly periodic meromorphic function satisfying $s(\mathrm{i}z)=\mathrm{i}\,s(z)$. Its period lattice is generated by $(1+\mathrm{i})\varpi$ and $(1-\mathrm{i})\varpi$. It is the Jacobi elliptic function $\sn(z,k)$ with the imaginary modulus $k=\mathrm{i}$; with the real modulus $k=1/\sqrt2$ one has $s(t)=\sn(\sqrt2t,\frac1{\sqrt2})/(\sqrt2\,\dn(\sqrt2 t,\frac1{\sqrt2}))$. Finally, $1/s^2$ is the Weierstrass function $\wp$ with the invariants $g_2=4$, $g_3=0$, that is, $\wp'^2=4\wp^3-4\wp$; indeed, if $w=s^{-2}$ then $w'^2=4s'^2s^{-6}=4(w^3-w)$ by \eqref{ODE1}.
\end{remark}

It will be convenient to work with the normalized functions
\begin{equation}\label{norm}
\sigma(t)=s(\varpi t),\qquad \gamma(t)=\cosl(\varpi t)=\sigma(t+\tfrac12).
\end{equation}
By Lemma~\ref{lem:sl}, $\sigma$ is odd, $\sigma(t+1)=-\sigma(t)$, $\sigma$ vanishes exactly on $\Z$, $\sigma$ has the sign $(-1)^k$ on $(k,k+1)$, $|\sigma(t)|=1$ if and only if $t\in\frac12+\Z$, and
\begin{equation}\label{TaylorN}
\begin{aligned}
\sigma(k+t)&=(-1)^k\bigl(\varpi t+O(t^5)\bigr),\\
\sigma(\tfrac12+k+t)&=(-1)^k\gamma(t)=(-1)^k\bigl(1-\varpi^2t^2+O(t^4)\bigr)
\end{aligned}
\end{equation}
for $k\in\Z$ and $t\to 0$. By Lemma~\ref{lem:sl}(vi),
\begin{equation}\label{sigmagamma}
\sigma^2(t)=\frac{1-\gamma^2(t)}{1+\gamma^2(t)},\qquad t\in\R{}.
\end{equation}
The addition theorem \eqref{addition} holds for $\sigma$ in the same form:
\begin{equation}\label{additionN}
\sigma(u+v)\,\sigma(u-v)=\frac{\sigma^2(u)-\sigma^2(v)}{1+\sigma^2(u)\sigma^2(v)},\qquad u,v\in\R{}.
\end{equation}

\section{Minimality and the local structure: proof of Theorem~\ref{thA}}\label{sec:thA}

\subsection{The mean curvature of a level set}
Let $u$ be a smooth function on an open set $\Omega\subset\R{n}$ and let $Du\ne0$ on $\Omega$. Then $\{u=0\}\cap\Omega$ is a smooth hypersurface with unit normal $\nu=Du/|Du|$, and its mean curvature (the trace of the second fundamental form with respect to $-\nu$) is
\begin{equation}\label{meancurv}
H=\operatorname{div}\frac{Du}{|Du|}=\frac{\Delta_1u}{|Du|^3},\qquad\text{where}\qquad
\Delta_1u:=|Du|^2\Delta u-\frac12\scal{Du}{D|Du|^2}
\end{equation}
(see, for example, \cite{Hsiang67} or \cite[Theorem~2.1]{Perdomo3}). Hence the level set $\{u=0\}$ is minimal (on the set where $Du\neq0$) as soon as $\Delta_1u$ vanishes on it.

Since homotheties preserve minimality, it suffices to study the hypersurface $\varpi\M=\{S_0=0\}$, where
$$
S_0(y)=S(y/\varpi)=s(y_1)s(y_3)-s(y_2)s(y_4),\qquad s=\sinl .
$$
We shall use the abbreviations $s_i=s(y_i)$ and
\begin{equation}\label{abbr}
a=s_1s_3,\qquad b=s_2s_4,\qquad p_1=s_1^2+s_3^2,\qquad p_2=s_2^2+s_4^2,
\end{equation}
so that $S_0=a-b$.

\begin{lemma}\label{lem:key}
In the notation \eqref{abbr},
\begin{equation}\label{grad}
|DS_0|^2=(1-a^2)p_1+(1-b^2)p_2,
\end{equation}
and
\begin{equation}\label{keyidentity}
\Delta_1S_0=-2S_0\Bigl(1-(a^4+a^3b+a^2b^2+ab^3+b^4)+p_1p_2(1+ab)\Bigr).
\end{equation}
In particular, $\Delta_1S_0$ vanishes identically on $\{S_0=0\}$.
\end{lemma}

\begin{proof}
Consider first the function $r(y_1,y_3)=s(y_1)s(y_3)$ of two variables, and let $p=s_1^2+s_3^2$. By \eqref{ODE1},
$$
|Dr|^2=s_1'^2s_3^2+s_1^2s_3'^2=(1-s_1^4)s_3^2+(1-s_3^4)s_1^2=(1-r^2)p,
$$
and by \eqref{ODE2},
$$
\Delta r=s_1''s_3+s_1s_3''=-2r p,\qquad
\scal{Dr}{Dp}=2s_1s_3(s_1'^2+s_3'^2)=2r(2-p^2+2r^2).
$$
Therefore
$$
\scal{Dr}{D\bigl((1-r^2)p\bigr)}=(1-r^2)\scal{Dr}{Dp}-2rp|Dr|^2=4r(1-r^2)(1+r^2-p^2).
$$
Now $a$ and $p_1$ depend only on $(y_1,y_3)$, while $b$ and $p_2$ depend only on $(y_2,y_4)$; hence the above formulas apply to $(a,p_1)$ and to $(b,p_2)$, and the corresponding gradients are orthogonal. This gives \eqref{grad} and
\begin{align*}
\Delta S_0&=-2ap_1+2bp_2,\\
\scal{DS_0}{D|DS_0|^2}&=4a(1-a^2)(1+a^2-p_1^2)-4b(1-b^2)(1+b^2-p_2^2).
\end{align*}
Substituting these expressions into the definition of $\Delta_1$, we find
\begin{align*}
\tfrac12\Delta_1S_0
&=\bigl((1-a^2)p_1+(1-b^2)p_2\bigr)(bp_2-ap_1)\\
&\qquad-a(1-a^2)(1+a^2-p_1^2)+b(1-b^2)(1+b^2-p_2^2)\\
&=p_1p_2\bigl(b(1-a^2)-a(1-b^2)\bigr)-a(1-a^4)+b(1-b^4),
\end{align*}
because the terms containing $p_1^2$ and $p_2^2$ cancel. Finally, $b(1-a^2)-a(1-b^2)=(b-a)(1+ab)$ and
$-a(1-a^4)+b(1-b^4)=(b-a)\bigl(1-(a^4+a^3b+a^2b^2+ab^3+b^4)\bigr)$, which proves \eqref{keyidentity}.
\end{proof}

\subsection{The singular set}

\begin{lemma}\label{lem:sing}
Let $x\in\M$. Then $DS(x)=0$ if and only if $x\in\Z^4\cup(h+D_4)$. Moreover, $\M$ does not contain any point of $h'+D_4$, and $|S|=2$ on $h'+D_4$.
\end{lemma}

\begin{proof}
By \eqref{grad}, on $\varpi\M$ we have $a=b$, and hence
$$
|DS_0|^2=(1-a^2)(p_1+p_2).
$$
Thus, in the normalized variables, a point $x\in\M$ is a critical point of $S$ if and only if either
$$
\sigma(x_1)=\sigma(x_2)=\sigma(x_3)=\sigma(x_4)=0,\qquad\text{or}\qquad \sigma(x_1)^2\sigma(x_3)^2=1 .
$$
In the first case $x\in\Z^4$; conversely, every point of $\Z^4$ lies in $\M$.

In the second case $|\sigma(x_1)|=|\sigma(x_3)|=1$, because $|\sigma|\le 1$. Since $x\in\M$, also $|\sigma(x_2)\sigma(x_4)|=|\sigma(x_1)\sigma(x_3)|=1$, hence $|\sigma(x_i)|=1$ for all $i$. By Lemma~\ref{lem:sl}(ii) this means that $x=h+m$ with $m\in\Z^4$, and then $\sigma(x_i)=(-1)^{m_i}$. The condition $S(x)=0$ now reads
$$
(-1)^{m_1+m_3}=(-1)^{m_2+m_4},
$$
i.e. $m\in D_4$. If instead $m\in e_1+D_4$, then $S(h+m)=\pm2$. Since $h'=h-e_4\in h+e_1+D_4$, these are exactly the points of $h'+D_4$.
\end{proof}

From now on we write
\begin{align*}
\Sing&=\Z^4\cup(h+D_4),\\
\Mreg&=\M\setminus\Sing.
\end{align*}

\begin{proof}[Proof of Theorem~\ref{thA}\,\textup{(a)}, \textup{(b)}]
Since $\sigma$ is real analytic, $\M$ is a closed real analytic subset of $\R{4}$. If $n\in\Z^4$ then $\sigma(x_i+n_i)=(-1)^{n_i}\sigma(x_i)$, hence
\begin{equation}\label{transl}
S(x+n)=(-1)^{n_1+n_3}\sigma(x_1)\sigma(x_3)-(-1)^{n_2+n_4}\sigma(x_2)\sigma(x_4),
\end{equation}
and for $n\in D_4$ the two signs coincide, so that $S(x+n)=\pm S(x)$. Thus $\M$ is $D_4$-invariant, and (a) follows from Lemma~\ref{lem:sing}. By the implicit function theorem, $\Mreg$ is a real analytic embedded hypersurface in $\R{4}\setminus\Sing$; it is properly embedded because it is closed in $\R{4}\setminus\Sing$. Its minimality follows from \eqref{meancurv} and Lemma~\ref{lem:key}.
\end{proof}

\subsection{Local structure near the singular points}

\begin{lemma}\label{lem:expansions}
Let $y\in\Sing$ and $x\to0$.
\begin{enumerate}
\item[(i)] If $y\in\Z^4$, then
$S(y+x)=\varpi^2\bigl(\epsilon_{13}\,x_1x_3-\epsilon_{24}\,x_2x_4\bigr)+O(|x|^6)$, where $\epsilon_{13}=(-1)^{y_1+y_3}$ and $\epsilon_{24}=(-1)^{y_2+y_4}$.
\item[(ii)] If $y=h+m$, $m\in D_4$, then
$S(y+x)=\epsilon\,\varpi^2\bigl(x_2^2+x_4^2-x_1^2-x_3^2\bigr)+O(|x|^4)$, where $\epsilon=(-1)^{m_1+m_3}$.
\end{enumerate}
In both cases the quadratic part of $S$ at $y$ is a nondegenerate quadratic form of signature $(2,2)$ whose eigenvalues are $\pm\frac{\varpi^2}{2}$ \textup{(}case (i)\textup{)} or $\pm\varpi^2$ \textup{(}case (ii)\textup{)}, each of multiplicity two.
\end{lemma}

\begin{proof}
Both expansions follow immediately from \eqref{TaylorN}; in case (ii) we also use $(-1)^{m_1+m_3}=(-1)^{m_2+m_4}$. The statement about the eigenvalues is obvious for (ii), and for (i) it follows from the identity
$4x_1x_3=(x_1+x_3)^2-(x_1-x_3)^2$ and the similar identity for $x_2x_4$.
\end{proof}

The quadratic forms of Lemma~\ref{lem:expansions} have the same eigenvalues up to a positive factor; hence their zero sets $C_y$ are all congruent to the Clifford cone \eqref{Clifford}. To pass from the quadratic part to the hypersurface itself we use the following real analytic version of the Morse lemma (cf. \cite[Lemma~2.2]{MilnorMorse}); we include the short proof because we need the normalization $d\Phi(0)=\Id$.

\begin{lemma}\label{lem:Morse}
Let $f$ be a real analytic function near $0\in\R{n}$ with $f(0)=0$, $Df(0)=0$, and nondegenerate Hessian, and let $q(x)=\frac12\scal{D^2f(0)x}{x}$. Then there is a real analytic diffeomorphism $\Psi$ between neighbourhoods of $0$ with $\Psi(0)=0$, $d\Psi(0)=\Id$ and $f(\Psi(z))=q(z)$.
\end{lemma}

\begin{proof}
We use a matrix-valued Hadamard argument to obtain the analytic Morse normal form.
By Taylor's formula with integral remainder, $f(x)=\scal{H(x)x}{x}$, where
$$
H(x)=\int_0^1(1-t)D^2f(tx)\,dt
$$
is a real analytic family of symmetric matrices with $H(0)=H_0:=\frac12D^2f(0)$. Put $C(x)=H_0^{-1}H(x)$; then $C(0)=I$, and $C(x)$ is self-adjoint with respect to the nondegenerate symmetric bilinear form $\scal{H_0\,\cdot}{\cdot}$. Let $B(x)=C(x)^{1/2}$ be defined, for small $x$, by the binomial series in $C(x)-I$. Then $B$ is real analytic, $B(0)=I$, $B(x)$ commutes with $C(x)$ and is $H_0$-self-adjoint, that is, $H_0B=B^{T}H_0$. Hence
$$
B^TH_0B=H_0B^2=H_0C=H.
$$
The map $z=\chi(x):=B(x)x$ satisfies $d\chi(0)=I$ and
$$
q(\chi(x))=\scal{H_0B(x)x}{B(x)x}=\scal{H(x)x}{x}=f(x).
$$
Therefore $\Psi=\chi^{-1}$ has the required properties.
\end{proof}

\begin{proof}[Proof of Theorem~\ref{thA}\,\textup{(c)}]
Let $y\in\Sing$ and $f(x)=S(y+x)$. By Lemma~\ref{lem:expansions}, $f(0)=0$, $Df(0)=0$, and the quadratic part of $f$ is the nondegenerate form
$$
q_y(x)=
\begin{cases}
\varpi^2\bigl(\epsilon_{13}\,x_1x_3-\epsilon_{24}\,x_2x_4\bigr), & y\in\Z^4,\\[2pt]
\epsilon\,\varpi^2\bigl(x_2^2+x_4^2-x_1^2-x_3^2\bigr), & y=h+m\in h+D_4 ,
\end{cases}
$$
with the signs $\epsilon_{13},\epsilon_{24},\epsilon$ of Lemma~\ref{lem:expansions}. In the first case $\{q_y=0\}=\{x_1x_3=\epsilon_{13}\epsilon_{24}\,x_2x_4\}$ and $\epsilon_{13}\epsilon_{24}=(-1)^{y_1+y_2+y_3+y_4}$; in the second case the sign $\epsilon$ does not affect the zero set. Thus, in both cases,
$$
C_y=\{x\in\R{4}:\ q_y(x)=0\}
$$
is exactly the cone in the statement of Theorem~\ref{thA}(c).

By Lemma~\ref{lem:Morse}, there is a real analytic diffeomorphism $\Psi$ between neighbourhoods $U$ and $\Psi(U)$ of $0$, with $\Psi(0)=0$, $d\Psi(0)=\Id$ and $f\circ\Psi=q_y$. Put $\Phi(z)=y+\Psi(z)$. Then $S(\Phi(z))=q_y(z)$ for $z\in U$, hence
$$
\M\cap\Phi(U)=\Phi(C_y\cap U),\qquad \Phi(0)=y,\qquad d\Phi(0)=\Id .
$$

Since $C_y$ is congruent to the cone over a $2$-torus, every small punctured neighbourhood of the vertex in $C_y$ is homeomorphic to $T^2\times(0,1)$. Hence
$$
H_2(C_y,C_y\setminus\{0\})\cong \tilde H_1(T^2)\cong \Z^2\neq0,
$$
while for a topological $3$-manifold this local homology group vanishes. Thus $\M$ is not a topological manifold near $y$.
\end{proof}

\subsection{$\M$ as a stationary varifold}
We use the standard terminology of geometric measure theory \cite{Allard,SimonGMT}. For an integral varifold $V$, a point $y$ and $r>0$, let $V_{y,r}$ denote the image of $V$ under the homothety $x\mapsto (x-y)/r$; a \emph{tangent cone} of $V$ at $y$ is a varifold limit of $V_{y,r_j}$ along a sequence $r_j\to0$.

\begin{proof}[Proof of Theorem~\ref{thA}\,\textup{(d)}]
We proceed in four steps.

\emph{Step 1: rescaled charts.} Let $y\in\Sing$, and let $\Phi=y+\Psi$ be as in (c) above, defined on a ball $B_\rho\subset U$. For $0<r<\rho/3$ consider the rescaled maps
$$
\Phi_r(z)=\frac{\Phi(rz)-y}{r}=\frac{\Psi(rz)}{r},\qquad z\in B_3 .
$$
Since $\Psi$ is real analytic with $\Psi(0)=0$ and $d\Psi(0)=\Id$, we have
$$
\Phi_r(z)=z+O(r|z|^2),\qquad d\Phi_r(z)=d\Psi(rz)=\Id+O(r|z|),
$$
uniformly for $z\in B_3$. In particular, $\Phi_r\to\Id$ in $C^1(B_3)$ as $r\to0$, and $\Phi_r(B_3)\supset B_2$ for all small $r$. Since $C_y$ is a cone, the conclusion of (c) rescales to
\begin{equation}\label{rescaledchart}
\frac{\M-y}{r}\cap\Phi_r(B_3)=\Phi_r(C_y\cap B_3).
\end{equation}

\emph{Step 2: area bounds.} For a nonnegative continuous function $g$ with support in $B_2$, the area formula and \eqref{rescaledchart} give
\begin{equation}\label{areaformula}
\int_{(\M-y)/r}g\,d\mathcal{H}^3=\int_{C_y\cap B_3}g(\Phi_r(z))\,J_r(z)\,d\mathcal{H}^3(z),
\end{equation}
where $J_r(z)$ is the Jacobian of $d\Phi_r(z)$ restricted to the tangent plane of $C_y$ at $z$ (defined for $z\neq0$). By Step 1, $J_r\to1$ uniformly on $C_y\cap B_3$. Taking $g\le 1$ with $g=1$ on $B_1$, we obtain
\begin{equation}\label{areabound}
\mathcal{H}^3(\M\cap B_r(y))\le c_y r^3,\qquad 0<r<r_y ,
\end{equation}
for suitable constants $c_y$ and $r_y>0$. Since $\Sing$ is discrete, $\M$ is $D_4$-periodic and $\Mreg$ is a properly embedded smooth hypersurface, it follows that $\M$ has locally finite $\mathcal{H}^3$-measure. Hence $\M$ defines an integral varifold $V$ with multiplicity one.

\emph{Step 3: stationarity.} Let $X$ be a smooth compactly supported vector field on $\R{4}$. Only finitely many singular points $y^{(1)},\dots,y^{(k)}$ lie in the support of $X$. For small $\varepsilon>0$ choose a smooth function $\eta_\varepsilon$ which equals $1$ outside $\bigcup_jB_{2\varepsilon}(y^{(j)})$, vanishes on $\bigcup_jB_{\varepsilon}(y^{(j)})$, and satisfies $|D\eta_\varepsilon|\le 2/\varepsilon$. Since $\Mreg$ is minimal and $\eta_\varepsilon X$ has compact support in $\R{4}\setminus\Sing$, the first variation formula gives
$$
\int_{\M}\operatorname{div}_{\M}(\eta_\varepsilon X)\,d\mathcal{H}^3=0 .
$$
Expanding the tangential divergence and using \eqref{areabound}, we obtain
$$
\Bigl|\int_{\M}\eta_\varepsilon\operatorname{div}_{\M}X\,d\mathcal{H}^3\Bigr|\le\int_{\M}|D\eta_\varepsilon||X|\,d\mathcal{H}^3\le \frac{2\sup|X|}{\varepsilon}\sum_{j=1}^k c_{y^{(j)}}(2\varepsilon)^3\longrightarrow0 .
$$
Letting $\varepsilon\to0$, we obtain $\int_{\M}\operatorname{div}_{\M}X\,d\mathcal{H}^3=0$, i.e. $V$ is stationary.

\emph{Step 4: tangent cone and density.} Let $g$ be a continuous function on the Grassmann bundle $\R{4}\times G(3,4)$ with compact support in $B_2\times G(3,4)$. Applying the area formula as in \eqref{areaformula}, and using that $d\Phi_r$ maps the tangent plane $T_zC_y$ onto the tangent plane of $(\M-y)/r$ at $\Phi_r(z)$, we get
$$
V_{y,r}(g)=\int_{C_y\cap B_3}g\bigl(\Phi_r(z),\,d\Phi_r(z)\,T_zC_y\bigr)\,J_r(z)\,d\mathcal{H}^3(z).
$$
By Step 1, $\Phi_r\to\Id$ and $d\Phi_r\to\Id$ uniformly, and $J_r\to1$ uniformly.
Thus, after pulling the integral back to $C_y$, the integrand converges pointwise and is uniformly bounded on the compact support of $g$.
By dominated convergence, $V_{y,r}(g)$ converges to
$$
\int_{C_y}g(z,T_zC_y)\,d\mathcal{H}^3(z)
$$
as $r\to0$, along \emph{all} $r\to0$, not only along subsequences. The restriction to test functions supported in $B_2\times G(3,4)$ is harmless. Indeed, if $g$ has compact support in $B_{2R}\times G(3,4)$ for some $R\ge1$, put $g_R(z,P)=g(Rz,P)$; then $V_{y,r}(g)=R^3\,V_{y,rR}(g_R)$, and the above applied to $g_R$, together with the homogeneity of $C_y$, gives
$$
V_{y,r}(g)\longrightarrow R^3\int_{C_y}g_R(z,T_zC_y)\,d\mathcal{H}^3(z)=\int_{C_y}g(z,T_zC_y)\,d\mathcal{H}^3(z).
$$
Hence $V_{y,r}$ converges, as varifolds, to the multiplicity-one varifold of the cone $C_y$, and $C_y$ is the unique tangent cone of $V$ at $y$. Finally, since $\mathcal{H}^3(C_y\cap\partial B_1)=0$, the masses converge as well:
$$
\Theta(V,y)=\lim_{r\to0}\frac{\mathcal{H}^3(\M\cap B_r(y))}{\mathcal{H}^3(B^3_1)\,r^3}=\frac{\mathcal{H}^3(C\cap B_1)}{\mathcal{H}^3(B^3_1)} .
$$
The cone $C$ is the cone over the Clifford torus of area $(2\pi/\sqrt2)^2=2\pi^2$, so $\mathcal{H}^3(C\cap B_1)=\frac{2\pi^2}{3}$, and therefore
$$
\Theta(V,y)=\frac{2\pi^2/3}{4\pi/3}=\frac{\pi}{2}.
$$
This completes the proof of Theorem~\ref{thA}.
\end{proof}

\section{The lattice $D_4$ and the $24$-cell honeycomb}\label{sec:lattices}

In this section we recall some standard facts on four-dimensional lattices; see \cite[Chapters~4 and 21]{Conway} and \cite{ConwaySmith}.
The \emph{checkerboard lattice} $D_4$ is the sublattice of index $2$ in $\Z^4$ formed by the integer vectors with even coordinate sum. Its shortest nonzero vectors are the $24$ vectors $\pm e_i\pm e_j$, $i<j$, which form the root system $D_4$. The dual lattice is
$$
D_4^*=\{x\in\R{4}:\ \scal{x}{m}\in\Z\ \text{for all}\ m\in D_4\}=\Z^4\cup(h+\Z^4),
$$
and $D_4^*/D_4\cong\Z_2\times\Z_2$. The four cosets of $D_4$ in $D_4^*$ are
\begin{equation}\label{cosets}
D_4,\qquad e_1+D_4,\qquad h+D_4,\qquad h'+D_4 .
\end{equation}
In this notation $\Z^4=D_4\sqcup(e_1+D_4)$ and, by Theorem~\ref{thA},
$$
\Sing=D_4\sqcup(e_1+D_4)\sqcup(h+D_4)=D_4^*\setminus(h'+D_4).
$$
Thus the singular set of $\M$ is obtained from the lattice $D_4^*$ by removing one of the three nontrivial cosets of $D_4$, namely the set of holes $h'+D_4$.

\begin{remark}\label{rem:Hurwitz}
If we identify $\R{4}$ with the quaternions $\Ham$ via $x\mapsto x_1+x_2\mathrm{i}+x_3\mathrm{j}+x_4\mathrm{k}$, then $\Z^4$ becomes the ring of Lipschitz quaternions, $D_4^*=\Z^4\cup(h+\Z^4)$ becomes the ring of Hurwitz integral quaternions $\mathcal{O}$, and $D_4$ becomes the two-sided ideal $(1+\mathrm{i})\mathcal{O}$. The quotient $\mathcal{O}/(1+\mathrm{i})\mathcal{O}$ is the field with four elements $\{0,1,\omega,\omega^2\}$, $\omega=\frac12(-1+\mathrm{i}+\mathrm{j}+\mathrm{k})$, and in this language
$$
\Sing=\mathcal{O}\setminus\bigl(\omega+(1+\mathrm{i})\mathcal{O}\bigr).
$$
The norm one elements of $\mathcal{O}$ are the $24$ Hurwitz units
$$
\pm1,\pm\mathrm{i},\pm\mathrm{j},\pm\mathrm{k}, \qquad
\tfrac12(\pm1\pm\mathrm{i}\pm\mathrm{j}\pm\mathrm{k});
$$
they are the vertices of a regular $24$-cell (Figure~\ref{fig:24cell}). The translate $\frac12(1+\mathrm{j})+\mathcal{O}$ of the Hurwitz order will appear in Theorem~\ref{thm:theta}.
\end{remark}

The reflections in the roots $\pm e_i\pm e_j$ generate the Weyl group $W(D_4)$; it consists of the permutations of coordinates combined with an even number of sign changes, so that
$$
|W(D_4)|=2^3\cdot4!=192 .
$$
Every element of $W(D_4)$ preserves $D_4$ and maps each of the four cosets \eqref{cosets} onto itself. The full automorphism group $\Aut(D_4)$ of $D_4$, i.e. the group of orthogonal transformations $L$ with $L(D_4)=D_4$, is larger. Since an orthogonal map preserves a lattice if and only if it preserves its dual, $\Aut(D_4)=\Aut(D_4^*)$, and therefore $\Aut(D_4)$ permutes the three nontrivial cosets in \eqref{cosets}. Every permutation occurs, and the kernel of this action is exactly $W(D_4)$ (this is the triality of $D_4$). Consequently, $\Aut(D_4)$ is the Weyl group $W(F_4)$, of order
$$
|W(F_4)|=6\,|W(D_4)|=2^7\cdot3^2=1152 .
$$
For example, the reflection
\begin{equation}\label{R4}
R(x_1,x_2,x_3,x_4)=(x_1,x_2,x_3,-x_4)
\end{equation}
fixes $e_1+D_4$ and interchanges $h+D_4$ and $h'+D_4$. The symmetric orthogonal matrix $A$ in \eqref{Amatrix} satisfies $A^2=I$ and maps the generators $e_1-e_2$, $e_2-e_3$, $e_3-e_4$, $e_3+e_4$ of $D_4$ to $(0,0,1,1)$, $(0,1,-1,0)$, $(0,0,1,-1)$, $(1,-1,0,0)$, respectively; hence $A\in\Aut(D_4)$. Moreover, $A$ fixes $h'$ and interchanges $h$ and $e_1$; thus it interchanges $e_1+D_4$ and $h+D_4$.

\medskip
\noindent\textit{The $24$-cell honeycomb.} The Voronoi cell of the lattice $D_4$ at the origin, i.e. the set of points of $\R{4}$ which are at least as close to $0$ as to any other point of $D_4$, is the regular $24$-cell whose vertices are the $24$ points of $D_4^*\setminus D_4$ of norm one: the $8$ points $\pm e_i$ and the $16$ points $\frac12(\pm1,\pm1,\pm1,\pm1)$ (these are the \emph{deep holes} of $D_4$) \cite[Chapters~4 and 21]{Conway}. The translates of this cell by $D_4$ form the regular honeycomb of $\R{4}$ by $24$-cells \cite{Coxeter}, whose vertex set is $D_4^*\setminus D_4$. Translating everything by $h'$, we obtain:

\begin{lemma}\label{lem:honeycomb}
$\Sing$ is the vertex set of the $24$-cell honeycomb formed by the Voronoi cells of the points of $h'+D_4$. In particular, every hole $c\in h'+D_4$ is surrounded by exactly $24$ singular points of $\M$ at distance $1$ from $c$, the vertices of a regular $24$-cell centred at $c$.
\end{lemma}

\begin{proof}
The vertex set of the honeycomb formed by the Voronoi cells of $h'+D_4$ is
$$
h'+(D_4^*\setminus D_4)=D_4^*\setminus(h'+D_4),
$$
which equals $\Sing$ by the discussion above. The second statement follows because the Voronoi cell of $c$ is the translate by $c$ of the $24$-cell with vertices $\pm e_i$, $\frac12(\pm1,\pm1,\pm1,\pm1)$.
\end{proof}

The symmetry group of the honeycomb, i.e. the group of all isometries preserving its vertex set $\Sing$, will be described in Lemma~\ref{lem:Gsigma}.

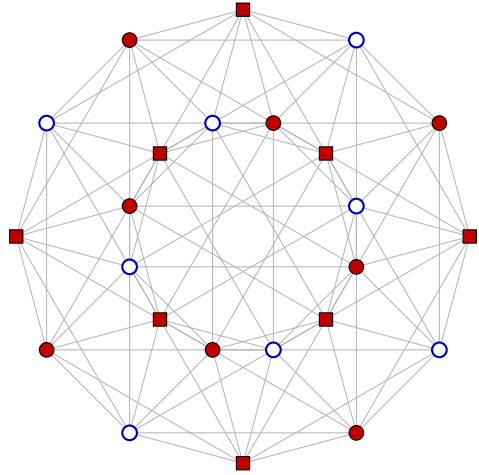
\begin{figure}[ht]
\centering

\begin{tikzpicture}[scale=1]
\draw[gray!60,line width=0.3pt] (0.000,3.000)--(-1.500,2.598);
\draw[gray!60,line width=0.3pt] (0.000,3.000)--(-2.598,1.500);
\draw[gray!60,line width=0.3pt] (0.000,3.000)--(-0.402,1.500);
\draw[gray!60,line width=0.3pt] (0.000,3.000)--(-1.500,0.402);
\draw[gray!60,line width=0.3pt] (0.000,3.000)--(1.500,2.598);
\draw[gray!60,line width=0.3pt] (0.000,3.000)--(0.402,1.500);
\draw[gray!60,line width=0.3pt] (0.000,3.000)--(2.598,1.500);
\draw[gray!60,line width=0.3pt] (0.000,3.000)--(1.500,0.402);
\draw[gray!60,line width=0.3pt] (-0.000,-3.000)--(-1.500,-0.402);
\draw[gray!60,line width=0.3pt] (-0.000,-3.000)--(-2.598,-1.500);
\draw[gray!60,line width=0.3pt] (-0.000,-3.000)--(-0.402,-1.500);
\draw[gray!60,line width=0.3pt] (-0.000,-3.000)--(-1.500,-2.598);
\draw[gray!60,line width=0.3pt] (-0.000,-3.000)--(1.500,-0.402);
\draw[gray!60,line width=0.3pt] (-0.000,-3.000)--(0.402,-1.500);
\draw[gray!60,line width=0.3pt] (-0.000,-3.000)--(2.598,-1.500);
\draw[gray!60,line width=0.3pt] (-0.000,-3.000)--(1.500,-2.598);
\draw[gray!60,line width=0.3pt] (-3.000,0.000)--(-1.500,2.598);
\draw[gray!60,line width=0.3pt] (-3.000,0.000)--(-2.598,1.500);
\draw[gray!60,line width=0.3pt] (-3.000,0.000)--(-0.402,1.500);
\draw[gray!60,line width=0.3pt] (-3.000,0.000)--(-1.500,0.402);
\draw[gray!60,line width=0.3pt] (-3.000,0.000)--(-1.500,-0.402);
\draw[gray!60,line width=0.3pt] (-3.000,0.000)--(-2.598,-1.500);
\draw[gray!60,line width=0.3pt] (-3.000,0.000)--(-0.402,-1.500);
\draw[gray!60,line width=0.3pt] (-3.000,0.000)--(-1.500,-2.598);
\draw[gray!60,line width=0.3pt] (3.000,-0.000)--(1.500,2.598);
\draw[gray!60,line width=0.3pt] (3.000,-0.000)--(0.402,1.500);
\draw[gray!60,line width=0.3pt] (3.000,-0.000)--(2.598,1.500);
\draw[gray!60,line width=0.3pt] (3.000,-0.000)--(1.500,0.402);
\draw[gray!60,line width=0.3pt] (3.000,-0.000)--(1.500,-0.402);
\draw[gray!60,line width=0.3pt] (3.000,-0.000)--(0.402,-1.500);
\draw[gray!60,line width=0.3pt] (3.000,-0.000)--(2.598,-1.500);
\draw[gray!60,line width=0.3pt] (3.000,-0.000)--(1.500,-2.598);
\draw[gray!60,line width=0.3pt] (-1.098,1.098)--(-1.500,2.598);
\draw[gray!60,line width=0.3pt] (-1.098,1.098)--(-2.598,1.500);
\draw[gray!60,line width=0.3pt] (-1.098,1.098)--(1.500,2.598);
\draw[gray!60,line width=0.3pt] (-1.098,1.098)--(0.402,1.500);
\draw[gray!60,line width=0.3pt] (-1.098,1.098)--(-1.500,-0.402);
\draw[gray!60,line width=0.3pt] (-1.098,1.098)--(-2.598,-1.500);
\draw[gray!60,line width=0.3pt] (-1.098,1.098)--(1.500,-0.402);
\draw[gray!60,line width=0.3pt] (-1.098,1.098)--(0.402,-1.500);
\draw[gray!60,line width=0.3pt] (1.098,-1.098)--(-0.402,1.500);
\draw[gray!60,line width=0.3pt] (1.098,-1.098)--(-1.500,0.402);
\draw[gray!60,line width=0.3pt] (1.098,-1.098)--(2.598,1.500);
\draw[gray!60,line width=0.3pt] (1.098,-1.098)--(1.500,0.402);
\draw[gray!60,line width=0.3pt] (1.098,-1.098)--(-0.402,-1.500);
\draw[gray!60,line width=0.3pt] (1.098,-1.098)--(-1.500,-2.598);
\draw[gray!60,line width=0.3pt] (1.098,-1.098)--(2.598,-1.500);
\draw[gray!60,line width=0.3pt] (1.098,-1.098)--(1.500,-2.598);
\draw[gray!60,line width=0.3pt] (1.098,1.098)--(-1.500,2.598);
\draw[gray!60,line width=0.3pt] (1.098,1.098)--(-0.402,1.500);
\draw[gray!60,line width=0.3pt] (1.098,1.098)--(1.500,2.598);
\draw[gray!60,line width=0.3pt] (1.098,1.098)--(2.598,1.500);
\draw[gray!60,line width=0.3pt] (1.098,1.098)--(-1.500,-0.402);
\draw[gray!60,line width=0.3pt] (1.098,1.098)--(-0.402,-1.500);
\draw[gray!60,line width=0.3pt] (1.098,1.098)--(1.500,-0.402);
\draw[gray!60,line width=0.3pt] (1.098,1.098)--(2.598,-1.500);
\draw[gray!60,line width=0.3pt] (-1.098,-1.098)--(-2.598,1.500);
\draw[gray!60,line width=0.3pt] (-1.098,-1.098)--(-1.500,0.402);
\draw[gray!60,line width=0.3pt] (-1.098,-1.098)--(0.402,1.500);
\draw[gray!60,line width=0.3pt] (-1.098,-1.098)--(1.500,0.402);
\draw[gray!60,line width=0.3pt] (-1.098,-1.098)--(-2.598,-1.500);
\draw[gray!60,line width=0.3pt] (-1.098,-1.098)--(-1.500,-2.598);
\draw[gray!60,line width=0.3pt] (-1.098,-1.098)--(0.402,-1.500);
\draw[gray!60,line width=0.3pt] (-1.098,-1.098)--(1.500,-2.598);
\draw[gray!60,line width=0.3pt] (-1.500,2.598)--(-2.598,1.500);
\draw[gray!60,line width=0.3pt] (-1.500,2.598)--(-0.402,1.500);
\draw[gray!60,line width=0.3pt] (-1.500,2.598)--(1.500,2.598);
\draw[gray!60,line width=0.3pt] (-1.500,2.598)--(-1.500,-0.402);
\draw[gray!60,line width=0.3pt] (-2.598,1.500)--(-1.500,0.402);
\draw[gray!60,line width=0.3pt] (-2.598,1.500)--(0.402,1.500);
\draw[gray!60,line width=0.3pt] (-2.598,1.500)--(-2.598,-1.500);
\draw[gray!60,line width=0.3pt] (-0.402,1.500)--(-1.500,0.402);
\draw[gray!60,line width=0.3pt] (-0.402,1.500)--(2.598,1.500);
\draw[gray!60,line width=0.3pt] (-0.402,1.500)--(-0.402,-1.500);
\draw[gray!60,line width=0.3pt] (-1.500,0.402)--(1.500,0.402);
\draw[gray!60,line width=0.3pt] (-1.500,0.402)--(-1.500,-2.598);
\draw[gray!60,line width=0.3pt] (1.500,2.598)--(0.402,1.500);
\draw[gray!60,line width=0.3pt] (1.500,2.598)--(2.598,1.500);
\draw[gray!60,line width=0.3pt] (1.500,2.598)--(1.500,-0.402);
\draw[gray!60,line width=0.3pt] (0.402,1.500)--(1.500,0.402);
\draw[gray!60,line width=0.3pt] (0.402,1.500)--(0.402,-1.500);
\draw[gray!60,line width=0.3pt] (2.598,1.500)--(1.500,0.402);
\draw[gray!60,line width=0.3pt] (2.598,1.500)--(2.598,-1.500);
\draw[gray!60,line width=0.3pt] (1.500,0.402)--(1.500,-2.598);
\draw[gray!60,line width=0.3pt] (-1.500,-0.402)--(-2.598,-1.500);
\draw[gray!60,line width=0.3pt] (-1.500,-0.402)--(-0.402,-1.500);
\draw[gray!60,line width=0.3pt] (-1.500,-0.402)--(1.500,-0.402);
\draw[gray!60,line width=0.3pt] (-2.598,-1.500)--(-1.500,-2.598);
\draw[gray!60,line width=0.3pt] (-2.598,-1.500)--(0.402,-1.500);
\draw[gray!60,line width=0.3pt] (-0.402,-1.500)--(-1.500,-2.598);
\draw[gray!60,line width=0.3pt] (-0.402,-1.500)--(2.598,-1.500);
\draw[gray!60,line width=0.3pt] (-1.500,-2.598)--(1.500,-2.598);
\draw[gray!60,line width=0.3pt] (1.500,-0.402)--(0.402,-1.500);
\draw[gray!60,line width=0.3pt] (1.500,-0.402)--(2.598,-1.500);
\draw[gray!60,line width=0.3pt] (0.402,-1.500)--(1.500,-2.598);
\draw[gray!60,line width=0.3pt] (2.598,-1.500)--(1.500,-2.598);
\node[singZ] at (0.000,3.000) {};
\node[singZ] at (-0.000,-3.000) {};
\node[singZ] at (-3.000,0.000) {};
\node[singZ] at (3.000,-0.000) {};
\node[singZ] at (-1.098,1.098) {};
\node[singZ] at (1.098,-1.098) {};
\node[singZ] at (1.098,1.098) {};
\node[singZ] at (-1.098,-1.098) {};
\node[singH] at (-1.500,2.598) {};
\node[hole] at (-2.598,1.500) {};
\node[hole] at (-0.402,1.500) {};
\node[singH] at (-1.500,0.402) {};
\node[hole] at (1.500,2.598) {};
\node[singH] at (0.402,1.500) {};
\node[singH] at (2.598,1.500) {};
\node[hole] at (1.500,0.402) {};
\node[hole] at (-1.500,-0.402) {};
\node[singH] at (-2.598,-1.500) {};
\node[singH] at (-0.402,-1.500) {};
\node[hole] at (-1.500,-2.598) {};
\node[singH] at (1.500,-0.402) {};
\node[hole] at (0.402,-1.500) {};
\node[hole] at (2.598,-1.500) {};
\node[singH] at (1.500,-2.598) {};
\end{tikzpicture}

\caption{The $24$ points of $D_4^*$ at distance $1$ from the singular point $0$ (the Hurwitz units) are the vertices of a regular $24$-cell; shown is its orthogonal projection onto a Coxeter plane of $W(F_4)$, with all $96$ edges. Squares: the singular points $\pm e_i\in e_1+D_4$; filled discs: the singular points $\frac12(\pm1,\pm1,\pm1,\pm1)\in h+D_4$ (even number of minus signs); open circles: the holes $\frac12(\pm1,\pm1,\pm1,\pm1)\in h'+D_4$ (odd number of minus signs). Each class consists of the vertices of a $16$-cell. By Lemma~\ref{lem:honeycomb}, around a hole all $24$ nearest points of $D_4^*$ are singular.}\label{fig:24cell}
\end{figure}

\medskip
\noindent\textit{The frame $\F$ and its hyperoctahedral group.}
The roots $\pm e_i\pm e_j$ of $D_4$ split into three \emph{frames}, i.e. sets of four mutually orthogonal pairs $\pm v$; one of them is the frame $\F$ defined in \eqref{frame}. Its hyperoctahedral group $\Oct$ consists of the orthogonal maps $L$ of the form
\begin{equation}\label{Lpi}
Lf_i=\varepsilon_if_{\pi(i)},\qquad i=1,\dots,4,\qquad \varepsilon_i\in\{\pm1\},\ \pi\in S_4,
\end{equation}
and has order $2^4\cdot4!=384$. Since the vectors $f_i$ are pairwise orthogonal and have the same length $\sqrt2$, every choice of $\varepsilon_i$ and $\pi$ in \eqref{Lpi} defines an orthogonal map. It will be convenient to use the \emph{frame coordinates}
\begin{equation}\label{xi00}
\xi_i=\xi_i(x)=\tfrac12\scal{x}{f_i},
\end{equation}
i.e.
\begin{equation}\label{xi}
\xi_1=\tfrac12(x_1+x_3),\ \quad \xi_2=\tfrac12(x_1-x_3),\ \quad \xi_3=\tfrac12(x_2+x_4),\quad \xi_4=\tfrac12(x_2-x_4) .
\end{equation}
Note that
\begin{equation}\label{xixx}
x_1=\xi_1+\xi_2,\quad  x_3=\xi_1-\xi_2,\quad  x_2=\xi_3+\xi_4,\quad  x_4=\xi_3-\xi_4.
\end{equation}
If $L$ is given by \eqref{Lpi}, then
$$
\scal{Lx}{f_{\pi(i)}}=\scal{x}{L^{-1}f_{\pi(i)}}=\varepsilon_i\scal{x}{f_i},
$$
that is,
\begin{equation}\label{xiL}
\xi_{\pi(i)}(Lx)=\varepsilon_i\,\xi_i(x),\qquad i=1,\dots,4.
\end{equation}

The frame $\F$ is adapted to $\M$ at the origin: by Lemma~\ref{lem:expansions}(i), the quadratic part of $S$ at $0$ is, up to the positive factor $\varpi^2$,
$$
x_1x_3-x_2x_4=\xi_1^2-\xi_2^2-\xi_3^2+\xi_4^2=0.
$$
Hence the tangent cone of $\M$ at $0$ is its zero set, and the positive and negative eigenspaces of $D^2S(0)$ are spanned by $\{f_1,f_4\}$ and $\{f_2,f_3\}$, respectively.

\section{The view from a hole: a symmetric equation and a theta series}\label{sec:hole}

In this section we put the origin at the hole $h'$ and use the frame coordinates \eqref{xi}. Throughout, $y=x-h'$,
$$
\xi_k=\xi_k(y)=\tfrac12\scal{y}{f_k},\qquad C_k=\gamma^2(\xi_k)\in[0,1],\qquad k=1,\dots,4,
$$
where $\gamma(t)=\cosl(\varpi t)$, and $\mathcal{E}_j(C)$ denotes the $j$-th elementary symmetric function of $C_1,\dots,C_4$, for example,
\begin{align*}
\mathcal{E}_2(C)&=\sum_{1\le i<j\le 4}C_iC_j, \\
\mathcal{E}_4(C)&=C_1C_2C_3C_4.
\end{align*}
We shall see that, seen from a hole, $\M$ is given by an equation which is symmetric in the four frame directions; this makes the full symmetry group of Theorem~\ref{thB} visible. We then show that the same equation is the vanishing of a theta series of a translate of the Hurwitz order.

\subsection{A symmetric equation}

\begin{theorem}\label{thm:symmetric}
Let $\beta_k=\arctan C_k\in[0,\frac{\pi}4]$. Then, for all $y\in\R{4}$,
\begin{equation}\label{Ssym}
S(h'+y)=-\frac{1-\mathcal{E}_2(C)+\mathcal{E}_4(C)}{\prod_{k=1}^4(1+C_k^2)^{1/2}\,\cos(\beta_1+\beta_2-\frac{\pi}{4})\cos(\beta_3+\beta_4-\frac{\pi}{4})},
\end{equation}
and both cosines in the denominator are at least $\frac{1}{\sqrt2}$. Consequently,
\begin{equation}\label{Msym}
\M=h'+\Bigl\{y\in\R{4}:\ \mathcal{E}_2(C)=1+\mathcal{E}_4(C)\Bigr\}
=h'+\Bigl\{y\in\R{4}:\ \sum_{k=1}^4\arctan\gamma^2(\xi_k)=\frac{\pi}{2}\Bigr\}.
\end{equation}
\end{theorem}

\begin{proof}
Since $\sigma(\frac12+t)=\gamma(t)$, $\sigma(-\frac12+t)=-\gamma(t)$ and $\gamma$ is even, we have
$S(h'+y)=\gamma(y_1)\gamma(y_3)+\gamma(y_2)\gamma(y_4)$. By \eqref{xi}, $y_1=\xi_1+\xi_2$ and $y_3=\xi_1-\xi_2$, so the addition theorem \eqref{additionN} with $u=\frac12+\xi_1$, $v=\xi_2$ and the identity \eqref{sigmagamma} give
$$
\gamma(y_1)\gamma(y_3)=\frac{C_1-\sigma^2(\xi_2)}{1+C_1\sigma^2(\xi_2)}
=\frac{C_1+C_2+C_1C_2-1}{1+C_1+C_2-C_1C_2}
=\tan\Bigl(\beta_1+\beta_2-\frac{\pi}4\Bigr);
$$
the last equality is the addition formula for the tangent, written as $\frac{T-1}{T+1}=\tan(\theta-\frac{\pi}{4})$ for
$$
T=\tan\theta=\frac{C_1+C_2}{1-C_1C_2}.
$$
In the same way $\gamma(y_2)\gamma(y_4)=\tan(\beta_3+\beta_4-\frac{\pi}4)$. Put $a=\beta_1+\beta_2-\frac\pi4$ and $b=\beta_3+\beta_4-\frac{\pi}4$; then $a,b\in[-\frac\pi4,\frac\pi4]$, and
$$
\tan a+\tan b=\frac{\sin(a+b)}{\cos a\cos b}=-\frac{\cos(\beta_1+\beta_2+\beta_3+\beta_4)}{\cos a\cos b}.
$$
Finally,
$$
\cos\sum\beta_k=\operatorname{Re}\prod_k(1+\mathrm{i}C_k)\big/\prod_k|1+\mathrm{i}C_k|=(1-\mathcal{E}_2+\mathcal{E}_4)/\prod_k(1+C_k^2)^{1/2}.
$$
This proves \eqref{Ssym}. Since $\sum\beta_k\in[0,\pi]$, the numerator vanishes if and only if $\sum\beta_k=\frac{\pi}{2}$, which gives \eqref{Msym}.
\end{proof}

\begin{corollary}\label{cor:Oct}
For every $L\in\Oct$ and every $n\in D_4$, the isometry
$$
x\mapsto h'+L(x-h')+n
$$
maps $\M$ onto itself.
\end{corollary}

\begin{proof}
By \eqref{xiL}, $L$ acts on the frame coordinates $\xi_1,\dots,\xi_4$ by a signed permutation. Since $\gamma^2$ is an even function, $L$ therefore permutes the numbers $C_1,\dots,C_4$. The equation $\mathcal{E}_2(C)=1+\mathcal{E}_4(C)$ in \eqref{Msym} is symmetric in $C_1,\dots,C_4$, hence it is preserved. Finally, the translations by $D_4$ preserve $\M$ by Theorem~\ref{thA}(a).
\end{proof}

\begin{remark}[A second proof of minimality]\label{rem:secondproof}
Theorem~\ref{thm:symmetric} shows that, seen from a hole, $\M$ is a \emph{separable} hypersurface in the frame coordinates, with the same function in all four variables. Indeed, let
$$
\phi(t)=\arctan\gamma^2(t)-\frac{\pi}8 ,\qquad -\frac\pi8\le\phi\le\frac\pi8 .
$$
Then, by \eqref{Msym},
$$
\M-h'=\{y\in\R{4}:\ u(y)=0\},\qquad u(y)=\sum_{k=1}^4\phi(\xi_k).
$$
A direct computation using $\gamma'^2=\varpi^2(1-\gamma^4)$ gives
$$
\phi'^2=\frac{4\varpi^2\gamma^2(1-\gamma^4)}{(1+\gamma^4)^2}=\varpi^2\cos 4\phi,
\qquad
\phi''=-2\varpi^2\sin4\phi .
$$
The coordinates $\xi_k$ are orthonormal coordinates up to the common factor $\frac1{\sqrt2}$. For a separable function, $\Delta_1u$ is therefore proportional to
$$
\sum_{i\neq j}\phi''(\xi_i)\,\phi'(\xi_j)^2 .
$$
Put $a_i=4\phi(\xi_i)$. Then
$$
\sum_{i\neq j}\phi''(\xi_i)\,\phi'(\xi_j)^2=-2\varpi^4\sum_{i\ne j}\sin a_i\cos a_j=-2\varpi^4\sum_{i<j}\sin(a_i+a_j).
$$
On $\M-h'$ we have $a_1+a_2+a_3+a_4=0$, hence $\sin(a_i+a_4)=-\sin(a_j+a_k)$ whenever $\{i,j,k\}=\{1,2,3\}$, and the last sum vanishes. This gives an independent proof of the minimality of $\Mreg$.

The singular set is also easily recovered: $Du=0$ if and only if $\cos4\phi(\xi_k)=0$ for all $k$, i.e. $\phi(\xi_k)=\pm\frac\pi8$, and together with $\sum_k\phi(\xi_k)=0$ this forces two plus signs and two minus signs.

Note that $\phi$ satisfies the autonomous equation $\phi'^2=\varpi^2\cos4\phi$, of trigonometric type, while the function $\psi$ in the first separable form \eqref{separable} of $\M$ satisfies an equation of exponential type (\S\ref{ssec:separable}).
\end{remark}

\subsection{Three-dimensional sections}\label{ssec:sections}
Theorem~\ref{thm:symmetric} gives a convenient way to look at $\M$. Let us cut $\M$ by the hyperplanes orthogonal to $f_4$,
$$
H_c=\{x\in\R{4}:\ \xi_4(x-h')=c\},\qquad c\in\R{},
$$
and use $(\xi_1,\xi_2,\xi_3)$ as coordinates in $H_c$. Since $\xi_k=\frac{1}{\sqrt2}\scal{y}{f_k/\sqrt2}$ and the vectors $f_k/\sqrt2$ are orthonormal, these coordinates are Cartesian coordinates up to the common factor $\frac1{\sqrt2}$. By \eqref{Msym}, the section $\M\cap H_c$ is, in these coordinates, the surface
\begin{equation}\label{section}
\Sigma_c=\Bigl\{\xi\in\R{3}:\ \sum_{k=1}^3\arctan\gamma^2(\xi_k)=\frac{\pi}{2}-\arctan\gamma^2(c)\Bigr\}.
\end{equation}
Equivalently, in the notation of Remark~\ref{rem:secondproof},
$$
\Sigma_c=\{\xi\in\R{3}:\ v(\xi)=-\phi(c)\},\qquad v(\xi)=\phi(\xi_1)+\phi(\xi_2)+\phi(\xi_3).
$$
Three of these surfaces are shown in Figure~\ref{fig:sections}.

\begin{figure}[ht]
\centering
\begin{tabular}{ccc}
\includegraphics[width=0.31\textwidth]{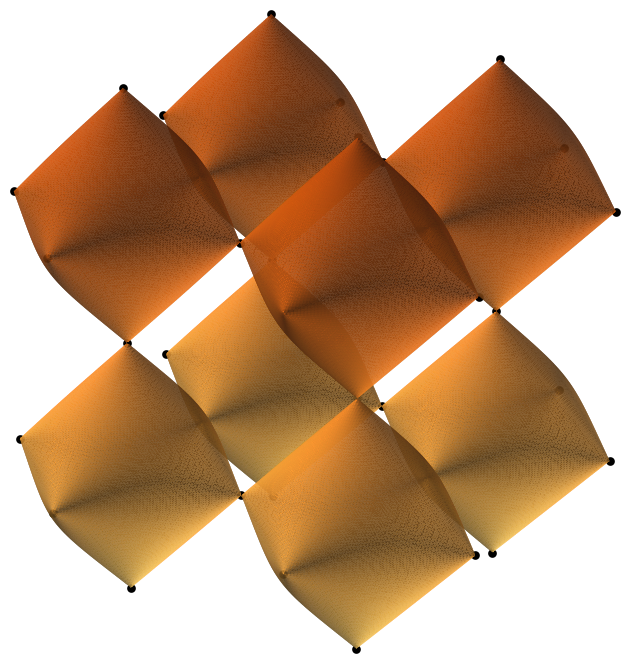} &
\includegraphics[width=0.31\textwidth]{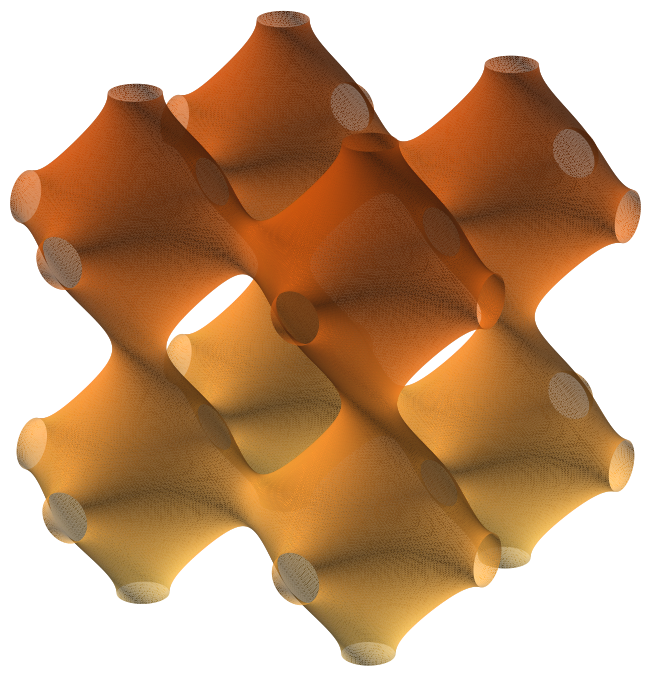} &
\includegraphics[width=0.31\textwidth]{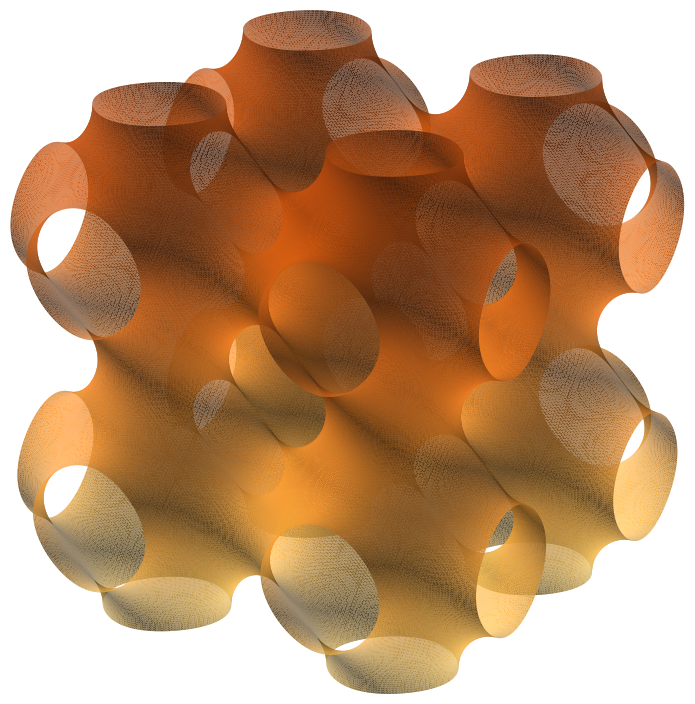}\\
(a) $c=\frac12$ & (b) $c=0.4$ & (c) $c=\frac14$
\end{tabular}
\caption{The sections $\Sigma_c$ of $\M$ by the hyperplanes $\xi_4=c$ orthogonal to $f_4$, shown in $2\times2\times2$ period cells of the lattice $\Z^3$. (a) For $c=\frac12$ the section is a cubic packing of closed ``pillows'' which touch each other only at the black points; these are singular points of $\M$, and near each of them the section is a double cone. (b) For $\frac14<c<\frac12$ the contact points open into necks. (c) For $c=\frac14$ the section is congruent to itself under the translation by $(\frac12,\frac12,\frac12)$ combined with an interchange of its two sides. The surfaces $\Sigma_c$ are \emph{not} minimal surfaces in $\R{3}$; see Remark~\ref{rem:notminimal}.}\label{fig:sections}
\end{figure}

\begin{proposition}\label{pro:sections}
In the above notation:
\begin{enumerate}
\item[(i)] Each $\Sigma_c$ is invariant under the translations by $\Z^3$ and under all signed permutations of $\xi_1,\xi_2,\xi_3$, i.e. under the full symmetry group of the cubic lattice $\Z^3$.
\item[(ii)] One has $\Sigma_{c+1}=\Sigma_{-c}=\Sigma_c$ and
$$
\Sigma_{c+\frac12}=\Sigma_c+(\tfrac12,\tfrac12,\tfrac12).
$$
Hence the whole family is determined by the sections with $0\le c\le\frac14$, and $\Sigma_{1/4}$ is invariant under the translation by $(\frac12,\frac12,\frac12)$.
\item[(iii)] If $c\notin\frac12\Z$, then $\Sigma_c$ is a smooth real analytic surface.
\item[(iv)] The section $\Sigma_{1/2}$ is singular exactly at the points congruent modulo $\Z^3$ to $(\frac12,0,0)$, $(0,\frac12,0)$ or $(0,0,\frac12)$, i.e. at the centres of the faces of the unit cubes $m+[-\frac12,\frac12]^3$, $m\in\Z^3$. Near each of these points $\Sigma_{1/2}$ is, up to a local analytic diffeomorphism tangent to the identity, a circular double cone.
\item[(v)] $\Sigma_{1/2}\cap[-\frac12,\frac12]^3$ is the boundary of a compact star-shaped region containing the origin, homeomorphic to a closed ball, and it meets the boundary of the cube only at the six face centres.
\end{enumerate}
\end{proposition}

\begin{proof}
(i) The function $\gamma^2$ is even and has period $1$, and the left-hand side of \eqref{section} is symmetric in $\xi_1,\xi_2,\xi_3$.

(ii) The first two identities follow in the same way. For the third one we use \eqref{sigmagamma}: since $\gamma^2(t+\frac12)=\sigma^2(t)=(1-\gamma^2(t))/(1+\gamma^2(t))$, we have
$$
\arctan\gamma^2(t+\tfrac12)=\frac{\pi}{4}-\arctan\gamma^2(t).
$$
Replacing $\xi$ by $\xi+(\frac12,\frac12,\frac12)$ therefore replaces the left-hand side $L(\xi)$ of \eqref{section} by $\frac{3\pi}4-L(\xi)$, while replacing $c$ by $c+\frac12$ replaces the right-hand side $\frac\pi2-\arctan\gamma^2(c)$ by $\frac{\pi}4+\arctan\gamma^2(c)$. The two equations are equivalent, which gives the identity.

(iii) The surface $\Sigma_c$ is the level set $v=-\phi(c)$, and $\nabla v=0$ only if $\phi'(\xi_k)=0$, i.e. $\phi(\xi_k)=\pm\frac{\pi}{8}$, for $k=1,2,3$. At such points $v\in\{\pm\frac\pi8,\pm\frac{3\pi}8\}$. On the other hand, $-\phi(c)\in[-\frac\pi8,\frac\pi8]$, with equality at the end points only if $\gamma^2(c)\in\{0,1\}$, i.e. $c\in\frac12\Z$. Hence for $c\notin\frac12\Z$ the level $-\phi(c)$ is a regular value of $v$.

(iv) For $c=\frac12$ we have $\gamma^2(c)=0$ and $-\phi(c)=\frac\pi8$. By the proof of (iii), the singular points are the points where two of the numbers $\phi(\xi_k)$ equal $\frac\pi8$ and one equals $-\frac\pi8$, i.e. two of the $\xi_k$ are integers and one lies in $\frac12+\Z$. At such a point the Hessian of $v$ is diagonal with entries $\phi''(\xi_k)=-2\varpi^2\sin4\phi(\xi_k)$, that is, $-2\varpi^2$, $-2\varpi^2$ and $2\varpi^2$ (in some order). The Morse lemma (Lemma~\ref{lem:Morse}) gives the double cone.

(v) For $\xi\in[-\frac12,\frac12]^3$ and $0\le t\le1$, each function $t\mapsto\gamma^2(t\xi_k)$ is non-increasing, and strictly decreasing if $\xi_k\neq0$, because $\gamma^2$ decreases strictly on $[0,\frac12]$. Hence the left-hand side of \eqref{section} decreases strictly along every ray from the origin inside the cube. It equals $\frac{3\pi}4>\frac\pi2$ at the origin. On a face of the cube, say $\xi_1=\frac12$, we have $\gamma^2(\xi_1)=0$, and the left-hand side is at most $\frac\pi2$, with equality only if $\xi_2=\xi_3=0$. Therefore every ray from the origin to a point of the boundary of the cube, other than a face centre, meets $\Sigma_{1/2}$ exactly once, and the region $\{L\ge\frac\pi2\}\cap[-\frac12,\frac12]^3$ is star-shaped with respect to the origin. Its boundary is $\Sigma_{1/2}\cap[-\frac12,\frac12]^3$, which is a topological sphere.
\end{proof}

\begin{remark}[Interpretation]\label{rem:interpretation}
Proposition~\ref{pro:sections} and Figure~\ref{fig:sections} give a concrete picture of how the Clifford-cone singularities organize the topology of $\M$. As $c$ moves from $\frac12$ to $\frac14$, the section $\Sigma_c$ changes from a packing of pillows centred at the points of $\Z^3$, which touch each other at conical points (Figure~\ref{fig:sections}(a)), to a connected surface with necks (Figure~\ref{fig:sections}(b),(c)). When $c$ continues to $0$, the same happens in reverse, but now the pillows are centred at the points of $(\frac12,\frac12,\frac12)+\Z^3$, by Proposition~\ref{pro:sections}(ii). In other words, each conical point of $\Sigma_{1/2}$ is the place where a neck of the nearby sections is pinched off; this is the three-dimensional shadow of the fact that the link of a singular point of $\M$ is a torus rather than a sphere. The sections through a singular point by hyperplanes of other directions look similar: the tangent cone of $\M$ is cut by a hyperplane into a quadratic cone in $\R{3}$.
\end{remark}

\begin{remark}\label{rem:notminimal}
The surfaces in Figure~\ref{fig:sections} resemble some classical triply periodic minimal surfaces in $\R{3}$ with cubic symmetry, notably Schwarz's P-surface \cite{Meeks90}. We stress, however, that $\Sigma_c$ are sections of a minimal \emph{hypersurface} of $\R{4}$ by hyperplanes, and such sections are, in general, not minimal. This is indeed the case here. Since $\Sigma_c$ is the level set of the separable function $v$ in three variables, the computation of Remark~\ref{rem:secondproof} shows that $\Delta_1v$ is proportional to
$$
\sum_{1\le i<j\le3}\sin(a_i+a_j)=\sum_{k=1}^3\sin(A-a_k),\qquad a_k=4\phi(\xi_k),\quad A=a_1+a_2+a_3=-4\phi(c),
$$
and this expression does not vanish identically on $\Sigma_c$. For example, take $c=\frac14$. Since
$$
\gamma^2(\tfrac14)=\sinl^2(\tfrac{\varpi}{4})=\sqrt2-1=\tan\tfrac\pi8,
$$
we have $\phi(\frac14)=0$, so that $A=0$ and
$$
\Sigma_{1/4}=\{\xi:\ \phi(\xi_1)+\phi(\xi_2)+\phi(\xi_3)=0\}.
$$
Choose $t$ with $\phi(t)=-\frac{\pi}{16}$ (this is possible because $\phi$ takes all values in $[-\frac\pi8,\frac\pi8]$). Since $\phi(0)=\frac\pi8$, the point $\xi=(t,t,0)$ lies on $\Sigma_{1/4}$, and there $a_1=a_2=-\frac\pi4$, $a_3=\frac\pi2$, so that
$$
\sum_{k=1}^3\sin(A-a_k)=-\sum_{k=1}^3\sin a_k=\sqrt2-1\neq0 .
$$
Since $\Sigma_{1/4}$ is smooth (Proposition~\ref{pro:sections}(iii)), its mean curvature does not vanish at this point. Thus the resemblance is only qualitative: what Figure~\ref{fig:sections} shows is how a single minimal hypersurface with Clifford-cone singularities looks in three-dimensional slices, not a family of minimal surfaces.
\end{remark}

\subsection{A theta series}
We use the Jacobi theta functions with nome $q=e^{\pi\mathrm{i}\tau}$, $\operatorname{Im}\tau>0$, in the convention of \cite[\S21.11]{WhittakerWatson} and \cite[Ch.~1]{Lawden}:
\begin{align*}
\theta_1(z|\tau)&=-\mathrm{i}\sum_{a\in\frac12+\Z}(-1)^{a-\frac12}q^{a^2}e^{2\mathrm{i}az},\\
\theta_3(z|\tau)&=\sum_{b\in\Z}q^{b^2}e^{2\mathrm{i}bz}.
\end{align*}
Equivalently,
\begin{align*}
\theta_1(z|\tau)&=2\sum_{n\ge0}(-1)^nq^{(n+\frac12)^2}\sin(2n+1)z\\
\theta_3(z|\tau)&=1+2\sum_{n\ge1}q^{n^2}\cos2nz.
\end{align*}
In this convention
$$
\theta_2(z)=\theta_1(z+\frac\pi2), \qquad \theta_4(z)=\theta_3(z+\frac\pi2),
$$
and Jacobi's imaginary transformation reads
$$
\theta_2(0|-1/\tau)=(-\mathrm{i}\tau)^{1/2}\theta_4(0|\tau)$$
(\cite[\S21.51]{WhittakerWatson}). Recall from the introduction the theta series
$$
\vartheta_\Lambda(y;\tau)=\sum_{v\in\Lambda}e^{\pi\mathrm{i}\tau|v|^2+2\pi\mathrm{i}\scal{v}{y}},\qquad \operatorname{Im}\tau>0 ,
$$
and the translate of the dual lattice $D_4^*$
$$
\Lambda=\tfrac12 f_1+D_4^*=\Bigl\{\tfrac12\sum_{k=1}^4\eta_kf_k:\ \eta\in\Z^4,\ \eta_1+\eta_2+\eta_3+\eta_4\ \text{odd}\Bigr\}.
$$
In the quaternionic notation of Remark~\ref{rem:Hurwitz}, $\Lambda=\frac12(1+\mathrm{j})+\mathcal{O}$ is a translate of the Hurwitz order.

\begin{theorem}\label{thm:theta}
With the above notation,
\begin{enumerate}
\item[(i)] For all real $t$,
$$
\sigma(t)=\frac{\theta_1(\pi t|\mathrm{i})}{\theta_3(\pi t|\mathrm{i})}.
$$
\item[(ii)] Let
$$
g(t)=\theta_3\Bigl(\pi t\,\Big|\,\frac{1+\mathrm{i}}{2}\Bigr),\qquad
\Theta(x)=S(x)\prod_{k=1}^4\theta_3(\pi x_k|\mathrm{i}).
$$
Then
$$
\Theta(h'+y)=\operatorname{Im}\prod_{k=1}^4 g(\xi_k)=-\mathrm{i}\,\vartheta_\Lambda(y;1+\mathrm{i}).
$$
\item[(iii)] Consequently,
$$
\M=h'+\{y\in\R{4}:\ \vartheta_{\Lambda}(y;1+\mathrm{i})=0\}.
$$
\item[(iv)] The group of orthogonal transformations $L$ with $L(\Lambda)=\Lambda$ is exactly $\Oct$. In particular,
$$
\Theta\bigl(h'+L(x-h')+n\bigr)=(-1)^{n_1+n_3}\,\Theta(x)\qquad\text{for all}\ L\in\Oct,\ n\in D_4 .
$$
\end{enumerate}
\end{theorem}

\begin{proof}
(i) Let $\tau=\mathrm{i}$, so that $q=e^{-\pi}$, and put
$$
f(t)=\frac{\theta_1(\pi t)}{\theta_3(\pi t)} .
$$
The classical quasi-periodicity of $\theta_1$ and $\theta_3$ under $z\mapsto z+\pi$ and $z\mapsto z+\pi\tau$ gives
$$
f(t+1)=f(t+\mathrm{i})=-f(t).
$$
Thus $f$ is an elliptic function whose period lattice is generated by $1+\mathrm{i}$ and $1-\mathrm{i}$. Its zeros are the simple zeros of $\theta_1(\pi t)$, i.e. the points of $\Z[\mathrm{i}]$, and its poles are the simple zeros of $\theta_3(\pi t)$, i.e. the points of $\tfrac12(1+\mathrm{i})+\Z[\mathrm{i}]$.

The function $\sigma(t)=\sinl(\varpi t)$ is elliptic with the same periods, the same simple zeros and the same simple poles \cite{LangerSinger,Markush} (cf. Remark~\ref{rem:wp}). Hence $\sigma/f$ is an elliptic function without zeros and poles, i.e. a constant. To find the constant, recall that $\theta_1(z+\frac\pi2)=\theta_2(z)$ and $\theta_3(z+\frac\pi2)=\theta_4(z)$. Therefore
$$
f(\tfrac12)=\frac{\theta_2(0|\mathrm{i})}{\theta_4(0|\mathrm{i})}=1,
$$
where the last equality follows from Jacobi's imaginary transformation $\theta_2(0|-1/\tau)=(-\mathrm{i}\tau)^{1/2}\theta_4(0|\tau)$ at the fixed point $\tau=\mathrm{i}$. Since also $\sigma(\frac12)=1$, the constant equals $1$.

(ii) By (i), $\Theta=P_1-P_2$, where (all theta functions with $\tau=\mathrm{i}$)
$$
P_1(x)=\theta_1(\pi x_1)\theta_1(\pi x_3)\theta_3(\pi x_2)\theta_3(\pi x_4),\qquad
P_2(x)=\theta_1(\pi x_2)\theta_1(\pi x_4)\theta_3(\pi x_1)\theta_3(\pi x_3).
$$
Multiplying the series, we obtain
$$
P_1(x)=\sum_{v\in\Lambda_1}(-1)^{v_1+v_3}q^{|v|^2}e^{2\pi\mathrm{i}\scal{v}{x}},\qquad
P_2(x)=\sum_{v\in\Lambda_2}(-1)^{v_2+v_4}q^{|v|^2}e^{2\pi\mathrm{i}\scal{v}{x}},
$$
where
$$
\Lambda_1=(\tfrac12,0,\tfrac12,0)+\Z^4,\qquad \Lambda_2=(0,\tfrac12,0,\tfrac12)+\Z^4 .
$$
Write $v=\frac12\sum_k\eta_kf_k$, i.e. $\eta_k=\scal{v}{f_k}$. One checks directly that $v\in\Lambda_1$ if and only if $\eta\in\Z^4$, $\eta_1+\eta_2$ is odd and $\eta_3+\eta_4$ is even, and similarly for $\Lambda_2$ with the two parities interchanged. Thus
$$
\Lambda=\Lambda_1\sqcup\Lambda_2,\qquad |v|^2=\tfrac12|\eta|^2,\qquad \scal{v}{y}=\sum_{k=1}^4\eta_k\xi_k .
$$

Now substitute $x=h'+y$. For $v\in\Lambda_1$ we have $e^{2\pi\mathrm{i}\scal{v}{h'}}=(-1)^{v_1+v_3}(-1)^{v_2+v_4}$, so the coefficient of $e^{2\pi\mathrm{i}\scal{v}{y}}$ in $P_1(h'+y)$ equals $(-1)^{v_2+v_4}q^{|v|^2}$. For $v\in\Lambda_2$ we have $e^{2\pi\mathrm{i}\scal{v}{h'}}=-(-1)^{v_1+v_3}(-1)^{v_2+v_4}$, so the coefficient in $-P_2(h'+y)$ equals $(-1)^{v_1+v_3}q^{|v|^2}$. In both cases the sign is
$$
(-1)^{|v|^2-\frac12}=(-1)^{(|\eta|^2-1)/2},
$$
because $(n+\frac12)^2\equiv\frac14$ and $m^2\equiv m$ modulo $2$ for all integers $n,m$. Since moreover $|\eta|^2\equiv\eta_1+\dots+\eta_4 \pmod 2$, we arrive at
\begin{equation}\label{thetaeta}
\Theta(h'+y)=\sum_{\eta\in\Z^4,\ |\eta|^2\ \mathrm{odd}}(-1)^{\frac{|\eta|^2-1}{2}}e^{-\frac{\pi}{2}|\eta|^2}e^{2\pi\mathrm{i}\scal{\eta}{\xi}} .
\end{equation}

On the other hand,
$$
\prod_{k=1}^4g(\xi_k)=\sum_{\eta\in\Z^4}\mathrm{i}^{|\eta|^2}e^{-\frac\pi2|\eta|^2}e^{2\pi\mathrm{i}\scal{\eta}{\xi}}
=\sum_{\eta\in\Z^4}\mathrm{i}^{|\eta|^2}e^{-\frac\pi2|\eta|^2}\cos\bigl(2\pi\scal{\eta}{\xi}\bigr),
$$
where in the second step we have grouped $\eta$ with $-\eta$. The imaginary part of this sum is the sum over odd $|\eta|^2$, where
$$\mathrm{i}^{|\eta|^2}=\mathrm{i}\,(-1)^{(|\eta|^2-1)/2};
$$
this is exactly \eqref{thetaeta}. Finally, for $v\in\Lambda$ we have $|v|^2\in\frac12+\Z$ and hence
$$
e^{\pi\mathrm{i}(1+\mathrm{i})|v|^2}=\mathrm{i}\,(-1)^{|v|^2-\frac12}q^{|v|^2},
$$
which gives the expression through $\vartheta_\Lambda$.

(iii) For real $z$ we have $$\theta_3(z|\mathrm{i})\ge1-2\sum_{b\ge1}e^{-\pi b^2}>0,$$ so (iii) follows from (ii).

(iv) The vectors of minimal norm in $\Lambda$ are those with $|\eta|^2=1$, i.e. the eight vectors $\pm\frac12f_k$. Hence every orthogonal $L$ with $L(\Lambda)=\Lambda$ permutes $\frac12\F$, that is, $L\in\Oct$. Conversely, every $L\in\Oct$ acts on $\eta$ by a signed permutation, which preserves the parity of $\eta_1+\dots+\eta_4$; hence $L(\Lambda)=\Lambda$. Since $\vartheta_\Lambda(Ly;\tau)=\vartheta_{L^{-1}\Lambda}(y;\tau)$, (ii) shows that $\Theta(h'+L(x-h'))=\Theta(x)$. The behaviour under translations follows from $\theta_1(z+\pi)=-\theta_1(z)$ and $\theta_3(z+\pi)=\theta_3(z)$.
\end{proof}

\begin{proof}[Proof of Theorem~\ref{thD}]
The first statement is Theorem~\ref{thm:symmetric}. By Theorem~\ref{thm:theta}(ii), $-\mathrm{i}\,\vartheta_\Lambda(y;1+\mathrm{i})=\Theta(h'+y)=S(h'+y)\prod_{k=1}^4\theta_3(\pi x_k|\mathrm{i})$ with $x=h'+y$, and the product is a positive real analytic function by the proof of Theorem~\ref{thm:theta}(iii). The last statement is Theorem~\ref{thm:theta}(iv).
\end{proof}

\begin{remark}\label{rem:thetaA}
Theorem~\ref{thm:theta} explains the hidden symmetry of Theorem~\ref{thB}(c). Since $A\in\Oct$ and $Ah'=h'$, part (iv) gives $\Theta(Ax)=\Theta(x)$. Therefore the factor $Q$ in the identity $S(x)=Q(x)S(Ax)$ is
$$
Q(x)=\prod_{k=1}^4\frac{\theta_3(\pi (Ax)_k|\mathrm{i})}{\theta_3(\pi x_k|\mathrm{i})} .
$$
\end{remark}

\begin{remark}
Comparing Theorem~\ref{thm:theta}(ii) with Theorem~\ref{thm:symmetric}, one may ask for the phase of $g(t)=\theta_3(\pi t|\frac{1+\mathrm{i}}2)$. It is the function $\phi$ of Remark~\ref{rem:secondproof}:
\begin{equation}\label{phase}
\arg g(t)=\arctan\gamma^2(t)-\frac{\pi}{8}.
\end{equation}
We do not need this identity, so we only sketch the proof. Splitting the series of $g$ into the terms with even and with odd $b$, we find
$$
g(t)=\theta_3(2\pi t\,|\,2\mathrm{i})+\mathrm{i}\,\theta_2(2\pi t\,|\,2\mathrm{i}),
$$
and $\operatorname{Re}g=\theta_3(2\pi t|2\mathrm{i})>0$. Multiplying the series and substituting $m=a+b$, $n=a-b$, one obtains the classical duplication formulas
\begin{align*}
\theta_2(z|\tau)^2&=\theta_2(0|2\tau)\,\theta_3(2z|2\tau)+\theta_3(0|2\tau)\,\theta_2(2z|2\tau),\\
\theta_4(z|\tau)^2&=\theta_3(0|2\tau)\,\theta_3(2z|2\tau)-\theta_2(0|2\tau)\,\theta_2(2z|2\tau).
\end{align*}
At $z=0$, $\tau=\mathrm{i}$ they give, together with $\theta_2(0|\mathrm{i})=\theta_4(0|\mathrm{i})$, the value $\varrho:=\theta_2(0|2\mathrm{i})/\theta_3(0|2\mathrm{i})=\sqrt2-1=\tan\frac\pi8$. Since $\gamma(t)=\sigma(t+\frac12)=\theta_2(\pi t|\mathrm{i})/\theta_4(\pi t|\mathrm{i})$ by Theorem~\ref{thm:theta}(i), we obtain, with $T=\theta_2(2\pi t|2\mathrm{i})/\theta_3(2\pi t|2\mathrm{i})=\tan\arg g(t)$,
$$
\gamma^2(t)=\frac{T+\varrho}{1-\varrho T}=\tan\Bigl(\arg g(t)+\frac\pi8\Bigr),
$$
which is \eqref{phase}.
\end{remark}

\subsection{The quaternionic picture}\label{ssec:quaternionic}
We identify $\R{4}$ with the quaternions $\Ham$ as in Remark~\ref{rem:Hurwitz}, so that $x=x_1+x_2\mathrm{i}+x_3\mathrm{j}+x_4\mathrm{k}$, and we denote by $\mathcal{O}$ the Hurwitz order. Let
$$
Q_8=\{\pm1,\pm\mathrm{i},\pm\mathrm{j},\pm\mathrm{k}\}
$$
be the quaternion group, and let $2O$ be the binary octahedral group, i.e. the group of the $48$ unit quaternions consisting of the $24$ units of $\mathcal{O}$ and the $24$ elements $(\pm a\pm b)/\sqrt2$, $a\neq b\in\{1,\mathrm{i},\mathrm{j},\mathrm{k}\}$. It is well known that $2O$ is the normalizer of $Q_8$ in the group of unit quaternions, and that $2O/Q_8\cong S_3$ acts faithfully, by conjugation, on the three axes $\{\pm\mathrm{i}\}$, $\{\pm\mathrm{j}\}$, $\{\pm\mathrm{k}\}$ \cite{ConwaySmith}.

The descriptions of Sections~\ref{sec:lattices}--\ref{sec:hole} take a particularly symmetric form in the coordinate
$$
w=\frac{\bar u\,(x-h')}{\sqrt2},\qquad u=\frac{1+\mathrm{j}}{\sqrt2}\in 2O .
$$

\begin{proposition}\label{pro:quaternionic}
In the coordinate $w=w_0+w_1\mathrm{i}+w_2\mathrm{j}+w_3\mathrm{k}$ the following holds.
\begin{enumerate}
\item[(i)] \textup{(The equation.)} $x\in\M$ if and only if
$$
\sum_{\nu=0}^{3}\arctan\gamma^2(w_\nu)=\frac{\pi}{2}.
$$
\item[(ii)] \textup{(Holes and singular points.)} The holes $h'+D_4$ correspond to the Hurwitz order $\mathcal{O}$, and the singular points correspond to the three cosets
$$
\tfrac12(1+q)+\mathcal{O},\qquad q\in\{\mathrm{i},\mathrm{j},\mathrm{k}\},
$$
whose union is $\frac12(1+\mathrm{i})\mathcal{O}\setminus\mathcal{O}$. The origin $x=0$ and the points of $D_4$ correspond to $q=\mathrm{i}$, the points of $h+D_4$ to $q=\mathrm{j}$, and the points of $e_1+D_4$ to $q=\mathrm{k}$. We say that a singular point is \emph{of type $q$}.
\item[(iii)] \textup{(Tangent cones.)} At a singular point of type $q$, the tangent cone of $\M$ is the Clifford cone of the complex structure $q$: writing $\Ham=\C_q\oplus\C_q^\perp$ with $\C_q=\R{}+\R{}q$,
$$
C_q=\{w:\ |w_{\C_q}|=|w_{\C_q^\perp}|\};\qquad\text{for instance,}\quad C_{\mathrm{i}}=\{w_0^2+w_1^2=w_2^2+w_3^2\}.
$$
\item[(iv)] \textup{(Symmetries.)} In the coordinate $w$, the isometry group of $\M$ is $\mathcal{O}\rtimes G$, where $\mathcal{O}$ acts by translations and
$$
G=\{w\mapsto a\,w\,\bar b,\ \ w\mapsto a\,\bar w\,\bar b:\ \ a,b\in 2O,\ a\bar b\in Q_8\}
$$
is the stabilizer of the frame $Q_8$, a group of order $384$. The group $G$ permutes the three types of singular points through $2O\to 2O/Q_8\cong S_3$: the map $w\mapsto a w\bar b$ sends type $q$ to the type of $aq\bar a$.
\item[(v)] \textup{(The hidden symmetry.)} The involution $A$ of Theorem~\ref{thB} is the reflection
$$
w\mapsto -n\,\bar w\,n,\qquad n=\frac{\mathrm{j}+\mathrm{k}}{\sqrt2}\in 2O,
$$
which fixes $1$ and $\mathrm{i}$ and interchanges $\mathrm{j}$ and $-\mathrm{k}$. Thus $A$ fixes the type $\mathrm{i}$ of the origin and interchanges the types $\mathrm{j}$ and $\mathrm{k}$.
\end{enumerate}
\end{proposition}

\begin{proof}
Write $y=x-h'=\sqrt2\,u w$. Since left multiplication by the unit $u$ is an orthogonal map with adjoint $w\mapsto\bar u w$, we have $\xi_k=\frac12\scal{y}{f_k}=\frac{1}{\sqrt2}\scal{w}{\bar uf_k}$. A direct computation gives
$$
\bar uf_1=\sqrt2,\qquad \bar uf_2=-\sqrt2\,\mathrm{j},\qquad \bar uf_3=\sqrt2\,\mathrm{k},\qquad \bar uf_4=\sqrt2\,\mathrm{i},
$$
hence
\begin{equation}\label{xiw}
(\xi_1,\xi_2,\xi_3,\xi_4)=(w_0,-w_2,w_3,w_1).
\end{equation}
In particular, the frame $\frac1{\sqrt2}\F$ is mapped onto $Q_8$. Now (i) follows from Theorem~\ref{thm:symmetric}, because $\gamma^2$ is even.

(ii) Recall that $D_4^*=\mathcal{O}$ and $D_4=(1+\mathrm{i})\mathcal{O}=(1+\mathrm{j})\mathcal{O}$, a two-sided ideal of the Hurwitz order $\mathcal{O}$; the two generators differ by a unit of $\mathcal{O}$. Since $h'\in\mathcal{O}$ and $\bar u/\sqrt2=(1+\mathrm{j})^{-1}$, the holes $h'+D_4$ correspond to $(1+\mathrm{j})^{-1}D_4=\mathcal{O}$, and the singular set $\Sing=\mathcal{O}\setminus(h'+D_4)$ corresponds to
$$
(1+\mathrm{j})^{-1}\bigl(\mathcal{O}\setminus(1+\mathrm{j})\mathcal{O}\bigr)=(1+\mathrm{j})^{-1}\mathcal{O}\setminus\mathcal{O}=\tfrac12(1+\mathrm{i})\mathcal{O}\setminus\mathcal{O}.
$$
The lattice $\frac12(1+\mathrm{i})\mathcal{O}$ contains $\mathcal{O}$ with index $4$, and the three nontrivial cosets are represented by $\frac12(1+q)$, $q\in\{\mathrm{i},\mathrm{j},\mathrm{k}\}$. For example, the origin $x=0$ corresponds to $w=-\frac12(1+\mathrm{i})$; the remaining assignments are checked in the same way.

(iii) By Lemma~\ref{lem:expansions}(i) and \eqref{xiw}, the tangent cone at the origin is
$$
x_1x_3=x_2x_4\iff \xi_1^2-\xi_2^2-\xi_3^2+\xi_4^2=0\iff w_0^2+w_1^2=w_2^2+w_3^2,
$$
which is $C_{\mathrm{i}}$. The cases $q=\mathrm{j},\mathrm{k}$ follow in the same way from Lemma~\ref{lem:expansions}, or from the symmetry in (iv).

(iv) By Theorem~\ref{thB}(b) and \eqref{xiw}, in the coordinate $w$ the isometry group of $\M$ consists of the translations by the image of $D_4$, which is $\mathcal{O}$, composed with the orthogonal maps preserving $Q_8$. A rotation $w\mapsto a w\bar b$ preserves $Q_8$ if and only if $c=a\bar b\in Q_8$ and $w\mapsto awa^{-1}=(aw\bar b)\bar c$ preserves $Q_8$, i.e. if and only if $a\in 2O$ and $c\in Q_8$; then also $b=\bar c a\in2O$. The orientation reversing maps are obtained by composing with $w\mapsto\bar w$, which preserves $Q_8$. This gives the description of $G$, and $|G|=2\cdot 48\cdot 8/2=384$. Finally, $a(1+q)\bar b=(1+aq\bar a)\,c$, and for $c\in Q_8$ the difference $\frac12(1+q')(c-1)$ lies in $\mathcal{O}$, because $c-1\in(1+q')\mathcal{O}$ and $(1+q')^2\mathcal{O}=2\mathcal{O}$. Moreover, $a\mathcal{O}\bar b=(a\mathcal{O}a^{-1})\,c=\mathcal{O}$, because conjugation by an element of $2O$ permutes the $24$ units of $\mathcal{O}$, which span $\mathcal{O}$. Hence $w\mapsto aw\bar b$ maps the coset $\frac12(1+q)+\mathcal{O}$ onto $\frac12(1+aq\bar a)+\mathcal{O}$; note that $\frac12(1+q)$ and $\frac12(1-q)$ lie in the same coset. The conjugation $w\mapsto\bar w$ fixes each coset.

(v) The map $A$ fixes $h'$, $f_1$ and $f_4$ and interchanges $f_2$ and $f_3$. By \eqref{xiw} it fixes $1$ and $\mathrm{i}$ and maps $-\mathrm{j}$ to $\mathrm{k}$, i.e. it interchanges $\mathrm{j}$ and $-\mathrm{k}$. This is the reflection in the hyperplane orthogonal to $n=(\mathrm{j}+\mathrm{k})/\sqrt2$, which is $w\mapsto w-2\scal{w}{n}n=-n\bar wn$.
\end{proof}

\begin{remark}\label{rem:quaternionic}
Proposition~\ref{pro:quaternionic} gives an intrinsic meaning to the ``two kinds'' of singular points in Theorem~\ref{thB}: there are in fact three equivalent types, labelled by the imaginary units $\mathrm{i},\mathrm{j},\mathrm{k}$, and the tangent cone at a point of type $q$ is the Clifford cone of the complex structure $q$. The isometry group acts on the types through the octahedral group, i.e. through $S_3$, and the hidden symmetry $A$ is simply the transposition of the types $\mathrm{j}$ and $\mathrm{k}$. In the same coordinate, the function $\Theta$ of Theorem~\ref{thm:theta} takes the form
\begin{equation}\label{thetaLipschitz}
\Theta\bigl(h'+\sqrt2\,u\,w\bigr)=\sum_{\eta\in\mathbb{L}}\chi_{-4}\bigl(N(\eta)\bigr)\,e^{-\frac{\pi}{2}N(\eta)}\cos\bigl(2\pi\operatorname{Re}(\eta\bar w)\bigr),
\end{equation}
where $\mathbb{L}=\Z+\Z\mathrm{i}+\Z\mathrm{j}+\Z\mathrm{k}$ is the \emph{Lipschitz order}, $N(\eta)=\eta\bar\eta=|\eta|^2$ is the norm, and $\chi_{-4}$ is the nontrivial Dirichlet character modulo $4$, i.e. $\chi_{-4}(n)=(-1)^{(n-1)/2}$ for odd $n$ and $\chi_{-4}(n)=0$ for even $n$. Indeed, by \eqref{xiw} the frame coordinates $\xi$ of $y=\sqrt2\,uw$ are obtained from the coordinates of $w$ by a signed permutation $P$, so that $\scal{\eta}{\xi}=\scal{P^{T}\eta}{w}$. Substituting this into \eqref{thetaeta}, replacing $\eta$ by $P^T\eta$ (which preserves $\Z^4$ and $|\eta|^2$), grouping $\eta$ with $-\eta$, and noting that $\scal{\eta}{w}=\operatorname{Re}(\eta\bar w)$, we obtain \eqref{thetaLipschitz}. We are not aware of an occurrence of this particular theta series in the literature; note only that its value at $w=0$ is $\sum_{n\ge1}\chi_{-4}(n)\,r_4(n)\,e^{-\pi n/2}$, where $r_4(n)$ is the number of representations of $n$ as a sum of four squares. We do not know whether this quaternionic structure reflects a deeper principle; see Question~\ref{q:quaternionic}.
\end{remark}

\section{Symmetries: proof of Theorem~\ref{thB}}\label{sec:sym}

\subsection{Translations}
We begin by determining the translational symmetries of $\M$.

\begin{proposition}\label{pro:transl}
A translation $x\mapsto x+v$ maps $\M$ onto itself if and only if $v\in D_4$. Moreover, $S(x+n)=(-1)^{n_1+n_3}S(x)$ for all $n\in D_4$.
\end{proposition}

\begin{proof}
The second statement is \eqref{transl}, and it shows that $D_4$ preserves $\M$. Conversely, suppose that $\M+v=\M$. The plane $\{x_1=x_2=0\}$ lies in $\M$, hence $S(v_1,v_2,v_3+t_3,v_4+t_4)=0$ for all $t_3,t_4\in\R{}$, i.e.
$$
\sigma(v_1)\sigma(v_3+t_3)=\sigma(v_2)\sigma(v_4+t_4)\qquad\text{for all}\ t_3,t_4 .
$$
Since $\sigma$ is not constant, this forces $\sigma(v_1)=\sigma(v_2)=0$, i.e. $v_1,v_2\in\Z$. Using the plane $\{x_3=x_4=0\}\subset\M$ in the same way we obtain $v_3,v_4\in\Z$. Thus $v\in\Z^4$, and since $h\in\M$, \eqref{transl} gives $0=S(h+v)=(-1)^{v_1+v_3}-(-1)^{v_2+v_4}$, i.e. $v\in D_4$.
\end{proof}

\subsection{The addition theorem and the stabilizer of the origin}
Let $\xi_i$ be the frame coordinates \eqref{xi}, and put
\begin{equation}\label{alpha}
\alpha_i(x)=\sigma^2(\xi_i(x))\in[0,1],\qquad i=1,\dots,4,
\end{equation}
and
$$
\varphi(a,b)=\frac{a-b}{1+ab},\qquad a,b\ge0 .
$$
Since $x_1=\xi_1+\xi_2$ and $x_3=\xi_1-\xi_2$, the addition theorem \eqref{additionN} gives $\sigma(x_1)\sigma(x_3)=\varphi(\alpha_1,\alpha_2)$, and similarly $\sigma(x_2)\sigma(x_4)=\varphi(\alpha_3,\alpha_4)$. Hence
\begin{equation}\label{SE}
S(x)=E(\alpha_1,\alpha_2,\alpha_3,\alpha_4),\qquad E(a_1,a_2,a_3,a_4):=\varphi(a_1,a_2)-\varphi(a_3,a_4).
\end{equation}
The following elementary identity is the source of the hidden symmetry of $\M$.

\begin{lemma}\label{lem:identity}
For all $a,b,c,d\ge0$,
$$
\varphi(a,b)-\varphi(c,d)=\frac{(1+ac)(1+bd)}{(1+ab)(1+cd)}\,\bigl(\varphi(a,c)-\varphi(b,d)\bigr).
$$
\end{lemma}

\begin{proof}
After multiplication by the denominators, both sides become
$a-b-c+d+acd-bcd-abc+abd$.
\end{proof}

Let $\Pi\subset S_4$ be the group of permutations of $\{1,2,3,4\}$ which preserve the partition $\{\{1,4\},\{2,3\}\}$; it has order $8$ and is generated by the permutations $(23)$, $(12)(34)$ and $(13)(24)$.
Using the notation
$$
a^\pi=(a_1,\dots,a_4)^\pi=(a_{\pi(1)},\dots,a_{\pi(4)}),\qquad \pi\in S_4,
$$
we have the following.

\begin{corollary}\label{cor:E}
For every $\pi\in\Pi$ there is a continuous function $q_\pi>0$ on $[0,1]^4$, with values in $[\frac14,4]$, such that $E(a^\pi)=\pm q_\pi(a)E(a)$, the sign depending only on $\pi$.
\end{corollary}

\begin{proof}
Directly from the definition, $E(a^{(12)(34)})=-E(a)$ and $E(a^{(13)(24)})=-E(a)$. By Lemma~\ref{lem:identity},
$E(a)=q(a)\,E(a^{(23)})$ with $q(a)=\frac{(1+a_1a_3)(1+a_2a_4)}{(1+a_1a_2)(1+a_3a_4)}$, and $\frac14\le q\le4$ on $[0,1]^4$. Since the three permutations generate $\Pi$ and $E(a^{\pi\pi'})$ is obtained by applying the rule for $\pi$ and then for $\pi'$, the statement follows; the bounds for $q_\pi$ follow by inspecting the products (note that $q(a^{(12)(34)})=q(a^{(13)(24)})=q(a)$, so $q_\pi\in\{1,q,1/q\}$).
\end{proof}

Let $\Gamma_0$ be the subgroup of $\Oct$ formed by the maps \eqref{Lpi} with $\pi\in\Pi$ and arbitrary signs $\varepsilon_i$; equivalently, $\Gamma_0$ consists of the $L\in\Oct$ which preserve the partition $\F=\{\pm f_1,\pm f_4\}\sqcup\{\pm f_2,\pm f_3\}$. Clearly $|\Gamma_0|=2^4\cdot 8=128$.

\begin{proposition}\label{pro:Gamma0}
For every $L\in\Gamma_0$ one has $S(Lx)=\pm Q_L(x)S(x)$ with a real analytic function $Q_L$ such that $\frac14\le Q_L\le4$. In particular, $\Gamma_0$ preserves $\M$. The matrix $A$ of \eqref{Amatrix} belongs to $\Gamma_0$, and
\begin{equation}\label{hidden}
S(x)=Q(x)\,S(Ax),\qquad Q(x)=\frac{(1+\alpha_1\alpha_3)(1+\alpha_2\alpha_4)}{(1+\alpha_1\alpha_2)(1+\alpha_3\alpha_4)} .
\end{equation}
\end{proposition}

\begin{proof}
Let $L$ be given by \eqref{Lpi}. By \eqref{xiL} and since $\sigma^2$ is even, $\alpha_{\pi(i)}(Lx)=\alpha_i(x)$, that is, $\alpha(Lx)=\alpha(x)^{\pi^{-1}}$. Hence, by \eqref{SE} and Corollary~\ref{cor:E}, $S(Lx)=E(\alpha(x)^{\pi^{-1}})=\pm q_{\pi^{-1}}(\alpha(x))S(x)$. The function $Q_L=q_{\pi^{-1}}\circ\alpha$ is real analytic, because the $\alpha_i$ are real analytic and the denominators do not vanish.

A direct computation gives $Af_1=f_1$, $Af_2=f_3$, $Af_3=f_2$, $Af_4=f_4$; thus $A\in\Gamma_0$ with $\pi=(23)$ and $\varepsilon_i=1$. In this case the proof of Corollary~\ref{cor:E} gives \eqref{hidden}. (Alternatively, \eqref{hidden} can be checked directly: $(Ax)_1=\xi_1+\xi_3$, $(Ax)_3=\xi_1-\xi_3$, $(Ax)_2=\xi_2+\xi_4$, $(Ax)_4=\xi_2-\xi_4$, so that $S(Ax)=\varphi(\alpha_1,\alpha_3)-\varphi(\alpha_2,\alpha_4)$ by the addition theorem.)
\end{proof}

\begin{remark}\label{rem:Gamma0}
In the original coordinates the group $\Gamma_0$ looks as follows. The signed permutations of coordinates contained in $\Gamma_0$ are exactly the $64$ maps $x\mapsto(d_1x_{\theta(1)},\dots,d_4x_{\theta(4)})$, $d_i=\pm1$, such that $\theta$ preserves the partition $\{\{1,3\},\{2,4\}\}$ and $d_1d_3=d_2d_4$; obviously, each of them maps $S$ to $\pm S$. For example, the permutation $(12)(34)$ of the frame vectors corresponds to $(x_1,x_2,x_3,x_4)\mapsto(x_1,x_2,-x_3,-x_4)$, the sign change of $f_2$ corresponds to the interchange of $x_1$ and $x_3$, and the permutation $(13)(24)$ corresponds to $(x_1,x_2,x_3,x_4)\mapsto(x_2,x_1,x_4,x_3)$. These $64$ maps form a subgroup of index two in $\Gamma_0$, and $\Gamma_0$ is generated by this subgroup and $A$. The symmetries of $\M$ which are signed permutations of coordinates are thus easy to see; the remaining half of $\Gamma_0$, the coset of $A$, is the hidden part.
\end{remark}

\subsection{A symmetry which does not fix the origin}

\begin{lemma}\label{lem:tau}
The isometry
$$\tau(x)=Rx+e_1=(x_1+1,x_2,x_3,-x_4)
$$
satisfies $S\circ\tau=-S$; in particular $\tau\in\Iso$. Moreover, $\tau(0)=e_1$, the linear part $R$ of $\tau$ belongs to $\Oct$, and $\tau(x)=h'+R(x-h')+n_0$ with $n_0=e_1+e_4\in D_4$.
\end{lemma}

\begin{proof}
We have
$$\sigma(x_1+1)\sigma(x_3)-\sigma(x_2)\sigma(-x_4)=-\sigma(x_1)\sigma(x_3)+\sigma(x_2)\sigma(x_4)=-S(x).
$$
The map $R$ fixes $f_1,f_2$ and interchanges $f_3$ and $f_4$, hence $R\in\Oct$. Finally, $h'-Rh'=h'-h=-e_4$, so $h'+R(x-h')+e_1+e_4=Rx+e_1$.
\end{proof}

Note that $R\notin\Gamma_0$, because the transposition $(34)$ does not preserve the partition $\{\{1,4\},\{2,3\}\}$. Note also that $R$ itself does not preserve $\M$: $S(Rx)=\sigma(x_1)\sigma(x_3)+\sigma(x_2)\sigma(x_4)$, which does not vanish, e.g., at $x=h$.

\subsection{Isometries of the singular set}
Every isometry of $\R{4}$ which maps $\M$ onto itself maps the set of points where $\M$ is not a topological manifold onto itself; by Theorem~\ref{thA}(c), this set is $\Sing$. Thus $\Iso$ is contained in the group
$$
G_\Sigma=\{g:\ g\ \text{is an isometry of}\ \R{4}\ \text{with}\ g(\Sing)=\Sing\}.
$$

\begin{lemma}\label{lem:Gsigma}
$G_\Sigma=\{x\mapsto h'+L(x-h')+n:\ L\in\Aut(D_4),\ n\in D_4\}$. In particular, $G_\Sigma$ contains the group of translations by $D_4$ as a normal subgroup of index $|\Aut(D_4)|=1152$.
\end{lemma}

\begin{proof}
Let $g\in G_\Sigma$, and write $g(x)=Lx+t$ with $L$ orthogonal.

\emph{Step 1: $L\in\Aut(D_4)$.} Since $0\in\Sing$, the additive group generated by $\Sing$ coincides with the group generated by the differences of points of $\Sing$. As $\Z^4\cup\{h\}\subset\Sing\subset D_4^*$, this group is $D_4^*$. For $y,y'\in\Sing$ we have
$$
L(y-y')=g(y)-g(y'),
$$
hence $L(D_4^*)\subset D_4^*$. Applying the same argument to $g^{-1}$, we get $L\in\Aut(D_4^*)=\Aut(D_4)$.

\emph{Step 2: $g$ fixes the set of holes.} Since $t=g(0)\in\Sing\subset D_4^*$, we have $g(D_4^*)=D_4^*$. As $g$ maps $\Sing=D_4^*\setminus(h'+D_4)$ onto itself, it maps $h'+D_4$ onto itself. In particular $g(h')=h'+n$ for some $n\in D_4$, and then
$$
g(x)=g(h')+L(x-h')=h'+L(x-h')+n .
$$

\emph{Step 3: the converse.} If $L\in\Aut(D_4)$ and $n\in D_4$, then the map $x\mapsto h'+L(x-h')+n$ preserves $D_4^*$ (because $h'\in D_4^*$) and $h'+D_4$, hence also $\Sing$.

Finally, the translations by $D_4$ form a normal subgroup of $G_\Sigma$, because $g\circ(x\mapsto x+n)\circ g^{-1}$ is the translation by $Ln\in D_4$. The quotient is isomorphic to $\Aut(D_4)$ via $g\mapsto dg$.
\end{proof}

By Lemma~\ref{lem:honeycomb}, $G_\Sigma$ is precisely the symmetry group of the $24$-cell honeycomb with vertex set $\Sing$.

\subsection{Proof of Theorem~\ref{thB}}
Part (a) is Proposition~\ref{pro:transl}.

\emph{Proof of \textup{(b)}.} Consider the group
$$
G=\{x\mapsto h'+L(x-h')+n:\ L\in\Oct,\ n\in D_4\}.
$$
By Corollary~\ref{cor:Oct}, $G\subset\Iso$, and clearly $|G/D_4|=|\Oct|=384$. On the other hand, by Lemma~\ref{lem:Gsigma},
$$
\Iso\subset G_\Sigma,\qquad |G_\Sigma/D_4|=1152 .
$$
Thus $\Iso/D_4$ is a subgroup of $G_\Sigma/D_4$ containing the subgroup $G/D_4$ of index $3$. Since $3$ is a prime, either $\Iso=G$ or $\Iso=G_\Sigma$.

The latter is impossible. Indeed, the coordinate interchange
$$
P(x_1,x_2,x_3,x_4)=(x_2,x_1,x_3,x_4)
$$
belongs to $\Aut(D_4)$ and fixes $h'$, hence $P\in G_\Sigma$. But the point $x=(0,\frac12,\frac12,0)$ lies in $\M$, while $S(Px)=S(\frac12,0,\frac12,0)=1\ne0$.

Hence $\Iso=G$, which proves (b). In particular, $\Iso=D_4\rtimes\Iso_{h'}$, and the stabilizer $\Iso_{h'}$ of the hole $h'$ is isomorphic to $\Oct$ via the map $g\mapsto dg$ which assigns to $g$ its linear part.

\emph{Transitivity.} We follow the orbit of the origin under $\Iso$.
\begin{itemize}
\item By Proposition~\ref{pro:transl}, the orbit contains $D_4$.
\item By Lemma~\ref{lem:tau}, it contains $\tau(D_4)=e_1+D_4$ (note that $R(D_4)=D_4$).
\item Since $A\in\Oct$ and $Ah'=h'$, the map $x\mapsto h'+A(x-h')=Ax$ belongs to $\Iso$. Since $Ae_1=h$ and $A\in\Aut(D_4)$, the orbit contains $A(e_1+D_4)=h+D_4$.
\end{itemize}
Hence $\Iso$ acts transitively on $\Sing$.

\emph{The stabilizer of the origin.} By Lemma~\ref{lem:Gsigma}, $D_4$ is a normal subgroup of $G_\Sigma$, and hence of $\Iso$. The group $\Iso/D_4$ acts transitively on the three-point set
$$
\Sing/D_4=\{[0],[e_1],[h]\}.
$$
The stabilizer of $[0]$ consists of the classes of those $g$ with $g(0)\in D_4$. Composing with a translation, we see that each such class contains exactly one element fixing $0$, i.e. one element of the stabilizer $\Iso_0$ of the origin. Therefore
\begin{equation}\label{orbit}
|\Iso/D_4|=3\,|\Iso_0| .
\end{equation}
Since $|\Iso/D_4|=384$, we get $|\Iso_0|=128$. By Proposition~\ref{pro:Gamma0}, $\Gamma_0\subset\Iso_0$, and both groups have order $128$; hence $\Iso_0=\Gamma_0$. This proves the first statement of (c); the last statement of (c) is Proposition~\ref{pro:Gamma0} (see also Remark~\ref{rem:thetaA}).

Finally, (d) follows from Lemma~\ref{lem:honeycomb}, Lemma~\ref{lem:Gsigma} and $|G_\Sigma/\Iso|=1152/384=3$. \qed

\begin{remark}\label{rem:sym}
(i) Theorem~\ref{thB} shows that the symmetry group of $\M$ is \emph{symmorphic}, i.e. a semidirect product of its translation group and its point group, but with respect to the hole $h'$ rather than to a singular point: no singular point is fixed by a subgroup of $\Iso$ isomorphic to $\Oct$. The stabilizer of any singular point $y$ is conjugate to $\Gamma_0$ (by transitivity) and has order $128$; it acts on the tangent cone $C_y$ and preserves the pair of eigenplanes of the Hessian of $S$ at $y$.

(ii) The point group $\Oct$ is not contained in $W(D_4)$: it acts on the three nontrivial cosets \eqref{cosets} through the full symmetric group $S_3$ (for instance, $A$ interchanges $e_1+D_4$ and $h+D_4$, and $R$ interchanges $h+D_4$ and $h'+D_4$), and $\Oct\cap W(D_4)$ has order $64$: since $W(D_4)$ consists of the signed permutations of coordinates with an even number of sign changes, $\Oct\cap W(D_4)$ consists exactly of the $64$ maps described in Remark~\ref{rem:Gamma0}. Of course, the \emph{affine} action of $\Iso$ fixes the class of holes.

(iii) By Proposition~\ref{pro:transl}, the subgroup $D_4^+=\{n\in D_4:\ n_1+n_3\ \text{even}\}$ of index two consists of the translations preserving $S$, while the translations by $D_4\setminus D_4^+$ change the sign of $S$, i.e. they reverse the unit normal $DS/|DS|$ of $\Mreg$. This fact is responsible for the non-orientability in Theorem~\ref{thC}(c).
\end{remark}

\begin{table}[ht]
\centering
\small
\renewcommand{\arraystretch}{1.35}
\setlength{\tabcolsep}{4pt}
\begin{tabular}{|l|c|c|c|}
\hline
isometry $g$ & $S\circ g$ & linear part in the coordinate $w$ & action on types\\
\hline
$x\mapsto x+n$, $n\in D_4^+$ & $S$ & $\Id$ & $\Id$\\
$x\mapsto x+n$, $n\in D_4\setminus D_4^+$ & $-S$ & $\Id$ & $\Id$\\
$(x_1,x_2,x_3,x_4)\mapsto(x_3,x_2,x_1,x_4)$ & $S$ & $(w_0,w_1,-w_2,w_3)$ & $\Id$\\
$(x_1,x_2,x_3,x_4)\mapsto(x_2,x_1,x_4,x_3)$ & $-S$ & $(w_3,-w_2,-w_1,w_0)$ & $\Id$\\
$(x_1,x_2,x_3,x_4)\mapsto(x_1,x_2,-x_3,-x_4)$ & $-S$ & $(-w_2,w_3,-w_0,w_1)$ & $\Id$\\
$x\mapsto Ax$ & $S/Q$ & $(w_0,w_1,-w_3,-w_2)$ & $\mathrm{j}\leftrightarrow\mathrm{k}$\\
$\tau:x\mapsto(x_1+1,x_2,x_3,-x_4)$ & $-S$ & $(w_0,w_3,w_2,w_1)$ & $\mathrm{i}\leftrightarrow\mathrm{k}$\\
$\tau\circ A$ & $-S/Q$ & $(w_0,-w_2,-w_3,w_1)$ & $\mathrm{i}\to\mathrm{k}\to\mathrm{j}\to\mathrm{i}$\\
\hline
\end{tabular}
\smallskip
\caption{Some elements of the isometry group $\Iso$. The second column gives the effect on the defining function $S$, where $Q$ is the positive function of Theorem~\ref{thB}(c). The third column gives the image of $(w_0,w_1,w_2,w_3)$ under the linear part of $g$ in the coordinate $w=\bar u(x-h')/\sqrt2$ of \S\ref{ssec:quaternionic}. The last column gives the induced permutation of the three types $\mathrm{i},\mathrm{j},\mathrm{k}$ of singular points (types of $D_4$, $h+D_4$ and $e_1+D_4$, respectively). The isometries $A$ and $\tau$ already induce the full symmetric group $S_3$ on the three types.}\label{tab:symmetries}
\end{table}

Table~\ref{tab:symmetries} collects some typical elements of $\Iso$ and their effect. Recall from Proposition~\ref{pro:quaternionic} that the singular points of $\M$ are of three types $\mathrm{i},\mathrm{j},\mathrm{k}$: the points of $D_4$, of $h+D_4$ and of $e_1+D_4$, respectively. In the coordinate $w$ of \S\ref{ssec:quaternionic}, the type of a singular point corresponds to one of the three ways of splitting the coordinates $(w_0,w_1,w_2,w_3)$ into two pairs: $\{w_0,w_1\}\sqcup\{w_2,w_3\}$ for type $\mathrm{i}$, $\{w_0,w_2\}\sqcup\{w_1,w_3\}$ for type $\mathrm{j}$, and $\{w_0,w_3\}\sqcup\{w_1,w_2\}$ for type $\mathrm{k}$. Hence the action of an isometry on the types can be read off from its linear part in the coordinate $w$, which is a signed permutation of $(w_0,w_1,w_2,w_3)$.

\section{Topology: proof of Theorem~\ref{thC}}\label{sec:top}

\subsection{Level curves in a square}
Let $Q=(0,1)^2$, $\bar Q=[0,1]^2$, $c=(\frac12,\frac12)$, and let
$$
\lambda(v_1,v_2)=\sigma(v_1)\sigma(v_2),\qquad (v_1,v_2)\in\bar Q .
$$
Then $\lambda=0$ on $\partial Q$, $0<\lambda\le1$ on $Q$, and $\lambda=1$ only at $c$ (Lemma~\ref{lem:sl}(ii)).

\begin{lemma}\label{lem:square}
For every $p\in\partial Q$ the function $t\mapsto\lambda(c+t(p-c))$ decreases strictly from $1$ to $0$ on $[0,1]$. Consequently, the map
$$
\Psi:\partial Q\times[0,1]\to\bar Q,\qquad \Psi(p,\mu)=\text{the point $v$ of the segment $[p,c]$ with $\lambda(v)=\mu$},
$$
is continuous, $\Psi(p,0)=p$, $\Psi(p,1)=c$, $\lambda(\Psi(p,\mu))=\mu$, and $\Psi$ induces a homeomorphism of the cone $\Cone(\partial Q)=\partial Q\times[0,1]/\partial Q\times\{1\}$ onto $\bar Q$. In particular, for $0<\mu<1$ the level set $\{\lambda=\mu\}\cap\bar Q$ is a smooth closed convex curve, and $\lambda$ has no critical points in $Q\setminus\{c\}$.
\end{lemma}

\begin{proof}
By Lemma~\ref{lem:sl}(v), $\ell=\log\lambda=\log\sigma(v_1)+\log\sigma(v_2)$ is strictly concave on $Q$, and $\ell(c)=0=\max\ell$. Hence the restriction of $\ell$ to the segment from $c$ to $p$ is strictly concave with its maximum at $c$, and therefore strictly decreasing; at $p$ it tends to $-\infty$. This proves the first statement. The map $\Psi$ is well defined by the monotonicity; its continuity follows from the continuity of $\lambda$ and the strict monotonicity along each segment, by a standard compactness argument. It is a bijection of $\partial Q\times[0,1)$ onto $\bar Q\setminus\{c\}$, since every point of $\bar Q\setminus\{c\}$ lies on exactly one segment $[c,p)$, $p\in\partial Q$ (here we use that $\bar Q$ is convex and $c$ is an interior point). Since $\Cone(\partial Q)$ is compact and $\bar Q$ is Hausdorff, the induced continuous bijection is a homeomorphism.
For $0<\mu<1$ the set $\{\ell\ge\log\mu\}$ is compact, strictly convex and contains $c$ in its interior, and $\{\lambda=\mu\}$ is its boundary. Finally, $D\lambda=(\sigma'(v_1)\sigma(v_2),\sigma(v_1)\sigma'(v_2))$ vanishes in $Q$ only if $\sigma'(v_1)=\sigma'(v_2)=0$, i.e. only at $c$.
\end{proof}

\begin{figure}[ht]
\centering
\IfFileExists{levels_normalized.pdf}{\includegraphics[width=0.45\textwidth]{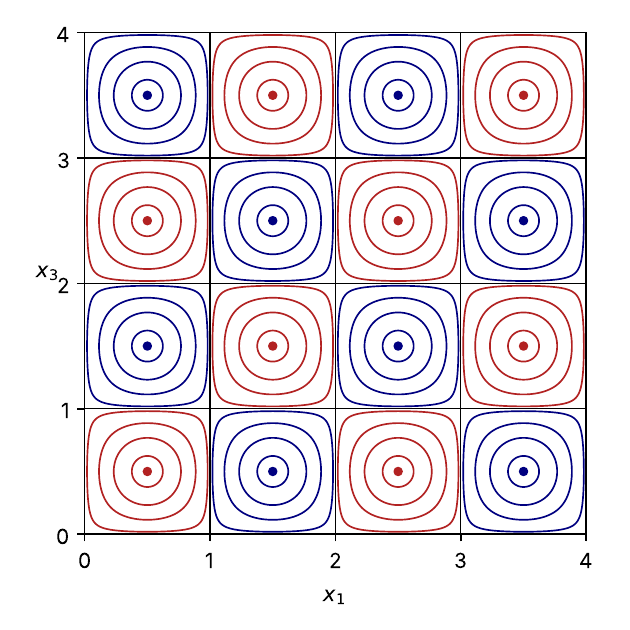}}{\fbox{\parbox{0.45\textwidth}{\centering figure levels\_normalized.pdf}}}
\caption{Level curves $\sigma(v_1)\sigma(v_2)=r$ in the unit squares of the plane: red for $r>0$, blue for $r<0$. The level sets are convex curves which shrink to the centres of the squares as $|r|\to1$.}\label{fig:levels}
\end{figure}

\subsection{The cell structure}
For $n\in\Z^4$ we put $Q^{13}_n=(n_1,n_1+1)\times(n_3,n_3+1)$ and $Q^{24}_n=(n_2,n_2+1)\times(n_4,n_4+1)$, so that, after the obvious permutation of coordinates, the open cube $\mathring P_n=n+(0,1)^4$ is $Q^{13}_n\times Q^{24}_n$ and $T_n=\partial Q^{13}_n\times\partial Q^{24}_n$.

\begin{proposition}\label{pro:cells}
\begin{enumerate}
\item[(i)] $\M\cap\partial P_n=T_n$ for every $n\in\Z^4$. Every point of $\M$ which does not lie in any open cube $\mathring P_n$ belongs to $\Mo$, and $\Mo\cap P_n=T_n$.
\item[(ii)] If $n\in e_1+D_4$, then $\M\cap\mathring P_n=\emptyset$.
\item[(iii)] If $n\in D_4$, then $K_n=\M\cap P_n$ is homeomorphic to the cone over $T_n$, by a homeomorphism which is the identity on $T_n$ and maps the vertex of the cone to $n+h$.
\end{enumerate}
\end{proposition}

\begin{proof}
Let $\lambda_{13}(x)=\sigma(x_1)\sigma(x_3)$ and $\lambda_{24}(x)=\sigma(x_2)\sigma(x_4)$, so that $\M=\{\lambda_{13}=\lambda_{24}\}$.

(i) Let $x\in\M$ and suppose that some coordinate of $x$ is an integer, say $x_1\in\Z$ (the other cases are similar). Then $\lambda_{13}(x)=0$, hence $\lambda_{24}(x)=0$, i.e. $x_2\in\Z$ or $x_4\in\Z$; thus $x\in\Mo$. Conversely, a point of $\Mo$ satisfies $\lambda_{13}=\lambda_{24}=0$. Applying this to $x\in P_n$, we see that $\M\cap\partial P_n$ and $\Mo\cap P_n$ both coincide with the set of $x\in P_n$ with $(x_1,x_3)\in\partial Q^{13}_n$ and $(x_2,x_4)\in\partial Q^{24}_n$, which is $T_n$.

(ii) On $\mathring P_n$ the functions $\lambda_{13}$ and $\lambda_{24}$ do not vanish and have the signs $(-1)^{n_1+n_3}$ and $(-1)^{n_2+n_4}$, respectively. If $n\in e_1+D_4$ these signs are different.

(iii) Let $n\in D_4$ and $\epsilon=(-1)^{n_1+n_3}=(-1)^{n_2+n_4}$. By the periodicity of $\sigma$, the function $\epsilon\lambda_{13}$ on $\bar Q^{13}_n$ is the translate of the function $\lambda$ of Lemma~\ref{lem:square}, and so is $\epsilon\lambda_{24}$ on $\bar Q^{24}_n$; let $\Psi_{13}$ and $\Psi_{24}$ be the corresponding maps. By (i) and (ii),
\begin{align*}
K_n&=\{x\in P_n:\ \epsilon\lambda_{13}(x)=\epsilon\lambda_{24}(x)\}\\
&=\{(\Psi_{13}(p,\mu),\Psi_{24}(p',\mu)):\ p\in\partial Q^{13}_n,\ p'\in\partial Q^{24}_n,\ \mu\in[0,1]\}.
\end{align*}
Hence the map $(p,p',\mu)\mapsto(\Psi_{13}(p,\mu),\Psi_{24}(p',\mu))$ induces a continuous bijection of $\Cone(T_n)=T_n\times[0,1]/T_n\times\{1\}$ onto $K_n$; it is a homeomorphism by compactness. For $\mu=0$ it is the identity of $T_n$, and the vertex $\mu=1$ is mapped to the centre $n+h$ of $P_n$.
\end{proof}

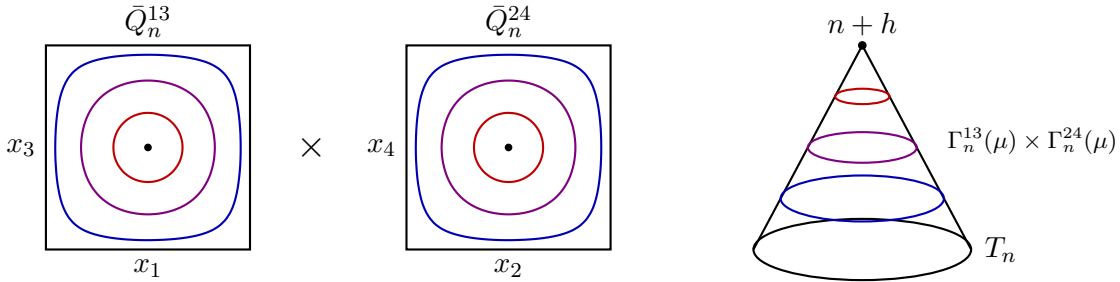
\begin{figure}[ht]
\centering

\begin{tikzpicture}[scale=0.9]
\draw[thick] (0,0) rectangle (3,3);
\draw[blue!70!black,thick] plot[smooth cycle] coordinates {(1.282,0.142) (1.462,0.137) (1.650,0.140) (1.823,0.149) (1.989,0.165) (2.152,0.191) (2.295,0.226) (2.430,0.277) (2.543,0.345) (2.640,0.438) (2.715,0.554) (2.767,0.682) (2.805,0.828) (2.831,0.983) (2.850,1.155) (2.860,1.328) (2.863,1.508) (2.859,1.695) (2.848,1.860) (2.829,2.032) (2.801,2.190) (2.762,2.333) (2.708,2.460) (2.631,2.573) (2.535,2.661) (2.415,2.730) (2.287,2.776) (2.135,2.812) (1.972,2.837) (1.808,2.852) (1.627,2.861) (1.448,2.862) (1.268,2.857) (1.087,2.844) (0.930,2.824) (0.772,2.792) (0.637,2.751) (0.510,2.691) (0.404,2.610) (0.323,2.511) (0.258,2.385) (0.211,2.243) (0.180,2.093) (0.158,1.927) (0.145,1.763) (0.138,1.583) (0.138,1.395) (0.146,1.223) (0.160,1.050) (0.183,0.892) (0.217,0.735) (0.264,0.600) (0.329,0.480) (0.413,0.382) (0.518,0.305) (0.645,0.246) (0.794,0.203) (0.945,0.174) (1.110,0.154)};
\draw[violet,thick] plot[smooth cycle] coordinates {(1.470,0.517) (1.590,0.520) (1.710,0.535) (1.823,0.561) (1.920,0.594) (2.015,0.637) (2.093,0.684) (2.170,0.742) (2.240,0.810) (2.301,0.885) (2.353,0.968) (2.397,1.057) (2.430,1.147) (2.459,1.260) (2.476,1.373) (2.483,1.500) (2.476,1.627) (2.459,1.740) (2.433,1.845) (2.397,1.942) (2.353,2.032) (2.301,2.115) (2.240,2.190) (2.170,2.258) (2.100,2.311) (2.015,2.362) (1.920,2.406) (1.823,2.439) (1.710,2.465) (1.590,2.480) (1.470,2.483) (1.350,2.474) (1.230,2.453) (1.125,2.423) (1.030,2.385) (0.945,2.340) (0.867,2.288) (0.795,2.226) (0.728,2.152) (0.674,2.078) (0.623,1.987) (0.585,1.897) (0.555,1.800) (0.531,1.688) (0.518,1.560) (0.518,1.440) (0.531,1.312) (0.555,1.200) (0.585,1.103) (0.623,1.013) (0.674,0.922) (0.728,0.848) (0.795,0.774) (0.867,0.713) (0.945,0.660) (1.030,0.615) (1.125,0.577) (1.230,0.547) (1.350,0.526)};
\draw[red!75!black,thick] plot[smooth cycle] coordinates {(1.425,0.997) (1.492,0.992) (1.560,0.995) (1.617,1.005) (1.669,1.020) (1.718,1.040) (1.763,1.064) (1.808,1.094) (1.845,1.125) (1.875,1.155) (1.906,1.193) (1.935,1.235) (1.958,1.277) (1.979,1.328) (1.995,1.383) (2.005,1.440) (2.008,1.508) (2.003,1.575) (1.992,1.627) (1.976,1.680) (1.953,1.732) (1.927,1.777) (1.900,1.815) (1.868,1.853) (1.836,1.883) (1.799,1.913) (1.755,1.941) (1.707,1.965) (1.657,1.984) (1.605,1.998) (1.545,2.006) (1.478,2.008) (1.417,2.002) (1.355,1.987) (1.305,1.970) (1.260,1.949) (1.215,1.922) (1.177,1.894) (1.140,1.860) (1.110,1.828) (1.078,1.785) (1.051,1.740) (1.030,1.695) (1.013,1.645) (1.000,1.590) (0.993,1.530) (0.994,1.455) (1.002,1.395) (1.016,1.343) (1.035,1.293) (1.057,1.248) (1.083,1.208) (1.112,1.170) (1.147,1.132) (1.185,1.100) (1.223,1.073) (1.268,1.047) (1.312,1.027) (1.365,1.010)};
\fill (1.5,1.5) circle (1.6pt);
\node[below] at (1.5,0) {$x_1$};
\node[left] at (0,1.5) {$x_3$};
\node at (1.5,3.35) {$\bar Q^{13}_n$};
\node at (3.9,1.5) {\Large$\times$};
\draw[thick] (5.3,0) rectangle (8.3,3);
\draw[blue!70!black,thick] plot[smooth cycle] coordinates {(6.582,0.142) (6.762,0.137) (6.950,0.140) (7.123,0.149) (7.289,0.165) (7.452,0.191) (7.595,0.226) (7.730,0.277) (7.843,0.345) (7.940,0.438) (8.015,0.554) (8.067,0.682) (8.105,0.828) (8.131,0.983) (8.150,1.155) (8.160,1.328) (8.163,1.508) (8.159,1.695) (8.148,1.860) (8.129,2.032) (8.101,2.190) (8.062,2.333) (8.008,2.460) (7.931,2.573) (7.835,2.661) (7.715,2.730) (7.587,2.776) (7.435,2.812) (7.272,2.837) (7.107,2.852) (6.928,2.861) (6.747,2.862) (6.567,2.857) (6.387,2.844) (6.230,2.824) (6.072,2.792) (5.938,2.751) (5.810,2.691) (5.704,2.610) (5.622,2.511) (5.558,2.385) (5.511,2.243) (5.480,2.093) (5.458,1.927) (5.445,1.763) (5.438,1.583) (5.438,1.395) (5.446,1.223) (5.460,1.050) (5.483,0.892) (5.517,0.735) (5.564,0.600) (5.629,0.480) (5.712,0.382) (5.817,0.305) (5.945,0.246) (6.094,0.203) (6.245,0.174) (6.410,0.154)};
\draw[violet,thick] plot[smooth cycle] coordinates {(6.770,0.517) (6.890,0.520) (7.010,0.535) (7.123,0.561) (7.220,0.594) (7.315,0.637) (7.393,0.684) (7.470,0.742) (7.540,0.810) (7.601,0.885) (7.653,0.968) (7.697,1.057) (7.730,1.147) (7.759,1.260) (7.776,1.373) (7.783,1.500) (7.776,1.627) (7.759,1.740) (7.733,1.845) (7.697,1.942) (7.653,2.032) (7.601,2.115) (7.540,2.190) (7.470,2.258) (7.400,2.311) (7.315,2.362) (7.220,2.406) (7.123,2.439) (7.010,2.465) (6.890,2.480) (6.770,2.483) (6.650,2.474) (6.530,2.453) (6.425,2.423) (6.330,2.385) (6.245,2.340) (6.167,2.288) (6.095,2.226) (6.028,2.152) (5.974,2.078) (5.923,1.987) (5.885,1.897) (5.855,1.800) (5.831,1.688) (5.818,1.560) (5.818,1.440) (5.831,1.312) (5.855,1.200) (5.885,1.103) (5.923,1.013) (5.974,0.922) (6.028,0.848) (6.095,0.774) (6.167,0.713) (6.245,0.660) (6.330,0.615) (6.425,0.577) (6.530,0.547) (6.650,0.526)};
\draw[red!75!black,thick] plot[smooth cycle] coordinates {(6.725,0.997) (6.792,0.992) (6.860,0.995) (6.917,1.005) (6.969,1.020) (7.018,1.040) (7.062,1.064) (7.107,1.094) (7.145,1.125) (7.175,1.155) (7.206,1.193) (7.235,1.235) (7.258,1.277) (7.279,1.328) (7.295,1.383) (7.305,1.440) (7.308,1.508) (7.303,1.575) (7.292,1.627) (7.276,1.680) (7.253,1.732) (7.227,1.777) (7.200,1.815) (7.168,1.853) (7.136,1.883) (7.099,1.913) (7.055,1.941) (7.007,1.965) (6.957,1.984) (6.905,1.998) (6.845,2.006) (6.777,2.008) (6.717,2.002) (6.655,1.987) (6.605,1.970) (6.560,1.949) (6.515,1.922) (6.478,1.894) (6.440,1.860) (6.410,1.828) (6.378,1.785) (6.351,1.740) (6.330,1.695) (6.312,1.645) (6.300,1.590) (6.293,1.530) (6.294,1.455) (6.302,1.395) (6.316,1.343) (6.335,1.293) (6.357,1.248) (6.383,1.208) (6.412,1.170) (6.447,1.132) (6.485,1.100) (6.523,1.073) (6.567,1.047) (6.612,1.027) (6.665,1.010)};
\fill (6.8,1.5) circle (1.6pt);
\node[below] at (6.8,0) {$x_2$};
\node[left] at (5.3,1.5) {$x_4$};
\node at (6.8,3.35) {$\bar Q^{24}_n$};
\begin{scope}[shift={(12,0)}]
\draw[thick] (0,0) ellipse (1.6 and 0.45);
\draw[thick] (-1.6,0)--(0,3) (1.6,0)--(0,3);
\draw[blue!70!black,thick] (0,0.750) ellipse (1.200 and 0.338);
\draw[violet,thick] (0,1.500) ellipse (0.800 and 0.225);
\draw[red!75!black,thick] (0,2.250) ellipse (0.400 and 0.113);
\fill (0,3) circle (1.8pt) node[above] {$n+h$};
\node[right] at (1.65,0) {$T_n$};
\node[right,font=\footnotesize] at (1.1,1.6) {$\Gamma^{13}_n(\mu)\times\Gamma^{24}_n(\mu)$};
\end{scope}
\end{tikzpicture} 

\caption{The cone structure of $K_n=\M\cap P_n$, $n\in D_4$ (Proposition~\ref{pro:cells}). Left: the level curves $\Gamma^{13}_n(\mu)=\{\epsilon\lambda_{13}=\mu\}\cap\bar Q^{13}_n$ and $\Gamma^{24}_n(\mu)=\{\epsilon\lambda_{24}=\mu\}\cap\bar Q^{24}_n$, $\epsilon=(-1)^{n_1+n_3}$, for $\mu=0.12,\,0.45,\,0.82$. Right: a schematic picture of $K_n$; its slice $\{\epsilon\lambda_{13}=\mu\}$ is the torus $\Gamma^{13}_n(\mu)\times\Gamma^{24}_n(\mu)$, which tends to $T_n=\partial Q^{13}_n\times\partial Q^{24}_n\subset\Mo$ as $\mu\to0$ and shrinks to the singular point $n+h$ as $\mu\to1$.}\label{fig:strata}
\end{figure}

\subsection{Proof of Theorem~\ref{thC}}
(a) This is Proposition~\ref{pro:cells}: since $\R{4}$ is the union of the closed cubes $P_n$, we have $\M=\bigcup_n(\M\cap P_n)$, and $\M\cap P_n=T_n\subset\Mo$ for $n\notin D_4$. Two different cubes meet only along their boundaries, so $K_n\cap K_{n'}\subset\M\cap\partial P_n\subset\Mo$ for $n\neq n'$.

(b) The skeleton $\Mo$ is path connected. Indeed, every plane of $\Mo$ contains points of $\Z^4$, and two points $m,m+e_k\in\Z^4$ are joined by the segment $[m,m+e_k]$, which lies in a plane of $\Mo$: for instance, if $k=1$ then the segment lies in the plane $\{x_3=m_3,\ x_2=m_2\}$. Each $K_n$ is a cone, hence path connected, and it meets $\Mo$. By (a), $\M$ is path connected.

To see that $\Mreg$ is connected, note that by Theorem~\ref{thA}(c) every $y\in\Sing$ has arbitrarily small neighbourhoods $V$ in $\M$ such that $V\setminus\{y\}$ is homeomorphic to $T^2\times(0,1)$, which is connected. Since $\Sing$ is discrete and $\M$ is connected and locally path connected, removing the points of $\Sing$ does not disconnect $\M$.

(c) By Proposition~\ref{pro:transl}, $D_4$ acts on $\M$, and $\M/D_4$ is a compact subset of the torus $\R{4}/D_4$, which is a minimal hypersurface away from the images of the singular points. Since $\Sing/D_4$ consists of the three classes of $0$, $e_1$ and $h$, $\M/D_4$ has exactly three singular points, each with a neighbourhood homeomorphic to the cone over $T^2$. It is connected by (b).

The regular part $\Mreg/D_4$ is non-orientable. Indeed, otherwise it would be two-sided in the orientable manifold $\R{4}/D_4$, so it would carry a continuous unit normal field; its lift to $\Mreg$ would be a $D_4$-invariant continuous unit normal field $N$. Since $\Mreg$ is connected, $N=\pm DS/|DS|$ everywhere. But the translation by $e_1+e_2\in D_4$ maps $DS/|DS|$ to $-DS/|DS|$ by Proposition~\ref{pro:transl}, a contradiction.

\emph{The cell structure.} To compute the homology we use the cell structure given by (a). Put
$$
X=\M/D_4,\qquad X_0=\Mo/D_4\subset X .
$$
We decompose $\Mo$ into the vertices $m\in\Z^4$, the unit edges $[m,m+e_k]$, and the unit squares $m+[0,1]e_i+[0,1]e_j$ with $\{i,j\}$ in the set of \emph{mixed pairs}
$$
\mathcal{P}=\bigl\{\{1,2\},\{1,4\},\{2,3\},\{3,4\}\bigr\},
$$
each square being oriented by $(e_i,e_j)$, $i<j$. Each torus $T_n$ is then a subcomplex with $16$ vertices, $32$ edges and $16$ squares, and every cone $K_n$, $n\in D_4$, is decomposed into its vertex $n+h$ and the cones over the cells of $T_n$. This decomposition is invariant under $D_4$, which acts freely on the cells, so it induces a cell decomposition of $X$. Since $\Z^4/D_4$ has two elements and $D_4/D_4$ has one element, $X$ has $3$ vertices, $8+16=24$ edges, $8+32=40$ two-cells and $16$ three-cells. Hence
$$
\chi(\M/D_4)=3-24+40-16=3 .
$$
Moreover, all cones $K_n$ are translates of $K_0$, and the interiors of different cones lie in different open cubes. Therefore
$$
X=X_0\cup_{\varphi}\Cone(T_0),
$$
where $\varphi:T_0\to X_0$ is the restriction of the projection $\R{4}\to\R{4}/D_4$.

\emph{The complex $X_0$.} Since $e_k\notin D_4$, the subcomplex $X_0$ has two vertices $v_0=[0]$ and $v_1=[e_1]$, eight edges $E_k^\varepsilon$ (the images of the edges $[m,m+e_k]$ with $m\equiv\varepsilon e_1$ modulo $D_4$, where $\varepsilon\in\{0,1\}$), and eight squares $S_{ij}^\varepsilon$, $\{i,j\}\in\mathcal{P}$, defined in the same way. The edge $E_k^\varepsilon$ goes from $v_\varepsilon$ to $v_{1-\varepsilon}$, and a direct inspection of the four sides of a square gives
$$
\partial E_k^\varepsilon=v_{1-\varepsilon}-v_\varepsilon,\qquad
\partial S_{ij}^\varepsilon=E_i^\varepsilon+E_j^{1-\varepsilon}-E_i^{1-\varepsilon}-E_j^{\varepsilon}=(-1)^\varepsilon(d_i-d_j),
$$
where $d_k=E_k^0-E_k^1$.

Let $\delta:C_1(X_0)\to\Z^4$ be the \emph{displacement}, $\delta(E_k^\varepsilon)=e_k$. A $1$-chain $\sum a_k^\varepsilon E_k^\varepsilon$ is a cycle if and only if $\sum_ka_k^0=\sum_ka_k^1$, and then $\delta$ of it has even coordinate sum, i.e. lies in $D_4$. The kernel of $\delta$ on cycles consists of the chains with $a_k^1=-a_k^0$ and $\sum_ka^0_k=0$, i.e. of the combinations $\sum c_kd_k$ with $\sum c_k=0$. Since the mixed pairs form the cycle $1-2-3-4-1$ on the index set, these combinations are exactly the boundaries $\partial C_2(X_0)$. Finally, $\delta$ is onto $D_4$, because $\delta(E_i^0+E_j^1)=e_i+e_j$ and $\delta(E_i^0-E_j^0)=e_i-e_j$. Hence
$$
\delta_*:H_1(X_0)\xrightarrow{\ \cong\ }D_4\cong\Z^4 ,
$$
and $\delta_*$ is the map induced by the inclusion $X_0\subset\R{4}/D_4$, under the identification $H_1(\R{4}/D_4)=D_4$.

A $2$-chain $\sum b_{ij}^\varepsilon S_{ij}^\varepsilon$ is a cycle if and only if $\sum_{ij}(b_{ij}^0-b_{ij}^1)(d_i-d_j)=0$. The only relation between the four vectors $d_1-d_2$, $d_2-d_3$, $d_3-d_4$, $d_1-d_4$ is that the sum of the first three equals the fourth. Therefore $H_2(X_0)=Z_2(X_0)$ is free of rank $5$, with the basis
$$
\tau_{ij}=S_{ij}^0+S_{ij}^1\quad(\{i,j\}\in\mathcal{P}),\qquad
\rho=S_{12}^0+S_{23}^0+S_{34}^0-S_{14}^0 .
$$
Geometrically, $\tau_{ij}$ is the class of the $2$-torus which is the image of a plane of $\Mo$ spanned by $e_i,e_j$.

\emph{The attaching map.} The torus $T_0=\partial Q^{13}_0\times\partial Q^{24}_0$ carries the fundamental class $[\partial Q^{13}_0]\times[\partial Q^{24}_0]$, where $[\partial Q^{13}_0]$ is the sum of the four edges of the square $[0,1]e_1+[0,1]e_3$ oriented counterclockwise, and similarly for $\partial Q^{24}_0$. Multiplying out these sums of edges and projecting each of the resulting $16$ squares to $X_0$, one finds that every square $S_{ij}^\varepsilon$ is hit twice, with the same sign:
$$
\varphi_*[T_0]=2\bigl(S_{12}^0-S_{14}^0+S_{23}^0+S_{34}^0\bigr)-2\bigl(S_{12}^1-S_{14}^1+S_{23}^1+S_{34}^1\bigr)
=2\zeta,
$$
where
$$
\zeta=2\rho-\tau_{12}-\tau_{23}-\tau_{34}+\tau_{14}
$$
is a primitive element of $H_2(X_0)$. On the other hand, $H_1(T_0)$ is generated by the loops $\partial Q^{13}_0\times\{pt\}$ and $\{pt\}\times\partial Q^{24}_0$, whose displacement is zero; hence $\varphi_*=0$ on $H_1(T_0)$.

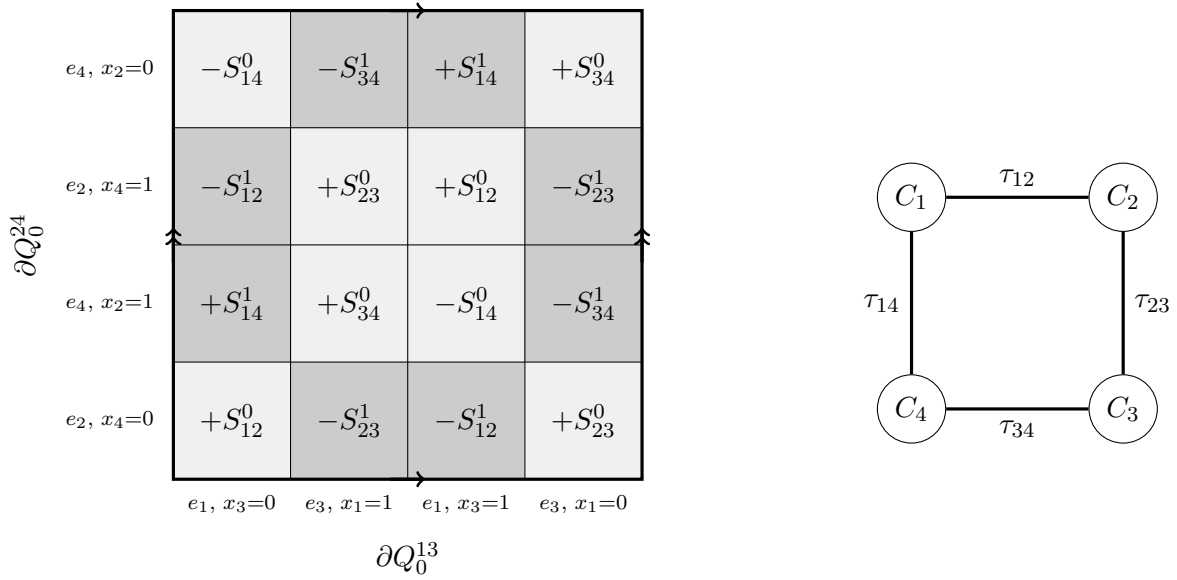
\begin{figure}[ht]
\centering

\begin{tikzpicture}[x=1.55cm,y=1.55cm]
\fill[gray!12] (0,0) rectangle (1,1);
\node at (0.5,0.5) {$+S^{0}_{12}$};
\fill[gray!40] (1,0) rectangle (2,1);
\node at (1.5,0.5) {$-S^{1}_{23}$};
\fill[gray!40] (2,0) rectangle (3,1);
\node at (2.5,0.5) {$-S^{1}_{12}$};
\fill[gray!12] (3,0) rectangle (4,1);
\node at (3.5,0.5) {$+S^{0}_{23}$};
\fill[gray!40] (0,1) rectangle (1,2);
\node at (0.5,1.5) {$+S^{1}_{14}$};
\fill[gray!12] (1,1) rectangle (2,2);
\node at (1.5,1.5) {$+S^{0}_{34}$};
\fill[gray!12] (2,1) rectangle (3,2);
\node at (2.5,1.5) {$-S^{0}_{14}$};
\fill[gray!40] (3,1) rectangle (4,2);
\node at (3.5,1.5) {$-S^{1}_{34}$};
\fill[gray!40] (0,2) rectangle (1,3);
\node at (0.5,2.5) {$-S^{1}_{12}$};
\fill[gray!12] (1,2) rectangle (2,3);
\node at (1.5,2.5) {$+S^{0}_{23}$};
\fill[gray!12] (2,2) rectangle (3,3);
\node at (2.5,2.5) {$+S^{0}_{12}$};
\fill[gray!40] (3,2) rectangle (4,3);
\node at (3.5,2.5) {$-S^{1}_{23}$};
\fill[gray!12] (0,3) rectangle (1,4);
\node at (0.5,3.5) {$-S^{0}_{14}$};
\fill[gray!40] (1,3) rectangle (2,4);
\node at (1.5,3.5) {$-S^{1}_{34}$};
\fill[gray!40] (2,3) rectangle (3,4);
\node at (2.5,3.5) {$+S^{1}_{14}$};
\fill[gray!12] (3,3) rectangle (4,4);
\node at (3.5,3.5) {$+S^{0}_{34}$};
\foreach \t in {1,2,3} {\draw[thin] (\t,0)--(\t,4); \draw[thin] (0,\t)--(4,\t);}
\draw[very thick] (0,0) rectangle (4,4);
\draw[very thick,->] (1.85,0)--(2.15,0); \draw[very thick,->] (1.85,4)--(2.15,4);
\draw[very thick,->>] (0,1.85)--(0,2.15); \draw[very thick,->>] (4,1.85)--(4,2.15);
\node[below,font=\scriptsize] at (0.5,-0.05) {$e_1,\,x_3{=}0$};
\node[below,font=\scriptsize] at (1.5,-0.05) {$e_3,\,x_1{=}1$};
\node[below,font=\scriptsize] at (2.5,-0.05) {$e_1,\,x_3{=}1$};
\node[below,font=\scriptsize] at (3.5,-0.05) {$e_3,\,x_1{=}0$};
\node[left,font=\scriptsize] at (-0.08,0.5) {$e_2,\,x_4{=}0$};
\node[left,font=\scriptsize] at (-0.08,1.5) {$e_4,\,x_2{=}1$};
\node[left,font=\scriptsize] at (-0.08,2.5) {$e_2,\,x_4{=}1$};
\node[left,font=\scriptsize] at (-0.08,3.5) {$e_4,\,x_2{=}0$};
\node[below] at (2,-0.45) {$\partial Q^{13}_0$};
\node[rotate=90] at (-1.25,2) {$\partial Q^{24}_0$};
\begin{scope}[shift={(6.3,0.6)},x=1cm,y=1cm]
\node[draw,circle,minimum size=0.85cm,fill=white] (C1) at (0,2.8) {$C_1$};
\node[draw,circle,minimum size=0.85cm,fill=white] (C2) at (2.8,2.8) {$C_2$};
\node[draw,circle,minimum size=0.85cm,fill=white] (C3) at (2.8,0) {$C_3$};
\node[draw,circle,minimum size=0.85cm,fill=white] (C4) at (0,0) {$C_4$};
\draw[very thick] (C1)--node[above] {$\tau_{12}$} (C2);
\draw[very thick] (C2)--node[right] {$\tau_{23}$} (C3);
\draw[very thick] (C3)--node[below] {$\tau_{34}$} (C4);
\draw[very thick] (C1)--node[left] {$\tau_{14}$} (C4);
\end{scope}
\end{tikzpicture} 

\caption{The cell structure of $\M/D_4$. Left: the torus $T_0=\partial Q^{13}_0\times\partial Q^{24}_0$, cut open along two circles (opposite sides are identified as indicated by the arrows), with its $16$ squares labelled by their images $\pm S^{\varepsilon}_{ij}$ in the skeleton $X_0=\Mo/D_4$; the shading indicates $\varepsilon\in\{0,1\}$. Each of the $8$ squares of $X_0$ appears exactly twice, with the same sign, in diagonally opposite positions. The space $\M/D_4$ is obtained from $X_0$ by attaching the cone over this torus. Right: the skeleton $X_0$ as a ``necklace'' of four $2$-tori $\tau_{12},\tau_{23},\tau_{34},\tau_{14}$ (the images of the planes of $\Mo$); two tori joined by a segment intersect along the circle $C_k=E^0_k\cup E^1_k$ of their common direction $e_k$, and two opposite tori meet only at the two vertices of $X_0$.}\label{fig:cells}
\end{figure}

\emph{The homology of $X$.}
We now combine the information on $X_0$ and on the attaching map. Recall that $X$ is obtained from $X_0$ by attaching the cone over the torus $T_0$ along the map $\varphi:T_0\to X_0$, that is,
$$
X=X_0\cup_\varphi\Cone(T_0),
$$
where the base $T_0\times\{0\}$ of the cone is glued to $X_0$ by $\varphi$. We use the long exact homology sequence of the pair $(X,X_0)$ (all homology groups have integer coefficients, and $\tilde H$ denotes reduced homology).
The argument involves only standard facts of algebraic topology---the exact sequence of a pair, excision, and cellular homology; see, e.g., \cite[Ch.~2]{Hatcher}.

\emph{Step 1: the relative groups.} By excision, the relative homology of $(X,X_0)$ does not change if we cut out a collar of $X_0$; hence it coincides with the relative homology of the pair $(\Cone T_0,T_0)$. Since the cone is contractible, the long exact sequence of this pair gives
$$
H_k(X,X_0)\cong H_k(\Cone T_0,T_0)\cong\tilde H_{k-1}(T_0),\qquad k\ge1 .
$$
Thus $H_3(X,X_0)\cong H_2(T_0)\cong\Z$, $H_2(X,X_0)\cong H_1(T_0)\cong\Z^2$ and $H_1(X,X_0)\cong\tilde H_0(T_0)=0$.

\emph{Step 2: the exact sequence.} Under these isomorphisms, the connecting homomorphism $H_k(X,X_0)\to H_{k-1}(X_0)$ of the pair becomes the map
$$
\varphi_*:H_{k-1}(T_0)\to H_{k-1}(X_0)
$$
induced by the attaching map. (Geometrically: a relative cycle in $X$ is a cone over a cycle $c$ of $T_0$, and its boundary is the image $\varphi(c)$ in $X_0$.) Since $X_0$ is a $2$-dimensional complex, $H_3(X_0)=0$, and the long exact sequence of the pair $(X,X_0)$ reads
\begin{multline*}
0\to H_3(X)\to H_2(T_0)\xrightarrow{\ \varphi_*\ }H_2(X_0)\to H_2(X)\to H_1(T_0)\xrightarrow{\ \varphi_*\ }H_1(X_0)\to H_1(X)\to 0 .
\end{multline*}
We have already computed all ingredients: $H_2(X_0)\cong\Z^5$ with the basis $\tau_{12},\tau_{23},\tau_{34},\tau_{14},\rho$; $H_1(X_0)\cong D_4\cong\Z^4$ via the displacement $\delta_*$; $\varphi_*[T_0]=2\zeta$; and $\varphi_*=0$ on $H_1(T_0)$.

\emph{Step 3: $H_3(X)=0$.} By exactness, $H_3(X)$ is the kernel of $\varphi_*:H_2(T_0)\to H_2(X_0)$. The group $H_2(T_0)\cong\Z$ is generated by the fundamental class $[T_0]$, and $\varphi_*[T_0]=2\zeta\neq0$ in the free group $H_2(X_0)$. Hence this map is injective and $H_3(X)=0$. (This agrees with the non-orientability of $\Mreg/D_4$.)

\emph{Step 4: $H_1(X)$.} Since $\varphi_*=0$ on $H_1(T_0)$, the end of the exact sequence shows that the map $H_1(X_0)\to H_1(X)$ induced by the inclusion is injective; it is surjective because $H_1(X,X_0)=0$. Thus
$$
H_1(X)\cong H_1(X_0)\cong D_4\cong\Z^4 .
$$
Moreover, $\delta_*$ is induced by the inclusion $X_0\subset\R{4}/D_4$, so the inclusion $X\subset\R{4}/D_4$ also induces an isomorphism $H_1(X)\cong H_1(\R{4}/D_4)$.

\emph{Step 5: $H_2(X)$.} The middle part of the exact sequence gives the short exact sequence
$$
0\to \operatorname{coker}\bigl(\varphi_*:H_2(T_0)\to H_2(X_0)\bigr)\to H_2(X)\to\ker\bigl(\varphi_*:H_1(T_0)\to H_1(X_0)\bigr)\to0 ,
$$
that is,
$$
0\to H_2(X_0)/\langle2\zeta\rangle\to H_2(X)\to H_1(T_0)\to0 .
$$
To compute the cokernel, recall that $\zeta=2\rho-\tau_{12}-\tau_{23}-\tau_{34}+\tau_{14}$ is \emph{primitive}, i.e. it is not a nontrivial multiple of another element of $H_2(X_0)$; for instance, its coefficient at $\tau_{12}$ equals $-1$. Hence $\zeta$ can be completed to a basis $\zeta,\tau_{23},\tau_{34},\tau_{14},\rho$ of $H_2(X_0)\cong\Z^5$, and in this basis the subgroup $\langle2\zeta\rangle$ is generated by twice the first basis vector. Therefore
$$
H_2(X_0)/\langle 2\zeta\rangle\cong\Z_2\oplus\Z^4 .
$$
Since the right-hand term $H_1(T_0)\cong\Z^2$ is free, the short exact sequence splits, and we conclude that
$$
H_2(X)\cong\Z^4\oplus\Z_2\oplus\Z^2=\Z^6\oplus\Z_2 .
$$
As a check, the Euler characteristic computed from these groups is $\chi(X)=1-4+6-0=3$, in accordance with the count of cells above.

(d) By Proposition~\ref{pro:transl}, the translations by $D_4^+$ preserve $S$ and hence the unit normal $DS/|DS|$, which therefore descends to $\Mreg/D_4^+$; thus $\Mreg/D_4^+$ is orientable. Since $D_4^+$ has index two in $D_4$ and $D_4/D_4^+$ acts freely on $\R{4}/D_4^+$, the projection $\M/D_4^+\to\M/D_4$ is a double cover, and $\chi(\M/D_4^+)=2\chi(\M/D_4)=6$. Every $2$-cell of the induced cell decomposition of $\M/D_4^+$ lies in the boundary of exactly two $3$-cells (as in the proof of (c)), and the regular part is connected and orientable. Hence the $3$-cycles are exactly the integer multiples of the fundamental cycle, and $H_3(\M/D_4^+)\cong\Z$. \qed

\begin{remark}
Let $N$ be the compact $3$-manifold with boundary obtained from $\M/D_4$ by removing small open conical neighbourhoods of the three singular points. Then $\partial N$ consists of three tori, and $\chi(N)=\chi(\M/D_4)-3=0$, in accordance with the fact that $\chi(N)=\frac12\chi(\partial N)$ for every compact $3$-manifold. The space $\M/D_4$ is obtained from $N$ by coning off each boundary torus.
\end{remark}

\begin{remark}[On the non-orientability of $\Mreg/D_4$]\label{rem:balance}
The non-orientability in Theorem~\ref{thC}(c) is a property of the quotient, not of $\M$ itself. In $\R{4}$ the regular part $\Mreg$ is orientable, with the global unit normal $DS/|DS|$, and $\M$ separates $\R{4}$ into the two open sets
$$
\Omega_+=\{S>0\},\qquad \Omega_-=\{S<0\}.
$$
By Proposition~\ref{pro:transl}, the translations by $D_4^+$ preserve $S$, and hence each of $\Omega_\pm$, while the translations by $D_4\setminus D_4^+$ change the sign of $S$ and interchange $\Omega_+$ and $\Omega_-$. Thus the two sides of $\M$ are congruent, and some of the periods of $\M$ exchange them. In the quotient by $D_4^+$ the two sides remain distinct, and $\M/D_4^+$ is a two-sided hypersurface which divides the torus $\R{4}/D_4^+$ into two congruent halves. In the quotient by the full lattice $D_4$ the two sides are identified, so that no unit normal can be chosen consistently and $\Mreg/D_4$ is one-sided, hence non-orientable.

The same phenomenon is familiar for triply periodic minimal surfaces in $\R{3}$. For instance, Schwarz's P-surface (approximated by the level set $\cos x+\cos y+\cos z=0$) divides space into two congruent labyrinths, which are interchanged by a translation preserving the surface, here by $(\pi,\pi,\pi)$. Its quotient by the lattice $2\pi\Z^3$ is an orientable surface of genus $3$, while its quotient by the full translation group, which also contains $(\pi,\pi,\pi)$, is one-sided. Triply periodic surfaces whose two sides are interchanged by a symmetry of the surface are called \emph{balance surfaces} in the crystallographic literature \cite{FischerKoch}; besides the P-surface, Schwarz's D-surface and Schoen's gyroid are of this kind. In this terminology, $\M$ is a balance hypersurface: the translations by $D_4\setminus D_4^+$ play the role of the translation $(\pi,\pi,\pi)$ for the P-surface. The same self-duality is visible elsewhere in the paper, for instance in the relation $\Theta(x+n)=(-1)^{n_1+n_3}\Theta(x)$ of Theorem~\ref{thm:theta} and in the identity $\Sigma_{c+\frac12}=\Sigma_c+(\frac12,\frac12,\frac12)$ for the sections of Proposition~\ref{pro:sections}, where the translation also interchanges the two sides.
\end{remark}

\section{Stability, Morse index and the level-set mechanism}\label{sec:index}

In this section we discuss two further aspects of $\M$. First, we show that $\M$ is very far from being stable: every neighbourhood of a singular point carries infinitely many independent destabilizing variations, so that the Morse index of the compact quotient $\M/D_4$ is infinite. Second, we explain the minimality of $\M$ by a factorization of the mean curvature operator applied to the hole-centred defining function of Section~\ref{sec:hole}.

\subsection{The second variation}\label{ssec:secondvariation}
Let $\Omega\subset\Mreg$ be an open set on which a unit normal field $\nu$ is defined; for instance, $\Omega=\Mreg$ with $\nu=DS/|DS|$. For $\varphi\in C_c^\infty(\Omega)$, the second variation of area in the direction of the normal variation $\varphi\nu$ is the quadratic form
\begin{equation}\label{Qform}
\mathcal{Q}_\Omega(\varphi)=\int_\Omega\bigl(|\nabla\varphi|^2-|A|^2\varphi^2\bigr)\,d\mathcal{H}^3,
\end{equation}
where $\nabla$ is the gradient along $\Mreg$ and $A$ is its second fundamental form \cite{SimonGMT}. The \emph{Morse index} of $\Omega$ is the supremum of the dimensions of the linear subspaces of $C_c^\infty(\Omega)$ on which $\mathcal{Q}_\Omega$ is negative definite. Since the variations are required to vanish near the singular set, this definition makes sense for $\Mreg$ itself, and also for the regular part of $\M/D_4$, where we use variations supported in small open sets on which a unit normal is defined (recall that $\Mreg/D_4$ is one-sided).

Two simple observations will be used. If $\varphi_1,\dots,\varphi_m$ have pairwise disjoint supports and $\mathcal{Q}(\varphi_k)<0$ for all $k$, then
$$
\mathcal{Q}\Bigl(\sum_k c_k\varphi_k\Bigr)=\sum_k c_k^2\,\mathcal{Q}(\varphi_k)<0\qquad\text{whenever}\ (c_1,\dots,c_m)\neq0,
$$
so that $\mathcal{Q}$ is negative definite on the span of $\varphi_1,\dots,\varphi_m$. Moreover, under homotheties $\mathcal{Q}$ changes only by a positive factor: if $\M_r=(\M-y)/r$ and $\tilde\varphi(z)=\varphi(y+rz)$, then $|\nabla\tilde\varphi|^2$ and $|A_{\M_r}|^2$ are multiplied by $r^2$, while $d\mathcal{H}^3$ is multiplied by $r^{-3}$, so that
\begin{equation}\label{Qscaling}
\mathcal{Q}_{\M}(\varphi)=r\,\mathcal{Q}_{\M_r}(\tilde\varphi).
\end{equation}

\subsection{The Clifford cone}
Let $C$ be the Clifford cone \eqref{Clifford}. Its regular part $C\setminus\{0\}$ is the cone over the Clifford torus $\Sigma=S^1(\frac1{\sqrt2})\times S^1(\frac1{\sqrt2})\subset S^3$, i.e. it is identified with $(0,\infty)\times\Sigma$ with the metric $dr^2+r^2g_\Sigma$ and the area element $r^2\,dr\,d\mathcal{H}^2_\Sigma$. The Clifford torus is a minimal surface in $S^3$ with principal curvatures $\pm1$, so that $|A_\Sigma|^2=2$ and
$$
|A_C|^2=\frac{|A_\Sigma|^2}{r^2}=\frac{2}{r^2}.
$$

\begin{lemma}\label{lem:conevariation}
Let $v$ be a smooth function on $\R{}$ with compact support, and put
$$
f(r)=r^{-1/2}\,v(\log r),\qquad r>0 .
$$
Then the function $\varphi(r,\omega)=f(r)$ on $C\setminus\{0\}$ satisfies
$$
\mathcal{Q}_C(\varphi)=|\Sigma|\int_{\R{}}\Bigl(v'(t)^2-\frac74\,v(t)^2\Bigr)\,dt,\qquad |\Sigma|=2\pi^2 .
$$
\end{lemma}

\begin{proof}
Since $\varphi$ does not depend on $\omega\in\Sigma$, we have $|\nabla\varphi|^2=f'(r)^2$, and therefore
$$
\mathcal{Q}_C(\varphi)=|\Sigma|\int_0^\infty\Bigl(f'(r)^2-\frac{2}{r^2}f(r)^2\Bigr)r^2\,dr .
$$
Substitute $r=e^t$. Then $f'(r)=r^{-3/2}\bigl(v'(t)-\frac12v(t)\bigr)$ and $dr=r\,dt$, so that
$$
\int_0^\infty f'^2\,r^2\,dr=\int_{\R{}}\Bigl(v'-\frac12v\Bigr)^2dt=\int_{\R{}}\Bigl(v'^2+\frac14v^2\Bigr)dt,
$$
because $\int vv'\,dt=\frac12\int(v^2)'\,dt=0$. Similarly, $\int_0^\infty 2f^2\,dr=2\int_{\R{}}v^2\,dt$. Subtracting, we obtain the formula.
\end{proof}

The number $\frac74=2-\frac14$ is the difference between the value $|A_\Sigma|^2=2$ and the constant $\frac14=\frac{(3-2)^2}{4}$ of the Hardy inequality in dimension three; its positivity is Simons' instability criterion for the Clifford cone \cite{Simons}. Lemma~\ref{lem:conevariation} makes the instability quantitative. Fix a smooth function $v_0$ with support in $(-3,0)$ such that
\begin{equation}\label{v0}
\int_{\R{}}v_0'^2\,dt<\frac74\int_{\R{}}v_0^2\,dt .
\end{equation}
For instance, one may take a smooth approximation of the function equal to $\sin(\pi t/3)$ on $[-3,0]$ and to $0$ elsewhere, for which the ratio of the two integrals is $\pi^2/9<\frac74$. Put
$$
f_0(r)=r^{-1/2}v_0(\log r),\qquad \delta_0=-\mathcal{Q}_C(f_0)>0 .
$$
The function $f_0$ is supported in the annulus $e^{-3}<r<1$.

\subsection{Infinite Morse index}

\begin{theorem}\label{thm:index}
Let $y\in\Sing$ and let $V$ be any neighbourhood of $y$ in $\R{4}$. Then there are functions $\varphi_1,\varphi_2,\ldots\in C_c^\infty(\Mreg\cap V)$ with pairwise disjoint supports such that $\mathcal{Q}(\varphi_k)<0$ for all $k$. In particular:
\begin{enumerate}
\item[(i)] the Morse index of $\Mreg\cap V$ is infinite for every neighbourhood $V$ of a singular point;
\item[(ii)] the Morse index of the compact minimal hypersurface $\M/D_4$ in the flat torus $\R{4}/D_4$, and that of $\M/D_4^+$, is infinite.
\end{enumerate}
\end{theorem}

\begin{proof}
Let $\Phi=y+\Psi$ be the analytic chart of Theorem~\ref{thA}(c), defined on a ball $B_\rho$, and let $\Phi_r(z)=\Psi(rz)/r$ be the rescaled charts of Step~1 in the proof of Theorem~\ref{thA}(d). Since $\Psi$ is real analytic with $\Psi(0)=0$ and $d\Psi(0)=\Id$, we have, uniformly for $z\in B_3$,
$$
\Phi_r(z)-z=O(r),\qquad d\Phi_r(z)-\Id=O(r),\qquad D^2\Phi_r(z)=r\,D^2\Psi(rz)=O(r).
$$
Thus $\Phi_r\to\Id$ in $C^2(B_3)$ as $r\to0$. By \eqref{rescaledchart}, $\Phi_r$ maps $C_y\cap B_3$ onto a neighbourhood of $0$ in $\M_r=(\M-y)/r$. Since $C_y$ is congruent to $C$, we may and shall replace $C$ by $C_y$ in Lemma~\ref{lem:conevariation} and in the definition of $f_0$.

Consider the compact piece $K=\{z\in C_y:\ e^{-3}\le|z|\le1\}$ of the regular part of the cone. On $\Phi_r(K)\subset\M_r$ we use $\Phi_r$ as a parametrization. The induced metric, the area element and the second fundamental form of $\M_r$, pulled back to $K$ by $\Phi_r$, depend on the first and second derivatives of $\Phi_r$ only. Since $\Phi_r\to\Id$ in $C^2$, they converge uniformly on $K$ to the corresponding quantities of $C_y$. Therefore, for the function $\tilde\varphi_r=f_0\circ|\cdot|\circ\Phi_r^{-1}$, which is supported in $\Phi_r(K)$,
$$
\mathcal{Q}_{\M_r}(\tilde\varphi_r)\longrightarrow\mathcal{Q}_{C_y}(f_0)=-\delta_0\qquad(r\to0).
$$
Choose $r_0>0$ such that $\mathcal{Q}_{\M_r}(\tilde\varphi_r)<-\delta_0/2$ for $0<r\le r_0$, and such that $y+r_0\Phi_{r_0}(B_3)=\Phi(r_0B_3)\subset V$.

Now put $r_k=e^{-3k}r_0$, $k=0,1,2,\dots$, and define functions on $\Mreg$ by
$$
\varphi_k(x)=\tilde\varphi_{r_k}\Bigl(\frac{x-y}{r_k}\Bigr).
$$
The support of $\varphi_k$ is $y+r_k\Phi_{r_k}(K)=\Phi(r_kK)$, which is contained in $\Phi(\{z\in C_y:\ e^{-3}r_k\le|z|\le r_k\})$. These sets are pairwise disjoint, because the annuli $e^{-3(k+1)}r_0\le|z|\le e^{-3k}r_0$ overlap at most along their boundary spheres, where $f_0$ vanishes, and $\Phi$ is injective. All supports lie in $\Phi(r_0B_3)\subset V$ and avoid $y$. By \eqref{Qscaling},
$$
\mathcal{Q}_{\M}(\varphi_k)=r_k\,\mathcal{Q}_{\M_{r_k}}(\tilde\varphi_{r_k})<-\frac{r_k\delta_0}{2}<0 .
$$
This proves the main statement, and (i) follows from the observation in \S\ref{ssec:secondvariation}. For (ii), choose $V$ so small that the projection $\R{4}\to\R{4}/D_4$ (or $\R{4}\to\R{4}/D_4^+$) is injective on $V$; then the functions $\varphi_k$ descend to the quotient, with disjoint supports in a set on which a unit normal is defined.
\end{proof}

\begin{remark}\label{rem:indexconsequences}
(i) Theorem~\ref{thm:index} sharpens the observation that $\M$ is not stable: the instability is not caused by a single large domain, but is concentrated, with infinite multiplicity, at every singular point. The mechanism is the same as for the Clifford cone itself, whose second variation is negative on infinitely many disjoint annuli shrinking to the vertex.

(ii) The theorem places $\M/D_4$ outside the range of the current variational theory of minimal hypersurfaces in low dimensions. For instance, by Sharp's compactness theorem \cite{Sharp}, a sequence of smooth closed embedded minimal hypersurfaces in a closed $4$-manifold with uniformly bounded area and Morse index subconverges, away from finitely many points, to a smooth closed embedded minimal hypersurface. Thus $\M/D_4$ is not such a limit, and it cannot be obtained by min-max methods, which produce hypersurfaces of bounded index. Likewise, the lower bounds of the Morse index in terms of the first Betti number, proved for smooth closed minimal hypersurfaces of flat tori in \cite{ACS}, have no counterpart here: $\M/D_4$ has $b_1=4$ (Theorem~\ref{thC}) and infinite index.

(iii) The translations of $\R{4}$ give the Jacobi fields $\scal{\nu}{e_i}$, $i=1,\dots,4$, on $\Mreg$, i.e. solutions of $\Delta\varphi+|A|^2\varphi=0$. A deformation theory of $\M$, for instance a statement that these are the only periodic Jacobi fields with suitable decay at the singular points, would require weighted function spaces adapted to the cones, as in \cite{CafHarSi,SmaleN}. Theorem~\ref{thm:index} shows that the choice of weights is delicate, because the Jacobi operator has infinitely many negative directions near each singular point.
\end{remark}

\subsection{The level-set mechanism}\label{ssec:mechanism}
We now return to the hole-centred description of Theorem~\ref{thm:symmetric}. Recall from Remark~\ref{rem:secondproof} the function
$$
\phi(t)=\arctan\gamma^2(t)-\frac\pi8,\qquad \phi'^2=\varpi^2\cos4\phi,\qquad \phi''=-2\varpi^2\sin4\phi,
$$
and the function $u(y)=\sum_{k=1}^4\phi(\xi_k(y))$, whose zero set is $\M-h'$. The following proposition shows that the minimality of $\M$ is caused by a factorization of $\Delta_1u$, in the same way as Lemma~\ref{lem:key} shows it for $S$.

\begin{proposition}\label{pro:mechanism}
Put $\phi_k=\phi(\xi_k)$. Then
\begin{equation}\label{factorization}
\Delta_1u=-\varpi^4\,\sin(2u)\,R,
\end{equation}
where
$$
R=\cos2(\phi_1+\phi_2-\phi_3-\phi_4)+\cos2(\phi_1+\phi_3-\phi_2-\phi_4)+\cos2(\phi_1+\phi_4-\phi_2-\phi_3).
$$
In particular, $\Delta_1u$ vanishes on every level set $\{u=c\}$ with $\sin2c=0$. Since $-\frac\pi2\le u\le\frac\pi2$, these are the level $u=0$, which is $\M-h'$, and the levels $u=\pm\frac\pi2$, which consist of isolated points (the holes).
\end{proposition}

\begin{proof}
The vectors $\frac12f_k$ are pairwise orthogonal and $|\frac12f_k|^2=\frac12$. Hence, by \eqref{xi},
$$
Du=\frac12\sum_k\phi'(\xi_k)f_k,\qquad |Du|^2=\frac12\sum_k\phi'(\xi_k)^2,\qquad \Delta u=\frac12\sum_k\phi''(\xi_k),
$$
and
$$
\scal{Du}{D|Du|^2}=\frac12\sum_k\phi'(\xi_k)^2\phi''(\xi_k).
$$
Substituting this into the definition \eqref{meancurv} of $\Delta_1$, we obtain
$$
\Delta_1u=\frac14\sum_{i\neq j}\phi''(\xi_i)\,\phi'(\xi_j)^2 .
$$
Put $a_k=4\phi_k$, so that $a_1+a_2+a_3+a_4=4u$. Using $\phi'^2=\varpi^2\cos a$ and $\phi''=-2\varpi^2\sin a$, we get
$$
\Delta_1u=-\frac{\varpi^4}{2}\sum_{i\ne j}\sin a_i\cos a_j=-\frac{\varpi^4}{2}\sum_{1\le i<j\le4}\sin(a_i+a_j).
$$
Now we group the six pairs $\{i,j\}$ into the three partitions of $\{1,2,3,4\}$ into two pairs. For a partition $\{i,j\}\sqcup\{k,l\}$,
\begin{align*}
\sin(a_i+a_j)+\sin(a_k+a_l)&=2\sin\frac{a_1+a_2+a_3+a_4}{2}\,\cos\frac{a_i+a_j-a_k-a_l}{2}\\
&=2\sin(2u)\cos2(\phi_i+\phi_j-\phi_k-\phi_l).
\end{align*}
Summing over the three partitions gives \eqref{factorization}. Finally, $u=\frac\pi2$ forces $\phi_k=\frac\pi8$, i.e. $\gamma^2(\xi_k)=1$, for all $k$, and $u=-\frac\pi2$ forces $\gamma^2(\xi_k)=0$ for all $k$; both conditions define discrete sets, which are holes of $\M$ (for example, $u(0)=\frac\pi2$ at $y=0$).
\end{proof}

\begin{remark}\label{rem:mechanism}
The proof of \eqref{factorization} uses the fact that the six pairs of indices in $\{1,2,3,4\}$ split into three complementary couples, i.e. that the complement of a pair is again a pair. This is a feature of four variables. For three variables the analogous sum $\sum_{i<j}\sin(a_i+a_j)$ does not vanish on the planes $a_1+a_2+a_3=\mathrm{const}$ (Remark~\ref{rem:notminimal}), and the same is true for five variables: on $a_1+\dots+a_5=0$, the point $a=(t,t,-2t,0,0)$ gives $\sum_{i<j}\sin(a_i+a_j)=2\sin t-\sin2t$. In this sense the minimality of $\M$, seen from a hole, is a four-dimensional phenomenon, and higher-dimensional analogues, if they exist, require a different mechanism.
\end{remark}

\section{Remarks and open questions}\label{sec:remarks}

\subsection{The lemniscatic catenoid}
The lemniscatic functions also appear in the simplest minimal hypersurface of $\R{4}$ which is not a cone, the catenoid. Indeed, the profile curve $r=r(t)$ of a minimal rotation hypersurface $|(x_1,x_2,x_3)|=r(x_4)$ satisfies $rr''=2(1+r'^2)$, with the first integral $r^4=\kappa^2(1+r'^2)$, where $\sqrt{\kappa}$ is the radius of the neck. For $\kappa=1$ the function $F=r^2$ satisfies $F'^2=4F^3-4F$, which is the equation of $\wp=1/\sinl^2$ (Remark~\ref{rem:wp}). Thus the four-dimensional catenoid with neck radius $1$ is
$$
\bigl\{x\in\R{4}:\ (x_1^2+x_2^2+x_3^2)\,\sinl^2x_4=1,\ \ 0<x_4<\varpi\bigr\},
$$
and its height, i.e. the width of the slab containing it, is exactly the lemniscate constant $\varpi$. Its translates by the multiples of $\varpi e_4$ form a family of pairwise disjoint congruent catenoids, which is the whole zero set of the periodic function $(x_1^2+x_2^2+x_3^2)\sinl^2x_4-1$.

\subsection{Relation to the separable ansatz}\label{ssec:separable}
On each open cube $\mathring P_n$, $n\in D_4$, the hypersurface $\M$ is given by the separable equation
\begin{equation}\label{separable}
\psi(x_1)+\psi(x_3)-\psi(x_2)-\psi(x_4)=0,\qquad \psi(t)=\log|\sigma(t)|,
\end{equation}
and, by \eqref{ODE1}, the function $\psi$ satisfies the autonomous equation
$$
\psi'^2=\varpi^2\bigl(e^{-2\psi}-e^{2\psi}\bigr).
$$
For comparison, the cone $x_1x_3=x_2x_4$ is given by \eqref{separable} with $\psi=\log|t|$, for which $\psi'^2=e^{-2\psi}$, while Scherk's surface $e^{x_3}\cos x_1=\cos x_2$ in $\R{3}$ corresponds to $\psi=\log|\cos t|$, for which $\psi'^2=e^{-2\psi}-1$. This explains in which sense $\M$ is a four-dimensional relative of Scherk's surface.

Thus $\M$ fits the separable ansatz $\sum_i f_i(x_i)=0$ for minimal hypersurfaces in $\R{N}$. Separable minimal surfaces in $\R{3}$ were classified by Weingarten \cite{Weing}; an analysis of the ansatz for $N\ge4$ can be found in Hoppe's lecture notes \cite[Section~II]{HoppeLect} (see also \cite{SakakiTanaka}). In particular, \cite[(II.26)--(II.27)]{HoppeLect} contains the equation $\wp(x)\wp(y)=\wp(z)\wp(v)$; since $\wp=1/\sinl^2$, it describes the union of $\M$ and $\M+e_1$ (up to a homothety and a permutation of coordinates), two congruent copies of $\M$ which intersect along the skeleton $\Mo$. We do not pursue the ansatz here. We only note that $\M$ is separable in a second way: seen from a hole, it is given in the frame coordinates by $\sum_k\phi(\xi_k)=0$, with the same function $\phi$ in all four variables (Remark~\ref{rem:secondproof}).

\subsection{Open questions}

\begin{question}
Does there exist a smooth embedded $D_4$-periodic minimal hypersurface in $\R{4}$ which is close to $\M$ (a desingularization of $\M$)? More generally, what can be said about the moduli space of periodic minimal hypersurfaces in $\R{4}$ near $\M$? Are the translational Jacobi fields $\scal{\nu}{e_i}$ the only $D_4$-periodic Jacobi fields of $\Mreg$ with suitable decay at the singular points (see Remark~\ref{rem:indexconsequences}(iii))?
\end{question}

\begin{question}
Is $\M$ characterized, among the stationary integral varifolds in $\R{4}$ with isolated singularities, by its symmetry group, or by its singular set $\Sing$ together with the types of the tangent cones?
\end{question}

\begin{question}
Theorem~\ref{thm:theta} represents $\M$ as the zero set of the theta series of the coset $\frac12(1+\mathrm{j})+\mathcal{O}$ of the Hurwitz order at the CM point $\tau=1+\mathrm{i}$. Is there a conceptual reason why the zero set of this theta series is a minimal hypersurface, for instance via the heat equation $\Delta_y\vartheta_\Lambda=4\pi\mathrm{i}\,\partial_\tau\vartheta_\Lambda$? For which lattices $L$, vectors $w$ and points $\tau$ is the zero set of $\vartheta_{w+L}(\cdot;\tau)$ minimal?
\end{question}

\begin{question}\label{q:quaternionic}
Proposition~\ref{pro:quaternionic} shows that, seen from a hole, all the structure of $\M$ (the lattice of holes, the three types of singular points, their tangent cones and the symmetry group) is expressed in terms of the Hurwitz order, the imaginary units and the binary octahedral group. Is there a construction of $\M$ which starts from the quaternions, for instance from the arithmetic of $\mathcal{O}$ or from the theta series of Remark~\ref{rem:quaternionic}, and explains its minimality?
\end{question}

\begin{question}
The construction of $\M$ combines the geometry of the root lattice $D_4$ (the $24$-cell honeycomb) with the Clifford cone, which is the zero set of the norm form of the split quaternions. Theta series such as those of Theorem~\ref{thm:theta} make sense in any dimension. Are there analogous periodic minimal hypersurfaces in $\R{8}$ and $\R{24}$, related to the lattice $E_8$ and the Leech lattice, respectively, whose singularities are modelled on other minimal cones, for instance on the Simons cone $x_1^2+\dots+x_4^2=x_5^2+\dots+x_8^2$ (the null cone of the norm form of the split octonions) or on the cubic minimal cones \cite{TkCliff,NTVbook,Tkachev-jordan}?
\end{question}

\subsection*{Acknowledgements}
I would like to thank Joaquín Pérez and Antonio Ros for giving me the opportunity to present these results in 2012 at the Differential Geometry Seminar at the University of Granada and in 2013 at the conference \emph{Variational Problems and Geometric PDEs}, and for several useful discussions concerning the topological aspects of the construction.
I am also grateful to Nikolai Nadirashvili for several useful discussions on singularities of nonlinear PDEs and related techniques from geometric measure theory during my visit to Marseille in 2015.
Finally, I thank Pavel Ryabov for helpful discussions concerning the homology computation in Section~\ref{sec:top}.

\bibliographystyle{plain}


\begin{thebibliography}{10}

\bibitem{AkhiezerEll}
N.~I. Akhiezer.
\newblock {\em Elements of the theory of elliptic functions}, volume~79 of {\em
  Translations of Mathematical Monographs}.
\newblock American Mathematical Society, Providence, RI, 1990.
\newblock Translated from the second Russian edition by H. H. McFaden.

\bibitem{Allard}
W.~K. Allard.
\newblock On the first variation of a varifold.
\newblock {\em Ann. of Math. (2)}, 95:417--491, 1972.

\bibitem{ACS}
L.~Ambrozio, A.~Carlotto, and B.~Sharp.
\newblock A note on the index of closed minimal hypersurfaces of flat tori.
\newblock {\em Proc. Amer. Math. Soc.}, 146(1):335--344, 2018.

\bibitem{BGG}
E.~Bombieri, E.~De~Giorgi, and E.~Giusti.
\newblock Minimal cones and the {B}ernstein problem.
\newblock {\em Invent. Math.}, 7:243--268, 1969.

\bibitem{CafHarSi}
L.~Caffarelli, R.~Hardt, and L.~Simon.
\newblock Minimal surfaces with isolated singularities.
\newblock {\em Manuscripta Math.}, 48(1-3):1--18, 1984.

\bibitem{ChoeHoppe}
J.~Choe and J.~Hoppe.
\newblock Higher dimensional minimal submanifolds generalizing the catenoid and
  helicoid.
\newblock {\em Tohoku Math. J. (2)}, 65(1):43--55, 2013.

\bibitem{ChoeHoppe18}
J.~Choe and J.~Hoppe.
\newblock Higher dimensional {S}chwarz's surfaces and {S}cherk's surfaces.
\newblock {\em Calc. Var. Partial Differential Equations}, 57(4):Paper No. 107,
  17, 2018.

\bibitem{Conway}
J.~H. Conway and N.~J.~A. Sloane.
\newblock {\em Sphere packings, lattices and groups}, volume 290 of {\em
  Grundlehren der Mathematischen Wissenschaften [Fundamental Principles of
  Mathematical Sciences]}.
\newblock Springer-Verlag, New York, third edition, 1999.
\newblock With additional contributions by E. Bannai, R. E. Borcherds, J.
  Leech, S. P. Norton, A. M. Odlyzko, R. A. Parker, L. Queen and B. B. Venkov.

\bibitem{ConwaySmith}
J.~H. Conway and D.~A. Smith.
\newblock {\em On quaternions and octonions: their geometry, arithmetic, and
  symmetry}.
\newblock A K Peters, Ltd., Natick, MA, 2003.

\bibitem{Coxeter}
H.~S.~M. Coxeter.
\newblock {\em Regular polytopes}.
\newblock Dover Publications, New York, third edition, 1973.

\bibitem{Dimler}
B.~Dimler.
\newblock Minimal submanifolds with multiple isolated singularities, 2024.
\newblock arXiv:2409.20327.

\bibitem{FischerKoch}
W.~Fischer and E.~Koch.
\newblock On 3-periodic minimal surfaces.
\newblock {\em Z. Kristallogr.}, 179:31--52, 1987.

\bibitem{HardtSimon}
R.~Hardt and L.~Simon.
\newblock Area minimizing hypersurfaces with isolated singularities.
\newblock {\em J. Reine Angew. Math.}, 362:102--129, 1985.

\bibitem{Hatcher}
A.~Hatcher.
\newblock {\em Algebraic topology}.
\newblock Cambridge University Press, Cambridge, 2002.

\bibitem{HoppeLect}
J.~Hoppe.
\newblock Lectures on minimal surfaces, 2019.
\newblock arXiv:1903.12062.

\bibitem{HoppeTkachev}
J.~Hoppe and V.~G. Tkachev.
\newblock New construction techniques for minimal surfaces.
\newblock {\em Complex Var. Elliptic Equ.}, 64(9):1546--1563, 2019.

\bibitem{Hsiang67}
W.-Y. Hsiang.
\newblock Remarks on closed minimal submanifolds in the standard {R}iemannian
  {$m$}-sphere.
\newblock {\em J. Differential Geometry}, 1:257--267, 1967.

\bibitem{KaabachiP}
S.~Kaabachi and F.~Pacard.
\newblock Riemann minimal surfaces in higher dimensions.
\newblock {\em J. Inst. Math. Jussieu}, 6(4):613--637, 2007.

\bibitem{LangerSinger}
J.~C. Langer and D.~A. Singer.
\newblock Reflections on the lemniscate of {B}ernoulli: the forty-eight faces
  of a mathematical gem.
\newblock {\em Milan J. Math.}, 78(2):643--682, 2010.

\bibitem{Lawden}
D.~F. Lawden.
\newblock {\em Elliptic functions and applications}, volume~80 of {\em Applied
  Mathematical Sciences}.
\newblock Springer-Verlag, New York, 1989.

\bibitem{Lawson}
H.~B. Lawson, Jr.
\newblock Complete minimal surfaces in {$S^{3}$}.
\newblock {\em Ann. of Math. (2)}, 92:335--374, 1970.

\bibitem{LeeLee}
E.~Lee and H.~Lee.
\newblock Generalizations of the {C}hoe--{H}oppe helicoid and {C}lifford cones
  in {E}uclidean space.
\newblock {\em J. Geom. Anal.}, 27(1):817--841, 2017.

\bibitem{Markush}
A.~I. Markushevich.
\newblock {\em The remarkable sine functions}.
\newblock American Elsevier Publishing Co., Inc., New York, 1966.
\newblock Translated from the Russian.

\bibitem{Mazet}
L.~Mazet.
\newblock Minimal hypersurfaces asymptotic to {S}imons cones.
\newblock {\em J. Inst. Math. Jussieu}, 16(1):39--58, 2017.

\bibitem{Meeks90}
W.~H. Meeks, III.
\newblock The theory of triply periodic minimal surfaces.
\newblock {\em Indiana Univ. Math. J.}, 39(3):877--936, 1990.

\bibitem{MilnorMorse}
J.~Milnor.
\newblock {\em Morse theory}.
\newblock Based on lecture notes by M. Spivak and R. Wells. Annals of
  Mathematics Studies, No. 51. Princeton University Press, Princeton, N.J.,
  1963.

\bibitem{NTV12}
N.~Nadirashvili, V.~G. Tkachev, and S.~Vl\u{a}du\c{t}.
\newblock A non-classical solution to a {H}essian equation from {C}artan
  isoparametric cubic.
\newblock {\em Adv. Math.}, 231(3-4):1589--1597, 2012.

\bibitem{NTVbook}
N.~Nadirashvili, V.~G. Tkachev, and S.~Vl\u{a}du\c{t}.
\newblock {\em Nonlinear elliptic equations and nonassociative algebras},
  volume 200 of {\em Mathematical Surveys and Monographs}.
\newblock American Mathematical Society, Providence, RI, 2014.

\bibitem{Pacard02}
F.~Pacard.
\newblock Higher-dimensional {S}cherk's hypersurfaces.
\newblock {\em J. Math. Pures Appl. (9)}, 81(3):241--258, 2002.

\bibitem{Perdomo3}
O.~M. Perdomo.
\newblock Minimal hypersurfaces in {$\Bbb R^n$} as regular values of a
  function.
\newblock {\em Rev. Integr. Temas Mat.}, 22(1-2):1--6, 2004.

\bibitem{SakakiTanaka}
M.~Sakaki and R.~Tanaka.
\newblock Some minimal hypersurfaces in the $(n+1)$-space with $2m$-norm, 2025.
\newblock arXiv:2508.12528.

\bibitem{Schappacher}
N.~Schappacher.
\newblock Some milestones of lemniscatomy.
\newblock In {\em Algebraic geometry ({A}nkara, 1995)}, volume 193 of {\em
  Lecture Notes in Pure and Appl. Math.}, pages 257--290. Dekker, New York,
  1997.

\bibitem{SergTk1999}
V.~V. Sergienko and V.~G. Tkachev.
\newblock Doubly periodic maximal surfaces with singularities.
\newblock {\em Siberian Adv. Math.}, 12(1):77--91, 2002.

\bibitem{Sharp}
B.~Sharp.
\newblock Compactness of minimal hypersurfaces with bounded index.
\newblock {\em J. Differential Geom.}, 106(2):317--339, 2017.

\bibitem{SimonGMT}
L.~Simon.
\newblock {\em Lectures on geometric measure theory}, volume~3 of {\em
  Proceedings of the Centre for Mathematical Analysis, Australian National
  University}.
\newblock Australian National University, Canberra, 1983.

\bibitem{SimonSing}
L.~Simon.
\newblock Stable minimal hypersurfaces in {$\mathbb{R}^{N+1+\ell}$} with
  singular set an arbitrary closed {$K\subset\{0\}\times\mathbb{R}^{\ell}$}.
\newblock {\em Ann. of Math. (2)}, 197(3):1205--1234, 2023.

\bibitem{Simons}
J.~Simons.
\newblock Minimal varieties in {R}iemannian manifolds.
\newblock {\em Ann. of Math. (2)}, 88:62--105, 1968.

\bibitem{SmaleN}
N.~Smale.
\newblock Minimal hypersurfaces with many isolated singularities.
\newblock {\em Ann. of Math. (2)}, 130(3):603--642, 1989.

\bibitem{TkCliff}
V.~G. Tkachev.
\newblock Minimal cubic cones via {C}lifford algebras.
\newblock {\em Complex Anal. Oper. Theory}, 4(3):685--700, 2010.

\bibitem{TkGranada}
V.~G. Tkachev.
\newblock Explicit examples of minimal hypersurfaces: a four-fold periodic
  minimal hypersurface in {$\mathbb{R}^4$} and beyond, 2013.
\newblock Talk at the conference \emph{Variational Problems and Geometric
  PDE's}, Granada, June 2013.

\bibitem{Tkachev-jordan}
V.~G. Tkachev.
\newblock A {J}ordan algebra approach to the cubic eiconal equation.
\newblock {\em J. Algebra}, 419:34--51, 2014.

\bibitem{Weing}
J.~{Weingarten}.
\newblock {Ueber die durch eine Gleichung von der Form ${\mathfrak X+Y+Z}=0$
  darstellbaren Minimalfl\"achen}.
\newblock {\em {G\"ott. Nachr.}}, 1887:272--275, 1887.

\bibitem{WhittakerWatson}
E.~T. Whittaker and G.~N. Watson.
\newblock {\em A course of modern analysis}.
\newblock Cambridge University Press, fourth edition, 1927.

\end{thebibliography}

\end{document}